\documentclass[11pt]{article}

\usepackage[margin=1.0truein]{geometry}
\usepackage{setspace}
\usepackage{lscape}
\usepackage{pdflscape}

\usepackage{xcolor}

\usepackage{amsmath,amssymb,amsthm}

\usepackage{wrapfig} 
\usepackage{graphicx}

\usepackage{mathrsfs}
\usepackage{url}
\usepackage{latexsym}
\usepackage{bm}

\usepackage{mathtools}

\usepackage{booktabs,multirow}
\usepackage{array}
\newcolumntype{P}[1]{>{\centering\arraybackslash}p{#1}}
\usepackage{diagbox}
\usepackage{threeparttable}
\usepackage{siunitx} 

\usepackage{tabularray}

\usepackage{algorithm}
\usepackage[noend]{algpseudocode}

\algnewcommand{\Break}{\textbf{break}}
\algnewcommand{\algorithmicgoto}{\textbf{go to} Line}
\algnewcommand{\Goto}[1]{\algorithmicgoto~\ref{#1}}
\algblockdefx{MRepeat}{EndRepeat}{\textbf{repeat}}{}
\algnotext{EndRepeat}

\usepackage{natbib}

\usepackage{thmtools,thm-restate}

\usepackage{hyperref}
\usepackage{cleveref}
\expandafter\def\csname ver@etex.sty\endcsname{3000/12/31}
 

\usepackage{breakurl}

\usepackage{multicol}
\usepackage{wasysym}
\usepackage{adjustbox}
\usepackage{xparse}
\usepackage{pbox}
\usepackage{authblk}
\usepackage{xspace}

\usepackage[subrefformat=parens]{subcaption}
\usepackage{enumitem} 
\hypersetup{
  colorlinks,
  linkcolor={green!35!black},
  citecolor={green!35!black},
  urlcolor={blue!50!black}
}

\usepackage{times}
\usepackage{soul}

\usepackage{secdot}
\usepackage{cases}
\usepackage{subcaption}
\usepackage{caption}

\declaretheoremstyle[
  shaded={bgcolor=gray!15},
]{thmsty}

\declaretheorem[
  name=Theorem,
  refname={Theorem,Theorems},
  style=thmsty,
]{theorem}

\declaretheorem[
  name=Lemma,
  refname={Lemma,Lemmas},
  style=thmsty,
  numberlike=theorem,
]{lemma}

\declaretheorem[
  name=Assumption,
  refname={Assumption,Assumptions},
  style=thmsty,
]{assumption}
\declaretheorem[
  name=Remark,
  refname={Remark,Remarks},
]{remark}

\declaretheorem[
  name=Lemma,
  refname={Lemma,Lemmas},
  style=thmsty,
]{techlemma}
\usepackage{chngcntr}
\counterwithin{techlemma}{section}

\crefname{algorithm}{Algorithm}{Algorithms}
\crefname{line}{Line}{Lines}
\crefname{section}{Section}{Sections}
\crefname{appendix}{Appendix}{Appendices}
\crefname{table}{Table}{Tables}
\crefname{figure}{Figure}{Figures}
\crefname{equation}{}{}
\Crefname{equation}{Eq.}{Eqs.}

\setlist[itemize]{
  topsep=0.4\baselineskip,
  itemsep=0\baselineskip,
  leftmargin=1.5em,
}

\setlist[enumerate]{
  font=\upshape,
  label=(\alph*),
  ref=(\alph*),
  topsep=0.4\baselineskip,
  itemsep=0\baselineskip,
  leftmargin=2em,
}

\newlist{enuminasm}{enumerate}{1} 
\setlist[enuminasm]{
  font=\upshape,
  label=(\alph*),
  ref=\theassumption(\alph*),
  topsep=0.4\baselineskip,
  itemsep=0\baselineskip,
  leftmargin=2em,
}
\crefalias{enuminasmi}{assumption}

\newlist{enuminthm}{enumerate}{1}
\setlist[enuminthm]{
  font=\upshape,
  label=(\alph*),
  ref=\thetheorem(\alph*),
  topsep=0.4\baselineskip,
  itemsep=0\baselineskip,
  leftmargin=2em,
}
\crefalias{enuminthmi}{theorem}

\newlist{enuminlem}{enumerate}{1}
\setlist[enuminlem]{
  font=\upshape,
  label=(\alph*),
  ref=\thelemma(\alph*),
  topsep=0.4\baselineskip,
  itemsep=0\baselineskip,
  leftmargin=2em,
}
\crefalias{enuminlemi}{lemma}

\newlist{enuminprop}{enumerate}{1}
\setlist[enuminprop]{
  font=\upshape,
  label=(\alph*),
  ref=\theproposition(\alph*),
  topsep=0.4\baselineskip,
  itemsep=0\baselineskip,
  leftmargin=2em,
}
\crefalias{enuminpropi}{proposition}

\newlist{enumincond}{enumerate}{1}
\setlist[enumincond]{
  font=\upshape,
  label=(\alph*),
  ref=\thecondition(\alph*),
  topsep=0.4\baselineskip,
  itemsep=0\baselineskip,
  leftmargin=2em,
}
\crefalias{enumincondi}{condition}

\DeclareMathOperator{\EE}{\mathbb{E}}

\DeclarePairedDelimiterX\inner[2]{\langle}{\rangle}{{#1},{#2}}

\DeclarePairedDelimiter\bra{[}{]}
\DeclarePairedDelimiterX\Set[2]{\{}{\}}{\mspace{2mu}{#1}\;\delimsize|\;{#2}\mspace{2mu}}
\DeclarePairedDelimiterX\Prn[2]{(}{)}{\mspace{2mu}{#1}\;\delimsize|\;{#2}\mspace{2mu}}
\DeclarePairedDelimiterX\Bra[2]{[}{]}{\mspace{2mu}{#1}\;\delimsize|\;{#2}\mspace{2mu}}

\newcommand{\N}{\mathbb N}

\newcommand{\R}{\mathbb R}

\newcommand{\bs}{\bm{s}}

\newcommand{\bx}{\bm{x}}

\newcommand{\ba}{\bm{a}}

\newcommand{\bv}{\bm{v}}

\newcommand{\bu}{\bm{u}}

\newcommand{\bA}{\bm{A}}
\newcommand{\bme}{\bm{e}}
\newcommand{\by}{\bm{y}}
\newcommand{\bmxi}{\bm{\xi}}

\newcommand{\bg}{\bm{g}}

\newcommand{\bnu}{\bm{\nu}}

\newcommand{\bI}{\bm{I}}
\newcommand{\mR}{\mathbb{R}}

\newcommand{\bone}{\bm{1}}
\newcommand{\bdelta}{\bm{\delta}}

\newcommand{\E}{\mathbb{E}}
\newcommand{\mpr}{\mathrm{Pr}}

\newcommand{\maN}{\mathcal{N}}
\newcommand{\maS}{\mathcal{S}}
\newcommand{\maB}{\mathcal{B}}
\newcommand{\md}{\mathrm{d}}

\renewcommand{\epsilon}{\varepsilon}

\NewDocumentCommand{\exsub}{s m O{} m}{%
  \IfBooleanT{#1}{\EE_{#2}\nolimits\bra*{#4}}%
  \IfBooleanF{#1}{\EE_{#2}\nolimits\bra[#3]{#4}}%
}
\NewDocumentCommand{\ex}{s O{} m}{%
  \IfBooleanT{#1}{\EE\nolimits\bra*{#3}}%
  \IfBooleanF{#1}{\EE\nolimits\bra[#2]{#3}}%
}
\NewDocumentCommand{\cex}{s O{} m m}{%
  \IfBooleanT{#1}{\EE\nolimits\Bra*{#3}{#4}}%
  \IfBooleanF{#1}{\EE\nolimits\Bra[#2]{#3}{#4}}%
}

\usepackage{CJKutf8}

\newcommand{\email}[1]{\href{mailto:#1}{\nolinkurl{#1}}}

\usepackage{amsmath}
\usepackage{comment} 
\usepackage{cancel} 
\usepackage{longtable}

\usepackage{CJKutf8}

\newcommand{\red}[1]{{\color{red}{#1}}}

\allowdisplaybreaks

\usepackage{datetime}
\date{\vspace{-2.5\baselineskip}}

\author[1]{Yuya Hikima\footnote{Corresponding author. E-mail: \url{yuya-hikima@g.ecc.u-tokyo.ac.jp}. ORCID: 0000-0003-3237-1618}}
\author[1,2]{Akiko Takeda}

\affil[1]{Graduate School of Information Science and Technology, University of Tokyo, Tokyo, Japan}
\affil[2]{Center for Advanced Intelligence Project, RIKEN, Tokyo, Japan}

\title{Zeroth-Order Gradient Estimators under Decision-Dependent Noise}

\begin{document}
\maketitle

\begin{abstract}
We study zeroth-order optimization of a generally nonconvex optimization problem that is accessible only through a noisy black-box oracle, 
where the observation noise has zero mean and its distribution depends on 
the queried decision vector. 
Since the gradient of the objective function is inaccessible, various zeroth-order methods have been developed to solve the problem.
However, it remains unclear how the choice of search directions affects the sample complexity of zeroth-order methods and how their algorithmic parameters should be selected.
In this paper, we conduct a unified sample complexity analysis of zeroth-order methods across gradient estimators with different search directions.
As a result, we show that gradient estimators that average over multiple directions, either uniformly from the unit sphere or from a Gaussian distribution, achieve the lowest sample complexity.
The resulting sample-complexity bounds improve upon those of existing zeroth-order methods for nonconvex objective functions that are not bounded above.
Moreover, through simulation experiments on multiple products pricing and strategic classification applications, we demonstrate the practical performance of zeroth-order methods with various gradient estimators.
\end{abstract}

\section{Introduction}
In this work, we study a stochastic optimization problem with an unknown decision-dependent distribution. 
Specifically, we consider
\begin{align}
\min_{\bx \in \R^d} \quad F(\bx):=\E_{\bmxi \sim D(\bx)}[f(\bx,\bmxi)], \label{eq:opt_problem}  
\end{align}
where $F:\R^d\to\R$ is generally nonconvex and is bounded
below by $F^*:= \inf_{\bx \in \R^d} F(\bx)$.
Here, the distribution $D(\bx)$ is unknown; only noisy function-value observations are available at queried decision vectors.
This problem is the generalized problem for the classical stochastic problem:
\begin{align}
\min_{\bx \in \R^d} \quad \E_{\bmxi \sim D}[f(\bx,\bmxi)], \label{eq:opt_problem_no_dependent}  
\end{align}
with a decision-independent distribution $D$.
Furthermore, Problem \eqref{eq:opt_problem} can be viewed as an optimization problem, $\min F(\bx)$, where $F$ is accessible only through a noisy black-box oracle:
\[
\mathcal{O}(\bx)=F(\bx)+\delta(\bx).
\]
Here, $\delta(\bx)$ is a zero-mean random variable whose distribution depends on the decision vector $\bx$.
This interpretation follows from the fact that the distribution $D(\bx)$ in \eqref{eq:opt_problem} is unknown, so that $F(\bx)$ cannot be evaluated exactly and only noisy function-value observations based on samples from $D(\bx)$ are available.
The formulation \eqref{eq:opt_problem} is also referred to as the \emph{performative prediction} problem \citep{perdomo2020performative}, emphasizing that the data-generating process adapts to the deployed decision.

This problem has attracted increasing attention in recent years due to its importance in applications \citep{perdomo2020performative, mendler2020stochastic, chen2023performative}.
In pricing applications, a seller selects a price vector $\bx$ to maximize revenue while demand follows a \emph{price-dependent} distribution $D(\bx)$ \citep{ray2022decision, hikima2025zeroth}. 
In strategic classification applications, a decision-maker trains model parameters $\bx$  to classify agents into groups (e.g., loan approval and denial) while affected agents may adapt their features in response to the deployed model, inducing a \emph{parameter-dependent} distribution $D(\bx)$ \citep{levanon2021strategic, liu2023time, cyffers2024optimal}. 
A key consideration in solving problem~\eqref{eq:opt_problem} is the \emph{sample complexity}, i.e., the number of noisy function-value
observations required to find a desirable point (e.g., a stationary or optimal point).
Increasing the number of samples is undesirable in practice, because each sample $\bmxi \sim D(\bx)$ can be obtained only after deploying a decision $\bx$ in the real world.

Since the gradient of the objective function is inaccessible as a result of the distribution $D(\bx)$ being unknown, 
various zeroth-order (or gradient-free) methods have been proposed \citep{ray2022decision,chen2023performative,liu2023time,hikima2025guided,frankel2025finite,hikima2025zeroth}.
These methods update the decision vector using gradient estimators constructed from function evaluation(s).
Typical gradient estimators \citep{ray2022decision,liu2023time,chen2023performative,hikima2025zeroth} for problem \eqref{eq:opt_problem} take the following forms:
\begin{align}
\begin{split}
&\textrm{one-point gradient estimator:} \quad \bg:= \frac{\kappa}{\mu} \left(\tilde{F}(\bx+\mu \bv)\right) \bv, \\
&\textrm{two-point gradient estimator:} \quad \bg:= \frac{\kappa}{2\mu} \left(\tilde{F}(\bx+\mu \bv)-\tilde{F}(\bx - \mu \bv)\right) \bv, \label{eq:standard_form_intro}
\end{split}
\end{align}
where $\mu \in \R_{>0}$ is a smoothing parameter, $\kappa \in \R_{>0}$ is a scaling factor, and $\bv \in \R^d$ is a method-dependent search direction.
Here, $\tilde{F}$ is the approximation of $F$ defined by:
$\tilde{F}(\bx):= \frac{1}{m}\sum_{j=1}^m f(\bx,\bmxi^{(j)}),$
where $\bmxi^{(j)} \sim D(\bx)$.\footnote{Here, $f(\bx,\bmxi^{(j)})$ denotes the noisy function-value observation
returned at $\bx$; the realization $\bmxi^{(j)}$ itself need not be
observed.}

While a growing literature has proposed various zeroth‑order methods with different gradient estimators, two key questions remain open: (i) which search direction (e.g., coordinate-wise directions; directions sampled uniformly from the unit sphere) is preferable, and (ii)  whether an estimator that averages over multiple directions such as
\begin{align}
\bg:= \frac{\kappa}{2\mu N} \sum_{i=1}^N \left(\tilde{F}(\bx+\mu \bv^i)-\tilde{F}(\bx - \mu \bv^i)\right) \bv^i \quad (N\ge 2) \label{eq:grad_est_N}
\end{align}
is superior to \eqref{eq:standard_form_intro} and, if so, how one should set the number $N$.
This is because, although individual studies analyze the convergence of their zeroth-order methods with different gradient estimators, they assume distinct problem settings, making direct comparisons difficult.
For example, \citet{ray2022decision,liu2023time} show the convergence rates for their methods which use a direction sampled uniformly from the unit sphere, considering \emph{bounded} objective function and \emph{time-varying} decision-dependent distributions.
On the other hand, \citet{hikima2025guided,hikima2025zeroth} show the convergence rates for their zeroth-order methods which use a direction following a Gaussian distribution, considering objective functions that are \emph{not bounded above} and \emph{time-invariant} decision-dependent distributions.
Moreover, all of the above analyze the estimator with 
a single direction per iteration; whether using multiple directions ($N\ge 2$) improves sample complexity remains unclear.

In this paper, we conduct a unified sample complexity analysis across different gradient estimators under a problem setting where problem \eqref{eq:opt_problem} is nonconvex, the objective function is not necessarily bounded above, and the distribution is time-invariant, together with appropriate selections of the algorithmic parameters.
Specifically, we consider gradient estimators that use coordinate-wise directions, directions drawn uniformly from the unit sphere, or directions following a Gaussian distribution.
Our proofs are built on techniques from  \citep{berahas2022theoretical}, which analyzes the errors between the true gradient and various gradient estimators  in a setting different from ours in terms of assumptions on noise (see Section \ref{subsec:zo_related}); under our assumptions, however, the resulting conclusions differ from theirs.

Consequently, we obtain three key findings. 
\begin{enumerate}
    \item For gradient estimators based on random directions, using more than one direction per iteration ($N>1$) effectively suppresses the error between the gradient estimator and the true gradient (Remark~\ref{rem:N}).
\item 
Averaging over multiple directions, sampled either uniformly from the unit sphere or from a Gaussian distribution, achieves lower sample complexity than that of coordinate-wise directions. This yields the lowest complexity among existing zeroth-order methods in problem settings that allow nonconvex objective functions that are not bounded above (Table \ref{table:SpeedOfLight}).
\item  After appropriately scaling the smoothing parameters ($\mu$ in \eqref{eq:grad_est_N}), gradient estimators using directions uniformly distributed on the unit sphere and those using Gaussian directions yield equivalent estimation error relative to the true gradient.
Specifically, as detailed in Remark~\ref{rem:ga_sp_order_same}, choosing the smoothing parameter for Gaussian directions to be $1/\sqrt{d}$ times that for unit-sphere directions ensures that the corresponding theoretical error bounds match in order.
\end{enumerate}

We conducted simulation experiments on multiple products pricing and strategic classification applications. 
The results show that, consistent with our theoretical analysis, zeroth-order methods employing gradient estimators with directions following a Gaussian distribution or directions uniformly sampled from the unit sphere achieve superior performance under our assumptions.

\subsection{Notation}
Bold lowercase symbols (e.g., $\bx, \by$) denote vectors, and $\|\bx\|$ denotes the Euclidean norm of a vector $\bx$.
For a matrix $A$, $\|A\|_{\mathrm{op}}$ denotes the operator norm, i.e., $\|A\|_{\mathrm{op}}
:=
\sup_{\|\bv\|=1}\|A\bv\|$.
The inner product of the vectors $\bx, \by$ is denoted by $\bx^\top \by$. 
Let $\mR_{>0}$ ($\mR_{\ge 0}$) be the set of positive (non-negative) real numbers and $\N$ be the set of natural numbers.
The gradient for a real-valued function $f(\bx)$ is denoted by $\nabla f(\bx)$. 
We let $\bI$ be an identity matrix and let $\maN(0, \bI)$ be the standard Gaussian distribution. 
Let $[N]$ be the set $\{1,2,\dots,N\}$.

\section{Related work}
\subsection{Zeroth-order optimization methods} \label{subsec:zo_related}
Zeroth-order methods have been developed for solving optimization problems where gradient information is either unavailable or computationally expensive to calculate \citep{flaxman2004online,ghadimi2013stochastic, nesterov2017random,berahas2022theoretical}.
Recently, various studies have proposed zeroth-order methods for solving problems with decision-dependent distributions (see Table~\ref{table:SpeedOfLight}).\footnote{
In Table \ref{table:SpeedOfLight}, \cite{izzo2021learn,izzo2022learn} and \cite{hikima2025guided} use gradient information of $f$ (not $F$) inside the expectation and thus, in a strict sense, are not zeroth-order methods.
However, because their gradient estimators are computed from function evaluations via finite differences or Gaussian perturbations, we include them in the table.}
These studies consider different problem settings and employ different gradient estimators.
For example, \cite{ray2022decision,liu2023time} adopt gradient estimators using directions sampled uniformly from the unit sphere, whereas \cite{hikima2025guided} employs directions drawn from a Gaussian distribution.
While a growing literature has employed diverse zeroth-order gradient estimators, it remains unclear which search directions are more appropriate.
In this work, for nonconvex problem \eqref{eq:opt_problem}, we analyze the sample complexities of zeroth-order methods that employ gradient estimators with different search directions.

Our problem setting can also be viewed as minimizing a deterministic function $F$ under the condition that we only have access to a noisy oracle $F(\bx)+\delta(\bx)$ with noise $\delta(\bx)$.
From this perspective, our setting is similar to that of  \citep{la2025gradient,bhavsar2022nonasymptotic}.
However, these existing works do not explicitly organize the performance of the search direction (e.g., coordinate-wise versus unit-sphere/Gaussian directions) in the sample-complexity analysis, nor do they analyze the effect of averaging over multiple directions per iteration ($N>1$).

Our theoretical analysis adapts the proof strategy of \citep{berahas2022theoretical}, which analyzes the errors between the true gradient and various gradient estimators in a different setting. The key distinction is that they assume bounded and possibly biased noise, whereas we assume generally unbounded and unbiased noise. This difference leads to distinct theoretical conclusions: \cite{berahas2022theoretical} show that gradient estimators with orthogonal directions attain higher accuracy than those with random directions, while our analysis demonstrates that estimators with random directions yield stronger convergence guarantees than those with orthogonal directions (see Section \ref{subsec:comparison}).

\begin{table}[t]
\caption{Overview of existing zeroth-order methods for problem \eqref{eq:opt_problem}. 
We state the order of the sample complexity with respect to the target accuracy $\epsilon$, the dimension $d$, and $\ell:=\sup_{\bx,\bmxi}|f(\bx,\bmxi)|$.
Other notes: \red{(a)} These methods reduce nonconvex optimization problems to convex ones under additional structural assumptions; 
\red{(b)} The result holds with probability $1-p$;
\red{(c)} This result indicates regret relative to the sample size $N_s$ (when the sample size $N_{KL}$ in their paper is a constant);
\red{(d)} Rather than directly approximating the gradient, these works approximate a component (also referred to as the ``performative part'') of the gradient;
\red{(e)} The result comes from \cite[Theorem 5]{izzo2021learn};
\red{(f)} The result comes from \cite[Theorem 1 and Lemma 17]{izzo2022learn};
\red{(g)} The result comes from \cite[Remark 2]{bhavsar2022nonasymptotic};
\red{(h)} It provides a unified analysis on search directions sharing common properties with the multivariate Gaussian distribution, the multivariate Bernoulli distribution, and the uniform distribution on the unit sphere;
\red{(i)} The former result holds when $\nabla F(x)$ is Lipschitz continuous, while the latter result holds when both $\nabla F(x)$ and $\nabla^2 F(x)$ are Lipschitz continuous; \red{(j)} These are auxiliary results from Section~\ref{sec:rademcher}.
\label{table:SpeedOfLight}}
 \scalebox{0.64}{
  \begin{tabular}{cllllll}
   \hline
 & Name  & Nonconvex & Direction & Metric& Sample Complexity \\
    \hline
\cite{ray2022decision}&Epoch-Based ZO  &\text{\texttimes}\ \red{(a)} &Unit sphere&$\E[\|\hat{\bx}-\bx^*\|^2] \le \epsilon^2$ &$O(\ell^2 d^2 \epsilon^{-4})$\\
\cite{chen2023performative}&Bandit Algorithm&\text{\texttimes}\ \red{(a)}&Unit sphere&$R:=\sum_{s=1}^{N_s} F(\bx_s)-N_s F(\bx^*)$ & $R=\tilde{O}(d N_s^{\frac{5}{6}}(\log{p^{-1}})^{\frac{1}{2}})$\ \red{(b), (c)} \\
\cite{frankel2025finite}&Derivative Free Method &\text{\texttimes}\ \red{(a)}&Unit sphere&$\E[\|\bx-\bx^*\|^2]\le \epsilon^2$ &$O(\ell^2 d^2 \epsilon^{-4})$\\
\cite{izzo2021learn} \red{(d)}&PerfGD& $\checkmark$ &Finite difference&$\min_{t} \|\nabla F(\bx_t)\|^2 \le \epsilon^2 + error$&$O(\ell^2 d^{\frac{3}{2}} \epsilon^{-4})$ \red{(e)}\\
\cite{izzo2022learn} \red{(d)}&Stateful PerfGD& $\checkmark$ &Finite difference &$\min_{t} \|\nabla F(\bx_t)\|^2 \le \epsilon^2 + error$&$\tilde{O}(\ell^9 d^{\frac{51}{2}} \epsilon^{-10}(\log{p^{-1}})^{\frac{45}{2}})$ \red{(b), (f)}\\
\cite{liu2023time}&DFO$(\lambda)$ Algorithm & $\checkmark$ &Unit sphere&$\E[\|\nabla F(\hat{\bx})\|]\le \epsilon$&$\tilde{O}(\ell^6 d^2 \epsilon^{-6})$\\
\cite{hikima2025zeroth}& Variance-reduced ZO &$\checkmark$&Gaussian&$\E[\|\nabla F(\hat{\bx})\|]\le \epsilon$&$O(d^{\frac{9}{2}} \epsilon^{-6})$\\
\cite{hikima2025guided}& Guided ZO &$\checkmark$&Gaussian&$\E[\|\nabla F(\hat{\bx})\|]\le \epsilon$&$O(d^4 \epsilon^{-6})$\\
\cite{bhavsar2022nonasymptotic}&RSG-BGO&$\checkmark$&Gaussian &$\E[\|\nabla F(\hat{\bx})\|]\le \epsilon$&$O(d^4 \epsilon^{-6})$ \red{(g)}\\
\cite{la2025gradient}&RSG (Biased Gradient)&$\checkmark$&Unified direction \red{(h)}&$\E[\|\nabla F(\hat{\bx})\|]\le \epsilon$&$O(d^4 \epsilon^{-6})$\\
\hline
\textbf{Ours} &Simple ZO&$\checkmark$&Coordinate-wise&$\E[\|\nabla F(\hat{\bx})\|]\le \epsilon$&$O(d^3 \epsilon^{-6})$ or $O(d^{\frac{5}{2}} \epsilon^{-5})$ \red{(i)} \\
& &$\checkmark$&Multiple Unit Sphere&$\E[\|\nabla F(\hat{\bx})\|]\le \epsilon$&$O(d^2\epsilon^{-6})$ or $O(d^2\epsilon^{-5})$ \red{(i)} \\
& &$\checkmark$&Multiple Gaussian&$\E[\|\nabla F(\hat{\bx})\|]\le \epsilon$&$O(d^2 \epsilon^{-6})$ or $O(d^2 \epsilon^{-5})$ \red{(i)} \\
 &&$\checkmark$&Multiple Rademacher&$\E[\|\nabla F(\hat{\bx})\|]\le \epsilon$&$O(d^3 \epsilon^{-6})$ or $O(d^{\frac
{5}{2}} \epsilon^{-5})$ \red{(i)}, \red{(j)} \\
\hline
\end{tabular}}
\end{table}

\subsection{Alternative approaches for stochastic problems with decision-dependent distributions}
\label{subsec:related_decision_dependent_noise}
We briefly review methods other than zeroth-order methods for problem~\eqref{eq:opt_problem}.

\paragraph{Retraining methods \citep{perdomo2020performative,mendler2020stochastic,mofakhami2023performative,NEURIPS2024_1055c730}.}
These methods seek a \emph{performatively stable point} 
$
\bx_{\mathrm{PS}}
\in \arg\min_{\bx}\,
\mathbb{E}_{\bmxi \sim D(\bx_{\mathrm{PS}})}[f(\bx,\bmxi)]
$
by repeatedly updating the decision vector while treating the distribution as fixed at the current iterate.
A typical method is \emph{repeated gradient descent} \citep{perdomo2020performative}:
$$
\bx_{k+1}
:= \mathrm{proj}_{\mathcal C}\!\big(\bx_k-\eta_k\,\mathbb{E}_{\bmxi\sim D(\bx_k)}[\nabla_{\bx} f(\bx_k,\bmxi)]\big),
$$
where $\mathcal C$ is the feasible set and $\mathrm{proj}_{\mathcal C}$ denotes Euclidean projection.
Subsequent work studies variants of this update \citep{mendler2020stochastic} and 
weaker assumptions on the objective \citep{mofakhami2023performative,NEURIPS2024_1055c730}.
However, the target of these methods is a \emph{performatively stable} solution, which may not exist in the nonconvex setting.

\paragraph{Stochastic first-order methods \citep{liu2024bayesian,hikima2025stochastic}.}
Assuming that $D(\bx)$ is known, these methods use the unbiased stochastic gradient for problem \eqref{eq:opt_problem} such as
\begin{align}
\nabla_{\bx} f(\bx,\bmxi)\;+\; f(\bx,\bmxi)\,\nabla_{\bx}\log \mpr(\bmxi\mid\bx),
\qquad \bmxi \sim D(\bx).
\label{eq:full_gradient_info}
\end{align}
While these methods achieve rapid convergence to stationary points, their assumption that $D(\bx)$ is known, i.e., access to 
$\nabla_{\bx}\log \mpr(\bmxi\mid\bx)$, makes them inapplicable to our setting, where $D(\bx)$ is unknown.

\paragraph{Distribution-modeling approaches \citep{miller2021outside,lin2024plug}.}
These methods estimate the distribution map $D(\cdot)$ and then optimize~\eqref{eq:opt_problem} using the estimated distribution.
They can perform well when the objective function and the distribution satisfy some assumptions.
For example, \cite{miller2021outside} assume that $f$ is strongly convex and $D(\bx)$ belongs to a location-scale family (specifically, the mean of the random variables is $A\bx$ for some constant matrix $A$, and the distribution of the centered random variables is independent of $\bx$).
Our approach does not impose such assumptions.

\paragraph{Search-based optimization \citep{jagadeesan2022regret,bergstra2012random,frazier2018tutorial,xue2020novel}.}
Search-based optimization such as Bayesian optimization \citep{frazier2018tutorial}, random search \citep{bergstra2012random}, and the sparrow search algorithm \citep{xue2020novel} can be used for problem \eqref{eq:opt_problem}.
In particular, \citet{jagadeesan2022regret} propose an efficient zooming algorithm for problem~\eqref{eq:opt_problem} that leverages knowledge about the objective $f$.
While these search-based methods are effective for finding global optima, their sample complexity  grows exponentially with the dimension $d$.
In contrast, our work focuses on finding a stationary point and establishes convergence guarantees with substantially lower sample complexity.

\section{Assumptions} \label{sec:assumption}
We make the following assumptions.
\begin{assumption} \label{asm:F_sample_var_bound}
There exists a constant $\sigma \in \R_{\ge 0}$ satisfying 
\begin{align*}
\E_{ \bmxi \sim D(\bx)} [ \left(F(\bx)-f(\bx, \bmxi)\right)^2] \le \sigma^2, \qquad
\forall \bx\in\R^d.
\end{align*}
\end{assumption}

\begin{assumption} \label{asm:F_smooth}
$F(\bx)$ is $M$-smooth, that is,
$$ \|\nabla F(\hat{\bx})- \nabla F(\bx)\| \le M \|\hat{\bx}-\bx\|, \qquad
\forall \bx,\hat{\bx}\in\R^d.$$
\end{assumption}

Under Assumptions \ref{asm:F_sample_var_bound} and \ref{asm:F_smooth}, we can guarantee convergence of our method.
Moreover, if the following assumption holds in addition to Assumptions \ref{asm:F_sample_var_bound} and \ref{asm:F_smooth}, we can achieve lower sample complexities.

\begin{assumption}\label{asp:F_hessian_lipschitz}
The Hessian of $F$ is $H$-Lipschitz continuous for all $\bx \in \R^d$, that is,
$$ \|\nabla^2 F(\hat{\bx}) - \nabla^2 F(
\bx)\|_{\textrm{op}} \le H \|\hat{\bx}-\bx\|, \qquad
\forall \bx,\hat{\bx}\in\R^d.$$
\end{assumption}

Assumption \ref{asm:F_sample_var_bound} is required for approximating $F(\bx)$ by noisy function-value observations $f(\bx,\bmxi)$.
Since the objective function  involves randomness, such an assumption is needed to evaluate the objective value by its sample.
Assumption \ref{asm:F_smooth} is standard in convergence analysis, ensuring the accuracy of the first-order approximation via Taylor expansion.
This is because descent methods with (estimated) gradients can be seen as optimizing a first-order approximation of the objective function at each iteration.
Assumption~\ref{asp:F_hessian_lipschitz} further improves the accuracy of the first-order approximation, which leads to improved accuracy of the gradient estimate and faster convergence.
Regarding Assumptions~\ref{asm:F_smooth}
and~\ref{asp:F_hessian_lipschitz},
Appendix~\ref{app:suf_condition} gives sufficient conditions
based on \citet[Lemma~1]{ray2022decision},
with corrected constants and a proof.\footnote{
The proof of Lemma~1 in \citet{ray2022decision}
contains a minor error in the application of the chain rule,
which affects the constants in the Lipschitz bounds.
We provide the corresponding corrections in
Appendix~\ref{app:suf_condition}.
}

\paragraph{Comparison with assumptions in existing studies.}
We compare our assumptions with those in \citep{ray2022decision} and \citep{hikima2025guided}.
First, our assumptions are looser than those in \citep{ray2022decision}.
Assumption 5 in \citep{ray2022decision} implies that $E_{ \bmxi \sim D(\bx)} [ \left(F(\bx)-f(\bx, \bmxi)\right)^2] \le (2\ell)^2$ for $\ell:=\sup_{\bx,\bmxi}|f(\bx,\bmxi)|$, 
which yields our Assumption \ref{asm:F_sample_var_bound}. Assumption 3 in \citep{ray2022decision} yields our Assumptions~\ref{asm:F_smooth} and \ref{asp:F_hessian_lipschitz}. 
Conversely, we do not require Assumption~1 (a, c, d) or Assumption 2 in \citep{ray2022decision}.\footnote{We also do not assume Assumption~1(a) or Assumption~4 of \citet{ray2022decision} in our analysis. 
However, they may be invoked as optional sufficient conditions to verify Assumptions~\ref{asm:F_smooth} and~\ref{asp:F_hessian_lipschitz}; see Appendix~\ref{app:suf_condition} for details.
} 
Next, regarding \citep{hikima2025guided}, our analysis does not require Assumptions 4.2 and 4.3 in \citep{hikima2025guided}; instead, in addition to Assumptions 4.1 and 4.4 in \citep{hikima2025guided}, we impose Assumption~\ref{asp:F_hessian_lipschitz} in this paper to achieve lower sample complexity.
However, without Assumption~\ref{asp:F_hessian_lipschitz}, our convergence analysis yields a strictly lower sample complexity than that of \citep{hikima2025guided}.

\section{Theoretical analysis for variants of gradient estimators}
\label{sec:theoretical_analysis}
Several options can be considered as gradient estimators for problem \eqref{eq:opt_problem}. 
In Section \ref{subsec:basic_form}, we first describe the basic form of gradient estimators for problem \eqref{eq:opt_problem}.
Then, in Section \ref{subsec:desired_prop}, we give a simple zeroth-order method and discuss the desirable property of gradient estimators.
Finally, in Section \ref{subsec:prop_each_method}, we evaluate each gradient estimator in terms of the desirable property and provide sample complexity analyses.

\subsection{Basic form of gradient estimator for problem \eqref{eq:opt_problem}} \label{subsec:basic_form}
As a zeroth-order gradient estimator for problem \eqref{eq:opt_problem}, we consider the following basic form:
\begin{align}
\bg:= &\frac{\kappa}{2N\mu} \sum_{i=1}^{N} \left(f(\bx+\mu \bv^i,\bmxi^{1,i})-f(\bx - \mu \bv^i,\bmxi^{2,i})\right) \bv^i. \label{eq:standard_form}
\end{align}
where $\kappa \in \R_{>0}$, $N \in \N$, $\mu \in \R_{>0}$, $\bv^i \in \R^d$, $\bmxi^{1,i} \sim D(\bx+\mu \bv^i)$, and $\bmxi^{2,i} \sim D(\bx-\mu \bv^i)$.\footnote{Throughout the analysis of our estimators,
$\bmxi^{1,i}$ and $\bmxi^{2,i}$ represent the latent randomness
underlying the noisy function-value observations; our methods observe
only the corresponding function values, not the realizations of
$\bmxi^{1,i}$ and $\bmxi^{2,i}$ themselves.}
The choices of $(\kappa, N, \bv^i)$ depend on the method:
for zeroth-order methods with coordinate-wise basis, $\kappa:=d$, $N:=d$, and $\bv^i$ is the $d$-dimensional vector whose $i$-th component is $1$ and all others are $0$;
for zeroth-order methods via smoothing on a sphere, $\kappa:=d$,
$N \ge 1$, and $\bv^i$ follows the uniform distribution on the unit sphere;
in the case of Gaussian smoothing methods, $\kappa:=1$,
$N \ge 1$, and $\bv^i$ is a random direction following a Gaussian distribution.
In Section \ref{subsec:prop_each_method}, we conduct a theoretical analysis of these variants.

\paragraph{Difference from gradient estimators for problem \eqref{eq:opt_problem_no_dependent}.}
For problem \eqref{eq:opt_problem_no_dependent} with a decision-independent distribution, the gradient estimator can be written as follows \citep{ghadimi2013stochastic,iwakiri2022single}:
\begin{align}
\bg':= &\frac{1}{2\mu}  \left(f(\bx+\mu \bv,\bmxi)-f(\bx-\mu \bv,\bmxi)\right) \bv, \label{eq:normal_basic_form}
\end{align}
where $\mu \in \R_{>0}$, $\bv \in \R^d$, and $\bmxi \sim D$.
The gradient estimator \eqref{eq:normal_basic_form} differs from \eqref{eq:standard_form} in that it uses the same random variable $\bmxi$ in both the first and second terms of the estimator. 
This has a substantial impact on the accuracy of the gradient estimator. 
For the estimator \eqref{eq:normal_basic_form}, the estimation accuracy improves as the smoothing parameter $\mu$ decreases; 
accordingly, prior analyses have typically established convergence guarantees by taking $\mu$ sufficiently small \cite[Corollary 3.3]{ghadimi2013stochastic}. 
In contrast, in the estimator \eqref{eq:standard_form}, the random variables in the forward and backward terms differ, which introduces noise. 
This noise is amplified as $\mu$ becomes smaller, necessitating an appropriate choice of $\mu$.

\begin{figure}[htbp]
        \centering
        \includegraphics[width=15cm]{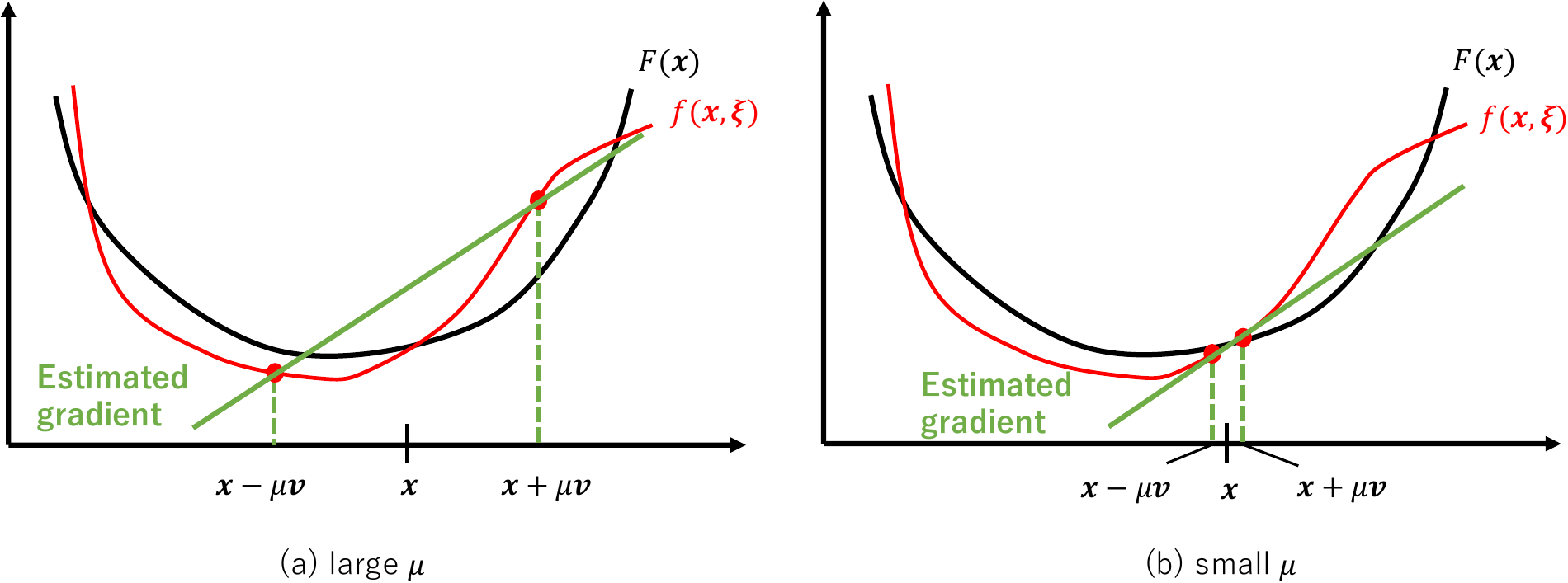}
    \caption{Conceptual illustration of gradient estimation with decision-independent noise for a one-dimensional problem. A smaller $\mu$ allows for an accurate estimation of $\nabla f(\bx, \bmxi)$.}
        \label{fig:indep}
\end{figure}

\begin{figure}[htbp]
        \centering
        \includegraphics[width=15cm]{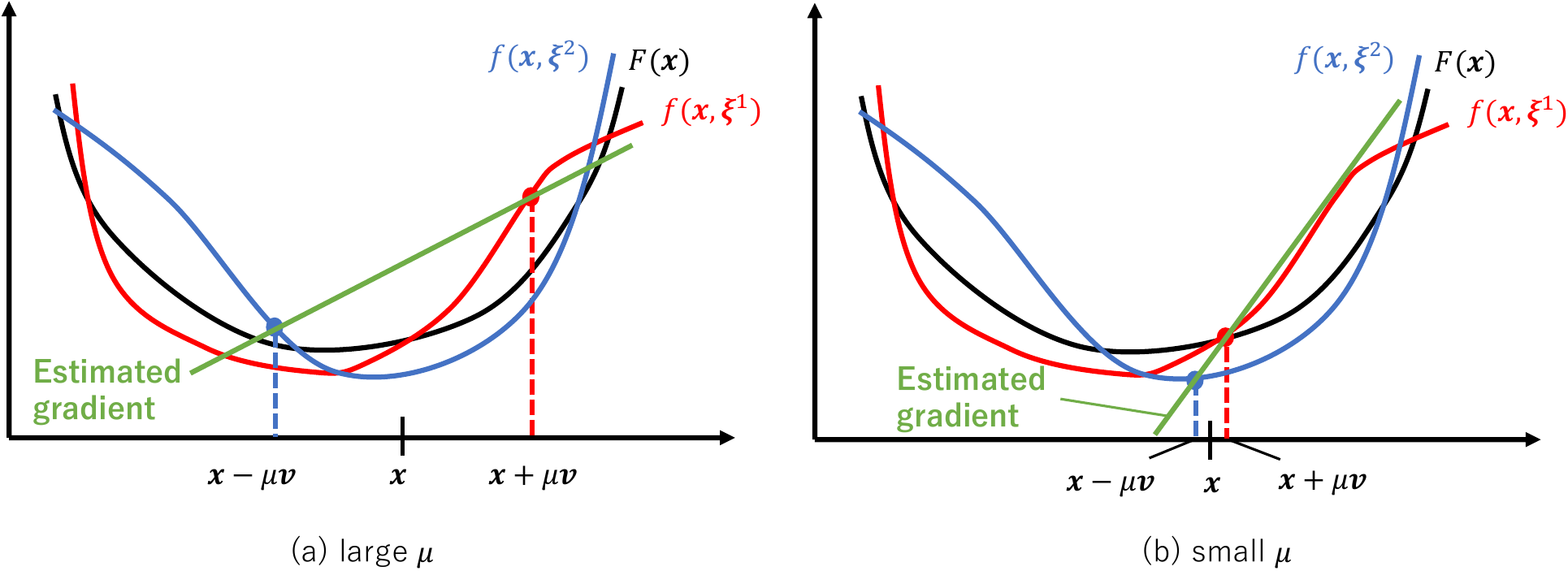}
\caption{Conceptual illustration of gradient estimation in the decision-dependent noise case for a one-dimensional problem with $N=1$. 
Since the noisy function-value observations at the two evaluation
points involve independent random variables $\bmxi^1$ and $\bmxi^2$,
decreasing the smoothing parameter $\mu$ significantly amplifies the
estimation variance.}
    \label{fig:dep}
\end{figure}

Figures \ref{fig:indep} and \ref{fig:dep} visualize the gradient estimation in a one-dimensional problem.
In the independent case (Figure \ref{fig:indep}), choosing a smaller smoothing parameter $\mu$ allows for an accurate approximation of the directional derivative for $f(\bx, \bmxi)$. 
Then, under standard assumptions such as $\mathbb{E}[\|\nabla f(\bx, \bmxi) - \nabla F(\bx)\|^2] \le \sigma^2$, the variance of the gradient estimator from the true gradient can be well-controlled.
Conversely, in our decision-dependent case (Figure \ref{fig:dep}), the noisy function-value observations at the perturbed points involve \textbf{different} noise variables ($\bmxi_1$ and $\bmxi_2$). 
Consequently, decreasing the smoothing parameter $\mu$ significantly amplifies the variance (error) of the estimator.

\subsection{Descent algorithm with gradient estimator}
\label{subsec:desired_prop}
We consider the following Algorithm \ref{alg:simple}.
\begin{algorithm}[h]
\caption{Descent algorithm with gradient estimator}
\label{alg:simple}
\begin{algorithmic}[h] 
\State {\bfseries Input:} total number of iterations $T$, step-size parameter $\eta$, and initial iterate $\bx_0\in \R^d$
\For{$t = 0, 1, 2, \ldots, T$}
\State Obtain an approximate gradient $\bg_t$, which may involve randomness
\State Update $\bx_{t+1} \gets \bx_t - \eta \bg_t$
\EndFor
\State {\bfseries Output:} $\bar{\bx}_T$ chosen uniformly at random from $\{\bx_t\}_{t=0}^{T}$
\end{algorithmic}
\end{algorithm}

Then, for Algorithm \ref{alg:simple}, the following lemma holds.
Here, let $\mathbb{E}[\cdot]$ be the total expectation over all algorithmic randomness, including the random choice of the output iterate $\bar{\bx}_T$.

\begin{lemma} \label{lem:simple_descent}
Suppose that Assumption \ref{asm:F_smooth} holds and $\eta\le \frac{1}{4M}$. 
Then, Algorithm \ref{alg:simple} obtains $\bar{\bx}_T$ such that
\[
\E[\|\nabla F(\bar{\bx}_T)\|^2] 
 \le \frac{4(F(\bx_0) - F^*)}{\eta (T+1)}  + \frac{3}{T+1} \sum_{t=0}^T\E[\|\nabla F(\bx_t)-\bg_t\|^2],
\]
where $F^* := \inf_{\bx} F(\bx)$.
\end{lemma}

\begin{proof}
By Lemma \ref{lem:first_order_taylor_remainder}, which follows from Assumption \ref{asm:F_smooth}, and the update rule $\bx_{t+1} = \bx_t - \eta \bg_t$ from the algorithm, we have 
\begin{align*}
F(\bx_{t+1})
    &\le F(\bx_t) + \nabla F(\bx_t)^\top (\bx_{t+1}- \bx_t)+\frac{M}{2}\|\bx_{t+1}- \bx_t\|^2 \\
&= F(\bx_t) - \eta \nabla F(\bx_t)^\top \bg_t+\frac{M\eta^2}{2}\|\bg_t\|^2\\
&= F(\bx_t) - \eta\|\nabla F(\bx_t)\|^2 + \eta\nabla F(\bx_t)^\top (\nabla F(\bx_t)-\bg_t) + \frac{M\eta^2}{2}\|\bg_t\|^2\\
&\overset{(a)}{\le} F(\bx_t) - \frac{\eta}{2} \|\nabla F(\bx_t)\|^2 + \frac{\eta}{2}\|\nabla F(\bx_t)-\bg_t\|^2 + \frac{M\eta^2}{2}\|\bg_t\|^2\\
&=F(\bx_t) - \frac{\eta}{2} \|\nabla F(\bx_t)\|^2 + \frac{\eta}{2}\|\nabla F(\bx_t)-\bg_t\|^2 + \frac{M\eta^2}{2}\|\bg_t-\nabla F(\bx_t)+\nabla F(\bx_t)\|^2\\
&\overset{(b)}{\le} F(\bx_t) - \frac{\eta}{2} \|\nabla F(\bx_t)\|^2 + \frac{\eta}{2}\|\nabla F(\bx_t)-\bg_t\|^2 + M\eta^2 \|\nabla F(\bx_t)-\bg_t\|^2 + M\eta^2 \|\nabla F(\bx_t)\|^2\\
&\overset{(c)}{\le}F(\bx_t) - \frac{\eta}{4} \|\nabla F(\bx_t)\|^2 + \frac{3\eta}{4} \|\nabla F(\bx_t)-\bg_t\|^2,
\end{align*}
where (a) is due to the fact that
\begin{align*}
 \eta\nabla F(\bx_t)^\top (\nabla F(\bx_t)-\bg_t)= (\sqrt{\eta}\nabla F(\bx_t))^\top (\sqrt{\eta} (\nabla F(\bx_t)-\bg_t))
\le \frac{\eta}{2}\|\nabla F(\bx_t)\|^2 + \frac{\eta}{2} \|\nabla F(\bx_t)-\bg_t\|^2
\end{align*}
from Young's inequality, (b) uses $\|\bx+\by\|^2\leq 2\|\bx\|^2+2\|\by\|^2$, and (c) uses $\eta\le \frac{1}{4M}$.

Rearranging terms and taking the expectation,
\begin{align*}
\E[\|\nabla F(\bx_t)\|^2] 
&\le \E\left[\frac{4}{\eta} (F(\bx_t) - F(\bx_{t+1})) + 3 \|\nabla F(\bx_t)-\bg_t\|^2\right]\\
& =\frac{4}{\eta}(\E[F(\bx_t)] - \E[F(\bx_{t+1})]) + 3\E[\|\nabla F(\bx_t)-\bg_t\|^2].
\end{align*}

Taking the average over $t=0, 1,\dots,T$ on both sides of the above inequality, we have
\begin{align*}
\frac{1}{T+1}\sum_{t=0}^T\E[\|\nabla F(\bx_t)\|^2] 
&\le \frac{4(F(\bx_0) - \E[F(\bx_{T+1})])}{\eta (T+1)}  + \frac{3}{T+1} \sum_{t=0}^T\E[\|\nabla F(\bx_t)-\bg_t\|^2] \\
& \overset{(a)}{\le} \frac{4(F(\bx_0) - F^*)}{\eta (T+1)}  + \frac{3}{T+1} \sum_{t=0}^T\E[\|\nabla F(\bx_t)-\bg_t\|^2],
\end{align*}
where (a) comes from the fact that $F^*= \inf_{\bx} F(\bx)  \le \E[F(\bx_{T+1})]$.
\end{proof}

Lemma \ref{lem:simple_descent} indicates that if $\E\left[\|\nabla F(\bx_t)-\bg_t\|^2\right]$ is small for all $t\in \{0,1,\dots,T\}$, then the expected gradient norm of the output solution becomes small, bringing the iterate closer to a stationary point. 
Therefore, by evaluating $\E\left[\|\nabla F(\bx_t)-\bg_t\|^2\right]$ for each gradient estimator, we discuss which gradient estimators are effective for problem \eqref{eq:opt_problem}.

\subsection{Theoretical analysis of various gradient estimators}
\label{subsec:prop_each_method}
In this section, for a fixed decision vector $\bx$, we analyze 
\[
\E\left[\|\nabla F(\bx)-\bg\|^2\right],
\]
where the expectation is taken with respect to the randomness used to construct 
the gradient estimator $\bg$, such as random search directions and samples from 
decision-dependent distributions.\footnote{When this bound is applied to Algorithm~1 with 
$\bx=\bx_t$, this expectation corresponds to the conditional expectation given the history up 
to iteration $t$.}
We then derive the resulting sample complexity for gradient estimators  that use (i) coordinate-wise directions, (ii) directions sampled uniformly from the unit sphere, and (iii) directions drawn from a multivariate Gaussian distribution.
Our proofs mimic the proofs in  \citep{berahas2022theoretical}.

\subsubsection{Gradient estimator with coordinate-wise directions} \label{subsec:grad_est_coordinate}
The gradient estimator with coordinate-wise directions can be written as follows: 
\begin{align}
\bg^{\rm{co}}(\bx, \Xi_1, \Xi_2):= &\frac{1}{2\mu} \sum_{i=1}^{d} \left(f(\bx+\mu \bme_i,\bmxi^{1,i})-f(\bx - \mu \bme_i,\bmxi^{2,i})\right) \bme_i,\label{eq:one_batch_gradient_est}
\end{align}
where $\mu \in \R_{>0}$ and $\bme_i$ is a $d$-dimensional vector whose $i$-th component is $1$ and all others are $0$. 
Moreover, $\Xi_1:=\{\bmxi^{1,i}\}_{i=1}^{d}$ and $\Xi_2:=\{\bmxi^{2,i}\}_{i=1}^{d}$, where $\bmxi^{1,i} \sim D(\bx+\mu \bme_i)$ and $\bmxi^{2,i} \sim D(\bx-\mu \bme_i)$ for $i=1, \dots, d$.
We consider the mini-batch version of the gradient estimator:
\begin{align}
\bg^{\rm{co}}_m(\bx, \{\Xi_1^{(j)}\}_{j=1}^m, \{\Xi_2^{(j)}\}_{j=1}^m): =\frac{1}{m}\sum_{j=1}^{m} \bg^{\rm{co}}(\bx,\Xi_1^{(j)},\Xi_2^{(j)}), \label{eq:mini_batch_gradient_est}
\end{align}
where $m \in \N$ is the mini-batch size for $\bmxi$.
We also discuss the appropriate choice of $m$ in our analysis.

Then, the following lemmas hold for the gradient estimator.
\begin{lemma}
\label{lem:g_variance_FD}
Suppose that Assumptions \ref{asm:F_sample_var_bound} and \ref{asm:F_smooth} hold.
Let $\bx \in \R^d$, $\Xi_1^{(j)}:=\{\bmxi^{1,i,(j)}\}_{i=1}^{d}$, and $\Xi_2^{(j)}:=\{\bmxi^{2,i,(j)}\}_{i=1}^{d}$, where $\bmxi^{1,i,(j)} \sim D(\bx+\mu \bme_i)$ and $\bmxi^{2,i,(j)} \sim D(\bx-\mu \bme_i)$ for $i \in [d]$ and $j \in [m]$.
Then, the following holds.
\begin{align}
&\E_{ \{\Xi_1^{(j)}\}_{j=1}^{m},\{\Xi_2^{(j)}\}_{j=1}^{m}} \left[ \left\| \bg^{\rm{co}}_m(\bx, \{\Xi_{1}^{(j)}\}_{j=1}^m, \{\Xi_{2}^{(j)}\}_{j=1}^m)  - \nabla F(\bx) \right\|^2  \right] \le \frac{3\sigma^2d}{2\mu^2 m} +\frac{3M^2 d\mu^2}{4}.
\end{align}
\normalsize
\end{lemma}

\begin{proof}
First, we have
\begin{align*}
&   \E_{ \{\Xi_1^{(j)}\}_{j=1}^{m},\{\Xi_2^{(j)}\}_{j=1}^{m}} \left[ \left\| \bg^{\rm{co}}_m(\bx, \{\Xi_{1}^{(j)}\}_{j=1}^m, \{\Xi_{2}^{(j)}\}_{j=1}^m)   - \nabla F(\bx) \right\|^2  \right] \\
&=   \E_{ \{\Xi_1^{(j)}\}_{j=1}^{m},\{\Xi_2^{(j)}\}_{j=1}^{m}} \left[ \left\|\sum_{i=1}^{d}\frac{\sum_{j=1}^{m} f(\bx+\mu \bme_i, \bmxi^{1,i,(j)})-\sum_{j=1}^{m} f(\bx-\mu \bme_i, \bmxi^{2,i,(j)}) }{2m\mu}\bme_i  -  \nabla F(\bx) \right\|^2  \right]  \\
&= \E_{ \{\Xi_1^{(j)}\}_{j=1}^{m},\{\Xi_2^{(j)}\}_{j=1}^{m}} \Bigg[ \Bigg\| \sum_{i=1}^{d} \frac{F(\bx+\mu \bme_i)-F(\bx - \mu \bme_i)}{2\mu}\bme_i - \nabla F(\bx)\\
 & \hspace{40mm} + \sum_{i=1}^{d}\frac{\frac{1}{m}\sum_{j=1}^{m} f(\bx+\mu \bme_i, \bmxi^{1,i,(j)}) - F(\bx+\mu \bme_i)}{2\mu}\bme_i
 \\
    & \hspace{40mm} + \sum_{i=1}^{d}\frac{-\frac{1}{m}\sum_{j=1}^{m} f(\bx-\mu \bme_i, \bmxi^{2,i,(j)}) +F(\bx-\mu \bme_i)}{2\mu} \bme_i\Bigg\|^2 \Bigg] \\
&\overset{(a)}{\le}  3\left\| \sum_{i=1}^{d}\frac{F(\bx+\mu \bme_i)-F(\bx-\mu \bme_i)}{2\mu}\bme_i  - \nabla F(\bx) \right\|^2 \\
    & \quad + 3\E_{ \{\Xi_1^{(j)}\}_{j=1}^{m}} \Bigg[ \bigg\| \sum_{i=1}^{d}\frac{\frac{1}{m}\sum_{j=1}^{m} f(\bx+\mu \bme_i, \bmxi^{1,i,(j)})-F(\bx+\mu \bme_i)}{2\mu} \bme_i \bigg\|^2 \Bigg] \\
    & \quad + 3\E_{\{\Xi_2^{(j)}\}_{j=1}^{m}} \Bigg[ \bigg\| \sum_{i=1}^{d}\frac{-\frac{1}{m}\sum_{j=1}^{m} f(\bx-\mu \bme_i, \bmxi^{2,i,(j)}) +F(\bx-\mu \bme_i)}{2\mu}\bme_i \bigg\|^2 \Bigg] \\
&= 3\left\| \sum_{i=1}^{d}\frac{F(\bx+\mu \bme_i)-F(\bx-\mu \bme_i)}{2\mu}\bme_i  - \nabla F(\bx) \right\|^2 \\
& \quad + 3\E_{ \{\Xi_1^{(j)}\}_{j=1}^{m}} \Bigg[\sum_{i=1}^{d}\left(\frac{\frac{1}{m}\sum_{j=1}^{m} f(\bx+\mu \bme_i, \bmxi^{1,i,(j)})-F(\bx+\mu \bme_i)}{2\mu} \right)^2 \Bigg] \\
& \quad + 3\E_{\{\Xi_2^{(j)}\}_{j=1}^{m}} \Bigg[\sum_{i=1}^{d}\left( \frac{-\frac{1}{m}\sum_{j=1}^{m} f(\bx-\mu \bme_i, \bmxi^{2,i,(j)}) +F(\bx-\mu \bme_i)}{2\mu} \bigg)^2 \right] 
\end{align*}
where (a) follows from Lemma \ref{lem:1_to_n_two_norm_bound}.
Then, from  Assumption~\ref{asm:F_sample_var_bound} and Lemma \ref{lem:minibatch_var_reduction},
\begin{align}
&   \E_{ \{\Xi_1^{(j)}\}_{j=1}^{m},\{\Xi_2^{(j)}\}_{j=1}^{m}} \left[ \left\| \bg^{\rm{co}}_m(\bx, \{\Xi_{1}^{(j)}\}_{j=1}^m, \{\Xi_{2}^{(j)}\}_{j=1}^m)   - \nabla F(\bx) \right\|^2  \right] \nonumber \\
&\le  3\left\| \sum_{i=1}^{d}\frac{F(\bx+\mu \bme_i)-F(\bx-\mu \bme_i)}{2\mu}\bme_i  - \nabla F(\bx) \right\|^2 +\frac{3\sigma^2d}{4\mu^2 m}  + \frac{3\sigma^2d}{4\mu^2 m} \nonumber \\
&=
3\left\| \sum_{i=1}^{d}\frac{F(\bx+\mu \bme_i)-F(\bx-\mu \bme_i)}{2\mu}\bme_i  - \nabla F(\bx) \right\|^2 + \frac{3\sigma^2d}{2\mu^2 m} \label{prog:co_var_ineq}\\
&\overset{(a)}{\le} \frac{3\sigma^2d}{2\mu^2 m} + \frac{3M^2d\mu^2}{4}, \nonumber
\end{align}
\normalsize
where (a) follows from Lemma \ref{lem:g_bound_FD}.
\end{proof}

\begin{lemma}
\label{lem:g_variance_FD_H_smooth}
Suppose that Assumptions \ref{asm:F_sample_var_bound} and \ref{asp:F_hessian_lipschitz} hold.
Let $\bx \in \R^d$, $\Xi_1^{(j)}:=\{\bmxi^{1,i,(j)}\}_{i=1}^{d}$, and $\Xi_2^{(j)}:=\{\bmxi^{2,i,(j)}\}_{i=1}^{d}$, where $\bmxi^{1,i,(j)} \sim D(\bx+\mu \bme_i)$ and $\bmxi^{2,i,(j)} \sim D(\bx-\mu \bme_i)$ for $i \in [d]$ and $j \in [m]$.
Then, the following holds.
\begin{align}
&\E_{ \{\Xi_1^{(j)}\}_{j=1}^{m},\{\Xi_2^{(j)}\}_{j=1}^{m}} \left[ \left\| \bg^{\rm{co}}_m(\bx, \{\Xi_{1}^{(j)}\}_{j=1}^m, \{\Xi_{2}^{(j)}\}_{j=1}^m)  - \nabla F(\bx) \right\|^2  \right] \le  
\frac{3\sigma^2d}{2\mu^2m} + \frac{H^2\mu^4d}{12}.
\end{align}
\normalsize
\end{lemma}

\begin{proof}
For all $i \in [d]$, we have
\begin{align}
&\left|\frac{F(\bx+\mu \bme_i)-F(\bx-\mu \bme_i)}{2\mu}  - \nabla F(\bx)^\top \bme_i\right| \nonumber\\
&=\left|\frac{F(\bx+\mu \bme_i)-F(\bx) - \mu \nabla F(\bx)^\top \bme_i}{2\mu}  - \frac{F(\bx-\mu \bme_i)-F(\bx) + \mu \nabla F(\bx)^\top \bme_i}{2\mu}\right| \nonumber\\
&= \Bigg|\frac{F(\bx+\mu \bme_i)-F(\bx) - \mu \nabla F(\bx)^\top \bme_i -\frac{\mu^2}{2} \bme_i^\top \nabla^2 F(\bx) \bme_i}{2\mu} \\
&\quad \  - \frac{F(\bx-\mu \bme_i)-F(\bx) + \mu \nabla F(\bx)^\top \bme_i-\frac{\mu^2}{2} \bme_i^\top \nabla^2 F(\bx) \bme_i}{2\mu}\Bigg| \nonumber\\
& \le  \left|\frac{F(\bx+\mu \bme_i)-F(\bx) - \mu \nabla F(\bx)^\top \bme_i -\frac{\mu^2}{2} \bme_i^\top \nabla^2 F(\bx) \bme_i}{2\mu} \right| \\
&\quad + \left| \frac{F(\bx-\mu \bme_i)-F(\bx) + \mu \nabla F(\bx)^\top \bme_i-\frac{\mu^2}{2} \bme_i^\top \nabla^2 F(\bx) \bme_i}{2\mu}\right| \nonumber\\
& \overset{(a)}{\le} \frac{H\|\mu \bme_i\|^3}{12\mu} + \frac{H\|\mu \bme_i\|^3}{12\mu} =\frac{H\mu^2}{6},\label{eq:hesse_lipschitz_continuous_first}
\end{align}
\normalsize
where (a) comes from Lemma \ref{lem:Hessian_Lipschitz_tr}, which follows from 
Assumption \ref{asp:F_hessian_lipschitz}.
As in the derivation of \eqref{prog:co_var_ineq} in the proof of Lemma \ref{lem:g_variance_FD},
\begin{align*}
&   \E_{ \{\Xi_1^{(j)}\}_{j=1}^{m},\{\Xi_2^{(j)}\}_{j=1}^{m}} \left[ \left\| \bg^{\rm{co}}_m(\bx, \{\Xi_{1}^{(j)}\}_{j=1}^m, \{\Xi_{2}^{(j)}\}_{j=1}^m)   - \nabla F(\bx) \right\|^2  \right] \\
& \le \frac{3\sigma^2d}{2\mu^2 m} + 3\left\| \sum_{i=1}^{d}\frac{F(\bx+\mu \bme_i)-F(\bx-\mu \bme_i)}{2\mu}\bme_i  - \nabla F(\bx) \right\|^2 \\
&=\frac{3\sigma^2d}{2\mu^2 m} + 3\sum_{i=1}^{d} \left( \frac{F(\bx+\mu \bme_i)-F(\bx-\mu \bme_i)}{2\mu}  - \nabla F(\bx)^\top \bme_i \right)^2.
\end{align*}
Therefore, from \eqref{eq:hesse_lipschitz_continuous_first},
\begin{align*}
&   \E_{ \{\Xi_1^{(j)}\}_{j=1}^{m},\{\Xi_2^{(j)}\}_{j=1}^{m}} \left[ \left\| \bg^{\rm{co}}_m(\bx, \{\Xi_{1}^{(j)}\}_{j=1}^m, \{\Xi_{2}^{(j)}\}_{j=1}^m)   - \nabla F(\bx) \right\|^2  \right] \le \frac{3\sigma^2d}{2\mu^2m}+ \frac{H^2\mu^4d}{12}.
\end{align*}
\end{proof}

\paragraph{Intuitive interpretation of Lemmas~\ref{lem:g_variance_FD} and~\ref{lem:g_variance_FD_H_smooth}.}
The first term in each upper bound is the error incurred when approximating $F(\bx)$ by a finite-sample average of $f(\bx,\bmxi)$.
As discussed in Section~\ref{subsec:basic_form} (\textbf{Difference from gradient estimators for problem~(2)}), this term is amplified as the smoothing parameter $\mu$ decreases.
The second term reflects the error introduced by approximating the gradient of $F$ via finite differences.
This error diminishes as $\mu$ decreases.
Therefore, there is a trade-off in $\mu$, which does not arise in stochastic zeroth-order methods \citep{ghadimi2013stochastic,iwakiri2022single} with decision-independent distributions.

By setting $\mu$ appropriately, we show the sample complexity for Algorithm \ref{alg:simple} with $\bg^{\rm{co}}_m$.
\begin{theorem} \label{thm:sample_complexity_co_grad_lip}
Suppose that Assumptions \ref{asm:F_sample_var_bound} and \ref{asm:F_smooth} hold.
Let 
$\bg_t:=\bg^{\rm{co}}_m(\bx_t, \{\Xi_1^{(j)}\}_{j=1}^m, \{\Xi_2^{(j)}\}_{j=1}^m)$, where $\bg^{\rm{co}}_m$ is defined by 
\eqref{eq:mini_batch_gradient_est}.
Let $\mu:= \sqrt[4]{\frac{2\sigma^2}{mM^2}}$, $m:=\Theta (d^2 \epsilon^{-4})$, $\eta\le \frac{1}{4M}$, and $T:=\Theta(\epsilon^{-2})$.
Then, the sample complexity, to satisfy $\E[\|\nabla F(\bar{\bx}_T)\|^2]\le \epsilon^2$ for output $\bar{\bx}_T$ of Algorithm \ref{alg:simple}, is $O(d^3\epsilon^{-6})$.
\end{theorem}
\begin{proof}
We have
\begin{align*}
\E[\|\nabla F(\bar{\bx}_T)\|^2] 
& \overset{(a)}{\le} \frac{4(F(\bx_0) - F^*)}{\eta (T+1)}  + \frac{3}{T+1} \sum_{t=0}^T\E[\|\nabla F(\bx_t)-\bg_t\|^2] \\
& \overset{(b)}{\le} \frac{4(F(\bx_0) - F^*)}{\eta (T+1)}  + \frac{9\sqrt{2} \sigma M d}{2\sqrt{m}}\\
&\overset{(c)}{=} O(\epsilon^2),
\end{align*}
where (a) comes from Lemma \ref{lem:simple_descent}; 
(b) follows from Lemma \ref{lem:g_variance_FD} and $\mu:= \sqrt[4]{\frac{2\sigma^2}{mM^2}}$; 
(c) is due to the facts that $T=\Theta(\epsilon^{-2})$ and $m=\Theta (d^2 \epsilon^{-4})$.
Since $2d$ samples are required to calculate \eqref{eq:one_batch_gradient_est}, the sample complexity is $O(dmT)=O(d^3\epsilon^{-6})$.
\end{proof}

\begin{theorem} \label{thm:sample_complexity_co_Hess_lip}
Suppose that Assumptions \ref{asm:F_sample_var_bound}--\ref{asp:F_hessian_lipschitz} hold.
Let 
$\bg_t:=\bg^{\rm{co}}_m(\bx_t, \{\Xi_1^{(j)}\}_{j=1}^m, \{\Xi_2^{(j)}\}_{j=1}^m)$, where $\bg^{\rm{co}}_m$ is defined by 
\eqref{eq:mini_batch_gradient_est}.
Let $\mu:= \sqrt[6]{\frac{18\sigma^2}{mH^2}}$, $m:=\Theta (d^{\frac{3}{2}} \epsilon^{-3})$, $\eta\le \frac{1}{4M}$, and $T:=\Theta(\epsilon^{-2})$.
Then, the sample complexity, to satisfy $\E[\|\nabla F(\bar{\bx}_T)\|^2]\le \epsilon^2$ for output $\bar{\bx}_T$ of Algorithm \ref{alg:simple}, is $O(d^{\frac{5}{2}}\epsilon^{-5})$.
\end{theorem}
\begin{proof}
We have
\begin{align*}
\E[\|\nabla F(\bar{\bx}_T)\|^2] 
& \overset{(a)}{\le} \frac{4(F(\bx_0) - F^*)}{\eta (T+1)}  + \frac{3}{T+1} \sum_{t=0}^T\E[\|\nabla F(\bx_t)-\bg_t\|^2] \\
& \overset{(b)}{\le} \frac{4(F(\bx_0) - F^*)}{\eta (T+1)}  + 3\left(\frac{3d^{3} H^{2} \sigma^{4}}{2 m^{2}}\right)^{1/3}\\
&\overset{(c)}{=} O(\epsilon^2),
\end{align*}
where (a) comes from Lemma \ref{lem:simple_descent}; 
(b) follows from Lemma \ref{lem:g_variance_FD_H_smooth} and $\mu:= \sqrt[6]{\frac{18\sigma^2}{mH^2}}$;
(c) is due to the facts that $T=\Theta(\epsilon^{-2})$ and $m=\Theta (d^{\frac{3}{2}} \epsilon^{-3})$.
Since $2d$ samples are required to calculate \eqref{eq:one_batch_gradient_est}, the sample complexity is $O(dmT)=O(d^{\frac{5}{2}}\epsilon^{-5})$.
\end{proof}

\subsubsection{Gradient Estimator via Smoothing on a Sphere}
The gradient estimator with directions drawn uniformly from the unit sphere can be written as follows:
\begin{align*}
\bg^{\rm{sp}}(\bx, S, \Xi_1, \Xi_2):= &\frac{d}{N} \sum_{i=1}^{N} \frac{f(\bx+\mu \bs^i,\bmxi^{1,i})-f(\bx - \mu \bs^i,\bmxi^{2,i})}{2\mu} \bs^i.
\end{align*}
Here, $\mu \in \R_{>0}$; 
$S:=\{\bs^i\}_{i=1}^N$, where $\bs^i$ follows the uniform distribution $\maS$ over the unit sphere; 
$\Xi_1:=\{\bmxi^{1,i}\}_{i=1}^N$, where $\bmxi^{1,i} \sim D(\bx+\mu \bs^i)$;
$\Xi_2:=\{\bmxi^{2,i}\}_{i=1}^N$, where $\bmxi^{2,i} \sim D(\bx-\mu \bs^i)$.
We consider the mini-batch version of the gradient estimator:
\begin{align}
\bg^{\rm{sp}}_m(\bx, S,\{\Xi_1^{(j)}\}_{j=1}^m, \{\Xi_2^{(j)}\}_{j=1}^m): =\frac{1}{m}\sum_{j=1}^{m} \bg^{\rm{sp}}(\bx,S,\Xi_1^{(j)},\Xi_2^{(j)}). \label{eq:mini_batch_gradient_est_sp}
\end{align}

Here, let
\begin{align}
&F_{\mu,\maB}(\bx):= \E_{\bs \sim \maB}[F(\bx + \mu \bs)], \label{def:F_mu_S}
\end{align}
where $\maB$ denotes the multivariate uniform distribution on a ball of radius $1$ centered at $\bm{0}$.
Then, the following lemma holds for the gradient estimator from \cite[Lemma 2.1]{flaxman2004online}.
\begin{lemma} \label{lem:estimator_sphere_unbiased}
Let $\bg^{\rm{sp}}_m(\bx, S,\{\Xi_1^{(j)}\}_{j=1}^m, \{\Xi_2^{(j)}\}_{j=1}^m)$ be defined by \eqref{eq:mini_batch_gradient_est_sp} and 
$F_{\mu,\maB}(\bx)$ be defined by \eqref{def:F_mu_S}.
Then, 
\begin{align*}
\E_{S, \{\Xi_{1}^{(j)}\}_{j=1}^m, \{\Xi_{2}^{(j)}\}_{j=1}^m}\left[\bg^{\rm{sp}}_m(\bx, S,\{\Xi_1^{(j)}\}_{j=1}^m, \{\Xi_2^{(j)}\}_{j=1}^m)\right]
&= \E_{\bs \sim \maS}\left[\frac{d(F(\bx+\mu \bs)-F(\bx - \mu \bs))}{2\mu}\bs\right] \\
&=\nabla F_{\mu,\maB}(\bx),
\end{align*}
\end{lemma}
\begin{proof}
\begin{align*}
\E_{\bs \sim \maS}\left[\frac{d(F(\bx+\mu \bs)-F(\bx - \mu \bs))}{2\mu}\bs\right] =\E_{\bs \sim \maS}\left[\frac{dF(\bx+\mu \bs)}{2\mu}\bs\right] +\E_{\bs \sim \maS}\left[\frac{dF(\bx + \mu \bs)}{2\mu}\bs\right] \overset{(a)}{=}\nabla F_{\mu,\maB}(\bx),
\end{align*}
where (a) comes from \cite[Lemma 2.1]{flaxman2004online}.
\end{proof}

From the preceding lemma, $\bg^{\mathrm{sp}}_m$ serves as an estimator of the gradient of the objective function averaged over a ball.
Then, we can bound the distance between the true gradient and the gradient estimator by the following lemmas.
In the proof, we let
\begin{align*}
    &\bg_F^{{\rm sp}}(\bx, S):= \frac{d}{N} \sum_{i=1}^N  \frac{F(\bx+\mu \bs^i)-F(\bx - \mu \bs^i)}{2\mu}\bs^i.
\end{align*}

\begin{lemma}
\label{lem:g_variance_sp_batch1}
Suppose that Assumptions \ref{asm:F_sample_var_bound} and \ref{asm:F_smooth} hold.
Then, the following holds.
\begin{align}
&\E_{S,\{\Xi_{1}^{(j)}\}_{j=1}^m, \{\Xi_{2}^{(j)}\}_{j=1}^m} \left[ \left\| \bg^{\rm{sp}}_m(\bx,S, \{\Xi_{1}^{(j)}\}_{j=1}^m, \{\Xi_{2}^{(j)}\}_{j=1}^m)  - \nabla F(\bx) \right\|^2  \right] \nonumber \\
&\le\frac{3\sigma^2d^2}{\mu^2 Nm} + 3M^2 \mu^2+ \frac{3M^2\mu^2d^2}{2N} + \frac{18d^2}{N(d+2)} \|\nabla F(\bx)\|^2.
\label{eq:sp_vound_grad_smooth}
\end{align}
\end{lemma}
\begin{proof}
We have 
\begin{align}
{\rm Var}(\bg_F^{{\rm sp}}(\bx, S)) 
&= \frac{1}{N} {\rm Var}\left( \frac{d(F(\bx+\mu \bs)-F(\bx - \mu \bs))}{2\mu}\bs\right)
\nonumber \\
&= \frac{1}{N} \E_{\bs \sim \maS}\left[\left(\frac{d(F(\bx+\mu \bs)-F(\bx - \mu \bs))}{2\mu}\bs\right) \left(\frac{d(F(\bx+\mu \bs)-F(\bx - \mu \bs))}{2\mu}\bs\right)^\top \right] 
\nonumber \\
& \quad - \frac{1}{N}\E_{\bs \sim \maS}\left[\frac{d(F(\bx+\mu \bs)-F(\bx - \mu \bs))}{2\mu}\bs\right] \E_{\bs \sim \maS}\left[\frac{d(F(\bx+\mu \bs)-F(\bx - \mu \bs))}{2\mu}\bs \right]^\top \nonumber \\
& \overset{(a)}{=} \frac{d^2}{N}\E_{\bs \sim \maS}\left[ \left( \frac{F(\bx+\mu \bs)-F(\bx - \mu \bs)}{2\mu}\right)^2  \bs \bs^\top \right]- \frac{1}{N} \nabla F_{\mu,\maB}(\bx) \nabla F_{\mu,\maB}(\bx)^\top, \nonumber
\end{align}
where (a) comes from Lemma \ref{lem:estimator_sphere_unbiased}.

Therefore, 
\begin{align}
{\rm Var}(\bg_F^{{\rm sp}}(\bx, S)) 
& \preceq  \frac{d^2}{N}\E_{\bs \sim \maS}\left[ \left( \frac{F(\bx+\mu \bs)-F(\bx - \mu \bs)}{2\mu} \right)^2 \bs \bs^\top  \right] \nonumber \\
& =\frac{d^2}{N}\E_{\bs \sim \maS}\left[ \left( \frac{F(\bx+\mu \bs)-F(\bx - \mu \bs) - 2\mu \nabla F(\bx)^\top \bs}{2\mu} + \nabla F(\bx)^\top \bs \right)^2 \bs \bs^\top  \right] \nonumber \\
& \preceq \frac{d^2}{N}\E_{\bs \sim \maS}\left[ 2\left( \frac{F(\bx+\mu \bs)-F(\bx - \mu \bs) - 2\mu \nabla F(\bx)^\top \bs}{2\mu} \right)^2 \bs \bs^\top + 2(\nabla F(\bx)^\top \bs)^2 \bs \bs^\top \right] \nonumber \\
& =\frac{2d^2}{N}\E_{\bs \sim \maS}\bigg[ \bigg( \frac{F(\bx+\mu \bs)-F(\bx) - \mu \nabla F(\bx)^\top \bs}{2\mu} \nonumber \\
    &\hspace{30mm} - \frac{F(\bx-\mu \bs)-F(\bx) + \mu \nabla F(\bx)^\top \bs}{2\mu} \bigg)^2 \bs \bs^\top  + (\nabla F(\bx)^\top \bs)^2 \bs \bs^\top \bigg] \label{eq:covar_ap} \\
&\overset{(a)}{\preceq}  \frac{2d^2}{N}\E_{\bs \sim \maS}\bigg[ \bigg( \frac{M\|\mu \bs\|^2}{4\mu} + \frac{M\|\mu \bs\|^2}{4\mu}\bigg)^2 \bs \bs^\top  + (\nabla F(\bx)^\top \bs)^2 \bs \bs^\top \bigg] \nonumber \\
& = \frac{2d^2}{N}\E_{\bs \sim \maS}\bigg[\frac{M^2\mu^2 \|\bs\|^4}{4}  \bs \bs^\top  + (\nabla F(\bx)^\top \bs)^2 \bs \bs^\top \bigg] \nonumber \\
& \overset{(b)}{=} \frac{2d^2}{N}\left(\frac{M^2\mu^2}{4d} I + \frac{\nabla F(\bx)^\top \nabla F(\bx)}{d(d+2)} I  + \frac{2}{d(d+2)} \nabla F(\bx) \nabla F(\bx)^\top \right) \nonumber \\
& \overset{(c)}{\preceq} \frac{2}{N}\left(\frac{M^2\mu^2d}{4} + \frac{3d}{(d+2)} \|\nabla F(\bx)\|^2 \right)I, \label{eq:var_g_F_sp_progress}
\end{align}
where (a) is due to Lemma \ref{lem:first_order_taylor_remainder}, which follows from Assumption \ref{asm:F_smooth}; (b) follows from Lemma \ref{lem:k_moment_ss_sp_noise};
(c) is due to the fact that $\nabla F(\bx) \nabla F(\bx)^\top \preceq \|\nabla F(\bx)\|^2I$ since 
$\bv^\top \bigl(\|\nabla F(\bx)\|^2 I - \nabla F(\bx)\nabla F(\bx)^\top \bigr) \bv
= \|\nabla F(\bx)\|^2 \|\bv\|^2 - (\nabla F(\bx)^\top \bv)^2
\;\ge\; 0$ for all $\bv \in \R^d$.
Here, 
\begin{align}
\E_{S}[\|\bg_F^{{\rm sp}}(\bx,S) - \nabla F_{\mu,\maB}(\bx) \|^2] 
& = \E_{S}\left[\sum_{i=1}^d \big((\bg_F^{{\rm sp}}(\bx,S) - \nabla F_{\mu,\maB}(\bx))_i\big)^2\right] \nonumber \\
& = \E_{S}\left[ {\rm tr}\left( (\bg_F^{{\rm sp}}(\bx,S) - \nabla F_{\mu,\maB}(\bx))(\bg_F^{{\rm sp}}(\bx,S) - \nabla F_{\mu,\maB}(\bx))^\top\right)\right] \nonumber \\
& = {\rm tr}\left(\E_{S}\left[ (\bg_F^{{\rm sp}}(\bx,S) - \nabla F_{\mu,\maB}(\bx))(\bg_F^{{\rm sp}}(\bx,S) - \nabla F_{\mu,\maB}(\bx))^\top\right] \right)\nonumber \\
& \overset{(a)}{=} {\rm tr}({\rm Var}(\bg_F^{{\rm sp}}(\bx, S))), \label{eq:var_transform_sp}
\end{align}
where (a) comes from Lemma \ref{lem:estimator_sphere_unbiased}.
Therefore, from \eqref{eq:var_g_F_sp_progress},
\begin{align}
\E_{S}[\|\bg_F^{{\rm sp}}(\bx,S) - \nabla F_{\mu,\maB}(\bx) \|^2] & \le \frac{2d}{N}\left(\frac{M^2\mu^2d}{4} + \frac{3d}{(d+2)} \|\nabla F(\bx)\|^2 \right). \label{eq:g_ga_F_variance_1_sp}
\end{align}

Here,
\begin{align}
& \E_{S, \{\Xi_{1}^{(j)}\}_{j=1}^m, \{\Xi_{2}^{(j)}\}_{j=1}^m} \left[ \left\| \bg^{\rm{sp}}_m(\bx,S, \{\Xi_{1}^{(j)}\}_{j=1}^m, \{\Xi_{2}^{(j)}\}_{j=1}^m)  -\bg_F^{\rm{sp}}(\bx,S) \right\|^2  \right] \nonumber \\
& = \E_{S, \{\Xi_{1}^{(j)}\}_{j=1}^m, \{\Xi_{2}^{(j)}\}_{j=1}^m} \bigg[ \bigg\|\frac{d}{N}\sum_{i=1}^{N}\frac{\frac{1}{m}\sum_{j=1}^{m} f(\bx+\mu \bs^i, \bmxi^{1,i,(j)}) - \frac{1}{m}\sum_{j=1}^{m} f(\bx-\mu \bs^i, \bmxi^{2,i,(j)})}{2\mu}\bs^i \nonumber\\
    &\hspace{40mm} - \frac{d}{N} \sum_{i=1}^{N} \frac{F(\bx+\mu \bs^i)-F(\bx - \mu \bs^i)}{2\mu}\bs^i \bigg\|^2  \bigg] \nonumber\\
& = \E_{S, \{\Xi_{1}^{(j)}\}_{j=1}^m, \{\Xi_{2}^{(j)}\}_{j=1}^m} \bigg[ \bigg\|\frac{d}{N}\sum_{i=1}^{N}\frac{\frac{1}{m}\sum_{j=1}^{m} f(\bx+\mu \bs^i, \bmxi^{1,i,(j)}) - F(\bx+\mu \bs^i)}{2\mu}\bs^i \nonumber\\
    &\hspace{40mm} - \frac{d}{N} \sum_{i=1}^{N} \frac{\frac{1}{m}\sum_{j=1}^{m} f(\bx-\mu \bs^i, \bmxi^{2,i,(j)})-F(\bx - \mu \bs^i)}{2\mu}\bs^i \bigg\|^2  \bigg] \nonumber\\
& \le 2\E_{S,\{\Xi_{1}^{(j)}\}_{j=1}^m, \{\Xi_{2}^{(j)}\}_{j=1}^m} \Bigg[ \left\| \frac{d}{N}\sum_{i=1}^{N}\frac{\frac{1}{m}\sum_{j=1}^{m} f(\bx+\mu \bs^i, \bmxi^{1,i,(j)}) - F(\bx+\mu \bs^i)}{2\mu}\bs^i\right\|^2 \nonumber\\
    &\hspace{40mm} + \left\|\frac{d}{N} \sum_{i=1}^{N} \frac{\frac{1}{m}\sum_{j=1}^{m} f(\bx-\mu \bs^i, \bmxi^{2,i,(j)})-F(\bx - \mu \bs^i)}{2\mu}\bs^i \right\|^2  \Bigg] \nonumber\\
& = \frac{d^2}{2\mu^2N^2} \E_{S} \Bigg[ \E_{\{\Xi_{1}^{(j)}\}_{j=1}^m} \Bigg[ \left\| \sum_{i=1}^{N}\left(\frac{1}{m}\sum_{j=1}^{m} f(\bx+\mu \bs^i, \bmxi^{1,i,(j)}) - F(\bx+\mu \bs^i)\right)\bs^i\right\|^2 \Bigg]  \nonumber\\
    &\hspace{20mm} + \E_{\{\Xi_{2}^{(j)}\}_{j=1}^m} \Bigg[ \left\|\sum_{i=1}^{N} \left(\frac{1}{m}\sum_{j=1}^{m} f(\bx-\mu \bs^i, \bmxi^{2,i,(j)}) - F(\bx-\mu \bs^i)\right)\bs^i \right\|^2  \Bigg]  \Bigg] \nonumber \\
& \overset{(a)}{=}\frac{d^2}{2\mu^2N^2} \E_{S} \Bigg[ \E_{\{\Xi_{1}^{(j)}\}_{j=1}^m} \Bigg[ \sum_{i=1}^{N} \left\| \left(\frac{1}{m}\sum_{j=1}^{m} f(\bx+\mu \bs^i, \bmxi^{1,i,(j)}) - F(\bx+\mu \bs^i)\right)\bs^i\right\|^2 \Bigg]  \nonumber\\
    &\hspace{20mm} + \E_{\{\Xi_{2}^{(j)}\}_{j=1}^m} \Bigg[ \sum_{i=1}^{N} \left\| \left(\frac{1}{m}\sum_{j=1}^{m} f(\bx-\mu \bs^i, \bmxi^{2,i,(j)}) - F(\bx-\mu \bs^i)\right)\bs^i \right\|^2  \Bigg]  \Bigg] \nonumber \\
& \overset{(b)}{\le} \frac{d^2}{2\mu^2N^2} \E_{S} \left[\sum_{i=1}^N \frac{\sigma^2}{m} \| \bs^i\|^2 + \sum_{i=1}^N \frac{\sigma^2}{m}\| \bs^i\|^2 \right] \nonumber \\
& = \frac{\sigma^2d^2}{\mu^2 Nm}, \label{eq:var_partial_sp}
\end{align}
where (a) is due to the fact that, for $i\neq k$, 
$$\E_{\bmxi^{1,i,(j)},\bmxi^{1,k,(\ell)}}\left[ \left(f(\bx+\mu \bs^i, \bmxi^{1,i,(j)}) - F(\bx+\mu \bs^i)\right)\left(f(\bx+\mu \bs^k, \bmxi^{1,k,(\ell)}) - F(\bx+\mu \bs^k)\right)\bs^{i\top} \bs^k\right]=0, \ {\rm and}$$
$$\E_{\bmxi^{2,i,(j)},\bmxi^{2,k,(\ell)}}\left[ \left(f(\bx-\mu \bs^i, \bmxi^{2,i,(j)}) - F(\bx-\mu \bs^i)\right)\left(f(\bx-\mu \bs^k, \bmxi^{2,k,(\ell)}) - F(\bx-\mu \bs^k)\right)\bs^{i\top} \bs^k\right]=0;$$
(b) is due to Assumption \ref{asm:F_sample_var_bound} and Lemma \ref{lem:minibatch_var_reduction}.

Then,
\begin{align*}
& \E_{S, \{\Xi_{1}^{(j)}\}_{j=1}^m, \{\Xi_{2}^{(j)}\}_{j=1}^m} \left[ \left\| \bg^{\rm{sp}}_m(\bx,S, \{\Xi_{1}^{(j)}\}_{j=1}^m, \{\Xi_{2}^{(j)}\}_{j=1}^m)  - \nabla F(\bx)\right\|^2  \right] \\
& =\E_{S, \{\Xi_{1}^{(j)}\}_{j=1}^m, \{\Xi_{2}^{(j)}\}_{j=1}^m} \Big[ \Big\| \bg^{\rm{sp}}_m(\bx,S, \{\Xi_{1}^{(j)}\}_{j=1}^m, \{\Xi_{2}^{(j)}\}_{j=1}^m)  - \nabla F(\bx) \\
& \hspace{40mm} -\bg_F^{\rm{sp}}(\bx,S) + \bg_F^{\rm{sp}}(\bx,S)  - \nabla F_{\mu,\maB}(\bx) + \nabla F_{\mu,\maB}(\bx)\Big\|^2  \Big] \\
& \le 3 \E_{S, \{\Xi_{1}^{(j)}\}_{j=1}^m, \{\Xi_{2}^{(j)}\}_{j=1}^m} \big[ \| \bg^{\rm{sp}}_m(\bx,S, \{\Xi_{1}^{(j)}\}_{j=1}^m, \{\Xi_{2}^{(j)}\}_{j=1}^m) -\bg_F^{\rm{sp}}(\bx,S) \|^2 + \|\nabla F_{\mu,\maB}(\bx)  - \nabla F(\bx)\|^2   \\
&\hspace{40mm}  + \| \bg_F^{\rm{sp}}(\bx,S)  - \nabla F_{\mu,\maB}(\bx)\|^2 \big] \\ 
& \overset{(a)}{\le} \frac{3\sigma^2d^2}{\mu^2 Nm}+ 3M^2 \mu^2 + \frac{3M^2\mu^2d^2}{2N} + \frac{18d^2}{N(d+2)} \|\nabla F(\bx)\|^2,
\end{align*}
\normalsize
where (a) follows from \eqref{eq:g_ga_F_variance_1_sp}, \eqref{eq:var_partial_sp}, and Lemma \ref{lem:sp_smoothed_error}.
\end{proof}

\begin{lemma}
\label{lem:g_variance_sp_batch1_H_smooth}
Suppose that Assumptions \ref{asm:F_sample_var_bound} and \ref{asp:F_hessian_lipschitz} hold.
Then, the following holds.
\begin{align}
&\E_{S,\{\Xi_{1}^{(j)}\}_{j=1}^m, \{\Xi_{2}^{(j)}\}_{j=1}^m} \left[ \left\| \bg^{\rm{sp}}_m(\bx,S, \{\Xi_{1}^{(j)}\}_{j=1}^m, \{\Xi_{2}^{(j)}\}_{j=1}^m)  - \nabla F(\bx) \right\|^2  \right] \nonumber \\
&\le\frac{3\sigma^2d^2}{\mu^2 Nm} + 3\mu^4 H^2 + \frac{H^2\mu^4d^2}{6N} + \frac{18d^2}{N(d+2)} \|\nabla F(\bx)\|^2. \label{eq:sp_vound_Hessian_smooth}
\end{align}
\end{lemma}
\begin{proof}
As in the derivation of \eqref{eq:covar_ap} in the proof of Lemma \ref{lem:g_variance_sp_batch1}, we obtain 
\begin{align*}
&{\rm Var}(\bg_F^{{\rm sp}}(\bx,S)) \\
& \preceq \frac{2d^2}{N}\E_{\bs \sim \maS}\bigg[ \bigg( \frac{F(\bx+\mu \bs)-F(\bx) - \mu \nabla F(\bx)^\top \bs}{2\mu} \nonumber \\
    &\hspace{30mm} - \frac{F(\bx-\mu \bs)-F(\bx) + \mu \nabla F(\bx)^\top \bs}{2\mu} \bigg)^2 \bs \bs^\top  + (\nabla F(\bx)^\top \bs)^2 \bs \bs^\top \bigg] \\
& =\frac{2d^2}{N}\E_{\bs \sim \maS}\bigg[ \bigg( \frac{F(\bx+\mu \bs)-F(\bx) - \mu \nabla F(\bx)^\top \bs - \frac{\mu^2}{2}\bs^\top \nabla^2F(\bx) \bs}{2\mu} \\
    &\hspace{30mm} - \frac{F(\bx-\mu \bs)-F(\bx) + \mu \nabla F(\bx)^\top \bs - \frac{\mu^2}{2}\bs^\top \nabla^2F(\bx) \bs}{2\mu} \bigg)^2 \bs \bs^\top  + (\nabla F(\bx)^\top \bs)^2 \bs \bs^\top \bigg].
\end{align*}
Then,
\begin{align*}
{\rm Var}(\bg_F^{{\rm sp}}(\bx,S)) 
& \overset{(a)}{\preceq}  \frac{2d^2}{N}\E_{\bs \sim \maS}\bigg[ \bigg( \frac{H\|\mu \bs\|^3}{12\mu} + \frac{H\|\mu \bs\|^3}{12\mu}\bigg)^2 \bs \bs^\top  + (\nabla F(\bx)^\top \bs)^2 \bs \bs^\top \bigg] \\
& = \frac{2d^2}{N}\E_{\bs \sim \maS}\bigg[ \bigg( \frac{H^2\mu^4 \|\bs\|^6}{36} \bigg) \bs \bs^\top  + (\nabla F(\bx)^\top \bs)^2 \bs \bs^\top \bigg] \\
& \overset{(b)}{=} \frac{2d^2}{N}\left(\frac{H^2\mu^4}{36d}I + \frac{\nabla F(\bx)^\top \nabla F(\bx)}{d(d+2)}  I + \frac{2}{d(d+2)} \nabla F(\bx) \nabla F(\bx)^\top\right)  \\
& \overset{(c)}{\preceq} \frac{2}{N}\left(\frac{H^2\mu^4d}{36} + \frac{3d}{d+2} \|\nabla F(\bx)\|^2\right) I,
\end{align*}
where (a) is due to Lemma \ref{lem:Hessian_Lipschitz_tr}, which follows from 
Assumption \ref{asp:F_hessian_lipschitz}, and (b) follows from Lemma \ref{lem:k_moment_ss_sp_noise};
(c) is due to the fact that $\nabla F(\bx) \nabla F(\bx)^\top \preceq \|\nabla F(\bx)\|^2I$ as confirmed in the proof of Lemma \ref{lem:g_variance_sp_batch1}.
Then, 
\begin{align}
\E_S \|\bg_F^{{\rm sp}}(\bx,S) - \nabla F_{\mu,\maB}(\bx) \|^2 
&\overset{(a)}{=}{\rm tr}({\rm Var}(\bg_F^{{\rm sp}}(\bx, S))) \le \frac{2d}{N}\left(\frac{H^2\mu^4d}{36} + \frac{3d}{d+2} \|\nabla F(\bx)\|^2\right), \label{eq:g_sp_F_variance}
\end{align}
where (a) holds as in the derivation of \eqref{eq:var_transform_sp} in the proof of Lemma \ref{lem:g_variance_sp_batch1},

Moreover, similar to the proof of Lemma \ref{lem:g_variance_sp_batch1},
\begin{align}
& \E_{S, \{\Xi_{1}^{(j)}\}_{j=1}^m, \{\Xi_{2}^{(j)}\}_{j=1}^m} \left[ \left\| \bg^{\rm{sp}}_m(\bx,S, \{\Xi_{1}^{(j)}\}_{j=1}^m, \{\Xi_{2}^{(j)}\}_{j=1}^m)  -\bg_F^{\rm{sp}}(\bx,S) \right\|^2  \right] \le \frac{\sigma^2d^2}{\mu^2 Nm}, \label{eq:var_partial_2_sp}
\end{align}

Then, 
\begin{align*}
& \E_{S, \{\Xi_{1}^{(j)}\}_{j=1}^m, \{\Xi_{2}^{(j)}\}_{j=1}^m} \left[ \left\| \bg^{\rm{sp}}_m(\bx,S, \{\Xi_{1}^{(j)}\}_{j=1}^m, \{\Xi_{2}^{(j)}\}_{j=1}^m)  - \nabla F(\bx)\right\|^2  \right] \\
& =\E_{S, \{\Xi_{1}^{(j)}\}_{j=1}^m, \{\Xi_{2}^{(j)}\}_{j=1}^m} [ \| \bg^{\rm{sp}}_m(\bx,S, \{\Xi_{1}^{(j)}\}_{j=1}^m, \{\Xi_{2}^{(j)}\}_{j=1}^m) - \nabla F(\bx) \\
&\hspace{40mm} + \bg_F^{\rm{sp}}(\bx,S) -\bg_F^{\rm{sp}}(\bx,S)  + \nabla F_{\mu,\maB}(\bx)  
 - \nabla F_{\mu,\maB}(\bx) 
\|^2  ] \\
& \le 3 \E_{S, \{\Xi_{1}^{(j)}\}_{j=1}^m, \{\Xi_{2}^{(j)}\}_{j=1}^m} \big[ \| \bg^{\rm{sp}}_m(\bx,S, \{\Xi_{1}^{(j)}\}_{j=1}^m, \{\Xi_{2}^{(j)}\}_{j=1}^m) -\bg_F^{\rm{sp}}(\bx,S) \|^2  + \|\nabla F_{\mu,\maB}(\bx)  - \nabla F(\bx)\|^2  \\
& \hspace{40mm} + \| \bg_F^{\rm{sp}}(\bx,S)  - \nabla F_{\mu,\maB}(\bx)\|^2 \big] \\ 
& \overset{(a)}{\le} \frac{3\sigma^2d^2}{\mu^2 Nm} + 3\mu^4 H^2 + \frac{H^2\mu^4d^2}{6N} + \frac{18d^2}{N(d+2)} \|\nabla F(\bx)\|^2,
\end{align*}
\normalsize
where (a) follows from \eqref{eq:g_sp_F_variance}, \eqref{eq:var_partial_2_sp}, and Lemma \ref{lem:sp_smoothed_error}.
\end{proof}

\paragraph{Intuitive interpretation of Lemmas~\ref{lem:g_variance_sp_batch1} and~\ref{lem:g_variance_sp_batch1_H_smooth}.}
The first term $\left(\frac{3\sigma^2d^2}{\mu^2 Nm}\right)$ in each upper bound is the error incurred when approximating $F(\bx)$ by a finite-sample average of $f(\bx,\bmxi)$.
As discussed in Section~\ref{subsec:basic_form}, this term is amplified as the smoothing parameter $\mu$ decreases.
The second term ($3M^2 \mu^2$ in Lemma~\ref{lem:g_variance_sp_batch1}  and $3\mu^4 H^2$ in Lemma \ref{lem:g_variance_sp_batch1_H_smooth}) measures the discrepancy between the true gradient ($\nabla F$) and the smoothed gradient ($\nabla F_{\mu,\maB}$).\footnote{
As shown in Lemma \ref{lem:estimator_sphere_unbiased}, $\E[\bg^{\rm{sp}}_m(\bx, S,\{\Xi_1^{(j)}\}_{j=1}^m, \{\Xi_2^{(j)}\}_{j=1}^m)]=\nabla F_{\mu,\maB}(\bx)$, so $\bg^{\rm{sp}}_m(\bx, S,\{\Xi_1^{(j)}\}_{j=1}^m, \{\Xi_2^{(j)}\}_{j=1}^m)$ does not directly estimate $\nabla F(\bx)$ but rather the smoothed gradient $\nabla F_{\mu,\maB}(\bx)$; this is the source of the error term.}
Because $F_{\mu,\maB}(\bx)\to F(\bx)$ as $\mu\to 0$ by the definition, this error diminishes for smaller $\mu$.
Thus, as with the gradient estimator using coordinate-wise directions, there is a trade-off in the smoothing parameter $\mu$ between the first and second terms.
The third and fourth terms are caused by the variance arising from the use of 
random directions $\{s^i\}_{i=1}^N$.

\paragraph{Difference from the gradient estimator with coordinate-wise directions.}
Lemmas~\ref{lem:g_variance_sp_batch1} and~\ref{lem:g_variance_sp_batch1_H_smooth} differ from Lemmas~\ref{lem:g_variance_FD} and~\ref{lem:g_variance_FD_H_smooth} in two respects:
(i) the second term (the bias from smoothing) loses its dependence on the dimension $d$; and
(ii) the third and fourth terms (the variance from random directions) arise additionally.
From a sample complexity perspective, this trade-off is favorable: (i) because the second-term bias is smaller, we need not choose $\mu$ excessively small to control the bias, thereby avoiding amplification of the error of the first term; 
(ii) when the number of random directions $N$ is taken large enough to reduce the first term, the added third and fourth terms become negligible due to $N^{-1}$.
As a result, the gradient estimator with directions drawn uniformly from the unit sphere achieves smaller sample complexity than that of the gradient estimator with coordinate-wise directions (Theorem \ref{thm:sp_sample_complexity} as shown later).

\paragraph{Dominance of $N$ over $m$.}
The gradient estimator \eqref{eq:mini_batch_gradient_est_sp} has two tunable parameters: $N$, the number of random directions, and $m$, the number of noisy function-value observations at each evaluation point.
Increasing either parameter raises the sample complexity by the same amount; however, increasing $N$ yields a tighter upper bound on the distance between the gradient estimator and the true gradient. 
Intuitively, increasing $N$ not only reduces the error in random directions $\bs^i$ but also reduces the approximation error of $F$ due to finite samples by resampling $\bmxi$ in each direction.
By contrast, increasing $m$ reduces only the approximation error of $F$ due to finite samples.

Then, we show the sample complexity for Algorithm \ref{alg:simple} with $\bg^{\rm{sp}}_m$.
\begin{theorem} \label{thm:sp_sample_complexity}
Suppose that $\epsilon\le \frac{1}{3}$ and Assumptions \ref{asm:F_sample_var_bound} and \ref{asm:F_smooth} hold.
Let $\bg_t:=\bg^{\rm{sp}}_m(\bx_t, S,\{\Xi_1^{(j)}\}_{j=1}^m, \{\Xi_2^{(j)}\}_{j=1}^m)$, where $\bg^{\rm{sp}}_m$ is defined by 
\eqref{eq:mini_batch_gradient_est_sp}.
Let $\mu:=\Theta(\epsilon)$, $N:=d^2 \epsilon^{-4}$, $\eta\le \frac{1}{4M}$, $m:=1$, and $T:=\Theta(\epsilon^{-2})$.
Then, the sample complexity, to satisfy $\E[\|\nabla F(\bar{\bx}_T)\|^2]\le \epsilon^2$ for output $\bar{\bx}_T$ of Algorithm \ref{alg:simple}, is $O(d^2\epsilon^{-6})$.
\end{theorem}
\begin{proof}
We have
\small
\begin{align*}
&\E[\|\nabla F(\bar{\bx}_T)\|^2] \\
& \overset{(a)}{\le} \frac{4(F(\bx_0) - F^*)}{\eta (T+1)}  + \frac{3}{T+1} \sum_{t=0}^T\E[\|\nabla F(\bx_t)-\bg_t\|^2] \\
& \overset{(b)}{\le} 
\frac{4(F(\bx_0) - F^*)}{\eta (T+1)}  + 
\frac{9\sigma^2d^2}{\mu^2 Nm} + 9M^2 \mu^2 + \frac{9M^2\mu^2d^2}{2N} + \frac{54d^2}{N(d+2)(T+1)} \sum_{t=0}^T \E [\|\nabla F(\bx_t)\|^2]\\
&\overset{(c)}{\le} O(\epsilon^2) + 54(d+2)^{-1} \epsilon^4 \frac{1}{T+1} \sum_{t=0}^T\E [\|\nabla F(\bx_t)\|^2]\\
&\overset{(d)}{\le} O(\epsilon^2) + \frac{2}{9} \frac{1}{T+1} \sum_{t=0}^T\E [\|\nabla F(\bx_t)\|^2],
\end{align*}
\normalsize
where (a) comes from Lemma \ref{lem:simple_descent}; 
(b) follows from Lemma \ref{lem:g_variance_sp_batch1};
(c) holds since $T=\Theta(\epsilon^{-2})$, $\mu:=\Theta(\epsilon)$, $N=d^2\epsilon^{-4}$, and $m=1$;
(d) is due to $\epsilon\le \frac{1}{3}$ and $d\ge 1$.

Since $\frac{1}{T+1} \sum_{t=0}^T\E [\|\nabla F(\bx_t)\|^2]=\E[\|\nabla F(\bar{\bx}_T)\|^2]$,
\begin{align*}
&\left(1- \frac{2}{9}\right)\E[\|\nabla F(\bar{\bx}_T)\|^2] = O(\epsilon^2). 
\end{align*}
Therefore,
\begin{align*}    
&\E[\|\nabla F(\bar{\bx}_T)\|^2] = \frac{9}{7} O(\epsilon^2)=O(\epsilon^2).
\end{align*}
Then, the sample complexity is $O(NmT)=O(\epsilon^{-6}d^2)$.
\end{proof}

\begin{theorem} \label{thm:sample_complexity_sp_Hessian_lip}
Suppose that $\epsilon\le \frac{1}{3}$ and Assumptions \ref{asm:F_sample_var_bound}--\ref{asp:F_hessian_lipschitz} hold.
Let $\bg_t:=\bg^{\rm{sp}}_m(\bx_t, S,\{\Xi_1^{(j)}\}_{j=1}^m, \{\Xi_2^{(j)}\}_{j=1}^m)$, where $\bg^{\rm{sp}}_m$ is defined by 
\eqref{eq:mini_batch_gradient_est_sp}.
Let $\mu:=\Theta(\epsilon^{\frac{1}{2}} )$, $N:=\epsilon^{-3}d^2$, $m=1$, $\eta\le \frac{1}{4M}$, and $T:=\Theta(\epsilon^{-2})$.
Then, the sample complexity, to satisfy $\E[\|\nabla F(\bar{\bx}_T)\|^2]\le \epsilon^2$ for output $\bar{\bx}_T$ of Algorithm \ref{alg:simple}, is $O(NmT)=O(d^2\epsilon^{-5})$.
\end{theorem}
\begin{proof}
We have
\small
\begin{align*}
\E[\|\nabla F(\bar{\bx}_T)\|^2] 
& \overset{(a)}{\le} \frac{4(F(\bx_0) - F^*)}{\eta (T+1)}  + \frac{3}{T+1} \sum_{t=0}^T\E[\|\nabla F(\bx_t)-\bg_t\|^2] \\
& \overset{(b)}{\le} 
\frac{4(F(\bx_0) - F^*)}{\eta (T+1)} +
\frac{9\sigma^2d^2}{\mu^2 Nm} + 9\mu^4 H^2 + \frac{H^2\mu^4d^2}{2N} + \frac{54d^2}{N(d+2)} \frac{1}{T+1} \sum_{t=0}^T\E [\|\nabla F(\bx_t)\|^2] \\
&\overset{(c)}{=} O(\epsilon^2) + 54 \epsilon^3(d+2)^{-1} \frac{1}{T+1} \sum_{t=0}^T\E [\|\nabla F(\bx_t)\|^2]\\
&\overset{(d)}{\le} O(\epsilon^2) + \frac{2}{3}  \frac{1}{T+1} \sum_{t=0}^T\E [\|\nabla F(\bx_t)\|^2],
\end{align*}
\normalsize
where (a) comes from Lemma \ref{lem:simple_descent}; 
(b) follows from Lemma \ref{lem:g_variance_sp_batch1_H_smooth};
(c) holds since $T=\Theta(\epsilon^{-2})$, $\mu:=\Theta(\epsilon^{\frac{1}{2}})$, $N=d^2\epsilon^{-3}$, and $m=1$;
(d) is due to $d\ge 1$ and $\epsilon \le \frac{1}{3}$.
Since $\E[\|\nabla F(\bar{\bx}_T)\|^2]=\frac{1}{T+1} \sum_{t=0}^T\E [\|\nabla F(\bx_t)\|^2]$,
\begin{align*}
&\left(1-\frac{2}{3}\right)\E[\|\nabla F(\bar{\bx}_T)\|^2] = O(\epsilon^2). 
\end{align*}
Therefore,
\begin{align*}    
&\E[\|\nabla F(\bar{\bx}_T)\|^2] = 3 O(\epsilon^2)=O(\epsilon^2).
\end{align*}
Therefore, the sample complexity is $O(NmT)=O(\epsilon^{-5}d^2)$.
\end{proof}

\subsubsection{Gaussian-smoothed gradient estimator} \label{subsec:gaussian_estimator}
The Gaussian-smoothed gradient estimator can be written as follows: 
\begin{align}
\bg^{\rm{ga}}(\bx, U, \Xi_1, \Xi_2):= &\frac{1}{N} \sum_{i=1}^{N} \frac{f(\bx+\mu \bu^i,\bmxi^{1,i})-f(\bx - \mu \bu^i,\bmxi^{2,i})}{2 \mu} \bu^i. \label{eq:one_batch_gradient_est_ga}
\end{align}
Here, $\mu \in \R_{>0}$; 
$U:=\{\bu^i\}_{i=1}^N$, where $\bu^i \sim \mathcal N(0, I)$; 
$\Xi_1:=\{\bmxi^{1,i}\}_{i=1}^N$, where $\bmxi^{1,i} \sim D(\bx+\mu \bu^i)$;
$\Xi_2:=\{\bmxi^{2,i}\}_{i=1}^N$, where $\bmxi^{2,i} \sim D(\bx-\mu \bu^i)$.
We consider the mini-batch version of the gradient estimator:
\begin{align}
\bg^{\rm{ga}}_m(\bx, U,\{\Xi_1^{(j)}\}_{j=1}^m, \{\Xi_2^{(j)}\}_{j=1}^m): =\frac{1}{m}\sum_{j=1}^{m} \bg^{\rm{ga}}(\bx,U,\Xi_1^{(j)},\Xi_2^{(j)}). \label{eq:mini_batch_gradient_est_ga}
\end{align}
Here, we let
\begin{align}
&F_{\mu,\maN}(\bx):= \E_{\bu \sim \maN(0,I)}[F(\bx + \mu \bu)],\label{def:F_mu_ga}
\end{align}
which is called the Gaussian-smoothed function of $F$.
Then, the following lemma holds for the gradient estimator from \cite[(26)]{nesterov2017random}.
\begin{lemma} \cite[(26)]{nesterov2017random} \label{lem:estimator_gaussian_unbiased}
Let $\bg^{\rm{ga}}_m(\bx, U,\{\Xi_1^{(j)}\}_{j=1}^m, \{\Xi_2^{(j)}\}_{j=1}^m)$ be defined by \eqref{eq:mini_batch_gradient_est_ga} and 
$F_{\mu,\maN}(\bx)$ be defined by \eqref{def:F_mu_ga}.
Then, 
\begin{align*}
\E_{U, \{\Xi_{1}^{(j)}\}_{j=1}^m, \{\Xi_{2}^{(j)}\}_{j=1}^m}\left[\bg^{\rm{ga}}_m(\bx, U,\{\Xi_1^{(j)}\}_{j=1}^m, \{\Xi_2^{(j)}\}_{j=1}^m)\right]
&= \E_{\bu \sim \maN(0, \bI)}\left[\frac{F(\bx+\mu \bu)-F(\bx - \mu \bu)}{2\mu}\bu \right] \\
&=\nabla F_{\mu,\maN}(\bx).
\end{align*}
\end{lemma}
From Lemma \ref{lem:estimator_gaussian_unbiased}, $\bg^{\mathrm{ga}}_m$ serves as an estimator of the gradient of the Gaussian-smoothed function $F_{\mu,\maN}$.
Then, we can bound the distance between the true gradient and the gradient estimator by the following lemmas.
In the proof, we let
\begin{align*}  
&\bg_F^{{\rm ga}}(\bx, U):= \frac{1}{N} \sum_{i=1}^N  \frac{F(\bx+\mu \bu^i)-F(\bx - \mu \bu^i)}{2\mu}\bu^i.
\end{align*}

\begin{lemma}
\label{lem:g_variance_Ga_batch1}
Suppose that Assumptions \ref{asm:F_sample_var_bound} and \ref{asm:F_smooth} hold.
Then, the following holds.
\begin{align*}
&\E_{U,\{\Xi_{1}^{(j)}\}_{j=1}^m, \{\Xi_{2}^{(j)}\}_{j=1}^m} \left[ \left\| \bg^{\rm{ga}}_m(\bx,U, \{\Xi_{1}^{(j)}\}_{j=1}^m, \{\Xi_{2}^{(j)}\}_{j=1}^m)  - \nabla F(\bx) \right\|^2  \right] \\
&\le \frac{3\sigma^2d}{\mu^2 N m}+ 3\mu^2 M^2 d + 
\frac{3dM^2\mu^2}{2N}(d+2)(d+4)
+\frac{18d}{N} \|\nabla F(\bx)\|^2.
\end{align*}
\end{lemma}
\begin{proof}
We have 
\begin{align}
&{\rm Var}(\bg_F^{{\rm ga}}(\bx, U)) \nonumber\\
&= \frac{1}{N} {\rm Var}\left( \frac{F(\bx+\mu \bu)-F(\bx - \mu \bu)}{2\mu}\bu\right)
\nonumber \\
&= \frac{1}{N} \E_{\bu \sim \maN(0, I)}\left[\left(\frac{F(\bx+\mu \bu)-F(\bx - \mu \bu)}{2\mu}\bu\right) \left(\frac{F(\bx+\mu \bu)-F(\bx - \mu \bu)}{2\mu}\bu\right)^\top \right] 
\nonumber \\
& \quad - \frac{1}{N}\E_{\bu \sim \maN(0, I)}\left[\frac{F(\bx+\mu \bu)-F(\bx - \mu \bu)}{2\mu}\bu\right] \E_{\bu \sim \maN(0, I)}\left[\frac{F(\bx+\mu \bu)-F(\bx - \mu \bu)}{2\mu}\bu \right]^\top \nonumber \\
& \overset{(a)}{=} \frac{1}{N}\E_{\bu \sim \maN(0, I)}\left[ \left( \frac{F(\bx+\mu \bu)-F(\bx - \mu \bu)}{2\mu}\right)^2  \bu \bu^\top \right]- \frac{1}{N} \nabla F_{\mu,\maN}(\bx) \nabla F_{\mu,\maN}(\bx)^\top, \nonumber
\end{align}
where (a) comes from Lemma \ref{lem:estimator_gaussian_unbiased}.
Then, we have
\begin{align}
&{\rm Var}(\bg_F^{{\rm ga}}(\bx, U)) \nonumber \\
& \preceq  \frac{1}{N}\E_{\bu \sim \maN(0, I)}\left[ \left( \frac{F(\bx+\mu \bu)-F(\bx - \mu \bu)}{2\mu} \right)^2 \bu \bu^\top  \right] \nonumber \\
& =\frac{1}{N}\E_{\bu \sim \maN(0, I)}\left[ \left( \frac{F(\bx+\mu \bu)-F(\bx - \mu \bu) - 2\mu \nabla F(\bx)^\top \bu}{2\mu} + \nabla F(\bx)^\top \bu \right)^2 \bu \bu^\top  \right] \nonumber \\
& \preceq \frac{2}{N}\E_{\bu \sim \maN(0, I)}\left[ \left( \frac{F(\bx+\mu \bu)-F(\bx - \mu \bu) - 2\mu \nabla F(\bx)^\top \bu}{2\mu} \right)^2 \bu \bu^\top + (\nabla F(\bx)^\top \bu)^2 \bu \bu^\top \right] \nonumber \\
& =\frac{2}{N}\E_{\bu \sim \maN(0, I)}\bigg[ \bigg( \frac{F(\bx+\mu \bu)-F(\bx) - \mu \nabla F(\bx)^\top \bu}{2\mu} \nonumber \\
    &\hspace{30mm} - \frac{F(\bx-\mu \bu)-F(\bx) + \mu \nabla F(\bx)^\top \bu}{2\mu} \bigg)^2 \bu \bu^\top  + (\nabla F(\bx)^\top \bu)^2 \bu \bu^\top \bigg] \label{eq:co_var_progress} \\
&\overset{(a)}{\preceq}  \frac{2}{N}\E_{\bu \sim \maN(0, I)}\bigg[ \bigg( \frac{M\|\mu \bu\|^2}{4\mu} + \frac{M\|\mu \bu\|^2}{4\mu}\bigg)^2 \bu \bu^\top  + (\nabla F(\bx)^\top \bu)^2 \bu \bu^\top \bigg] \nonumber \\
& = \frac{2}{N}\E_{\bu \sim \maN(0, I)}\bigg[\frac{M^2\mu^2 \|\bu\|^4}{4}  \bu \bu^\top  + (\nabla F(\bx)^\top \bu)^2 \bu \bu^\top \bigg] \nonumber \\
& \overset{(b)}{=} \frac{2}{N}\left(\frac{M^2\mu^2}{4}(d+2)(d+4)I + \nabla F(\bx)^\top \nabla F(\bx) I  + 2 \nabla F(\bx) \nabla F(\bx)^\top \right) \nonumber \\
& \overset{(c)}{\preceq} \frac{2}{N}\left(\frac{M^2\mu^2}{4}(d+2)(d+4) + 3 \|\nabla F(\bx)\|^2\right) I, \label{eq:var_g_F_ga_progress}
\end{align}
where (a) is due to Lemma \ref{lem:first_order_taylor_remainder}, which follows from Assumption \ref{asm:F_smooth}; (b) follows from Lemma \ref{lem:k_moment_uu_gauss_noise};
(c) is due to the fact that $\nabla F(\bx) \nabla F(\bx)^\top \preceq \|\nabla F(\bx)\|^2I$ as confirmed in the proof of Lemma \ref{lem:g_variance_sp_batch1}.

Here, 
\begin{align}
\E_{U}[\|\bg_F^{{\rm ga}}(\bx,U) - \nabla F_{\mu,\maN}(\bx) \|^2] 
& = \E_{U}\left[\sum_{i=1}^d \big((\bg_F^{{\rm ga}}(\bx,U) - \nabla F_{\mu,\maN}(\bx))_i\big)^2\right] \nonumber \\
& = \E_{U}\left[ {\rm tr}\left( (\bg_F^{{\rm ga}}(\bx,U) - \nabla F_{\mu,\maN}(\bx))(\bg_F^{{\rm ga}}(\bx,U) - \nabla F_{\mu,\maN}(\bx))^\top\right)\right] \nonumber \\
& = {\rm tr}\left(\E_{U}\left[ (\bg_F^{{\rm ga}}(\bx,U) - \nabla F_{\mu,\maN}(\bx))(\bg_F^{{\rm ga}}(\bx,U) - \nabla F_{\mu,\maN}(\bx))^\top\right] \right)\nonumber \\
& \overset{(a)}{=} {\rm tr}({\rm Var}(\bg_F^{{\rm ga}}(\bx, U))), \label{eq:var_transform_ga}
\end{align}
where (a) comes from Lemma \ref{lem:estimator_gaussian_unbiased}.
Therefore, from \eqref{eq:var_g_F_ga_progress},
\begin{align}
\E_{U}[\|\bg_F^{{\rm ga}}(\bx,U) - \nabla F_{\mu,\maN}(\bx) \|^2] & \le \frac{2d}{N}\left(\frac{M^2\mu^2}{4}(d+2)(d+4) + 3 \|\nabla F(\bx)\|^2 \right).\label{eq:g_ga_F_variance_1}
\end{align}

Moreover,
\begin{align}
& \E_{U, \{\Xi_{1}^{(j)}\}_{j=1}^m, \{\Xi_{2}^{(j)}\}_{j=1}^m} \left[ \left\| \bg^{\rm{ga}}_m(\bx,U, \{\Xi_{1}^{(j)}\}_{j=1}^m, \{\Xi_{2}^{(j)}\}_{j=1}^m)  -\bg_F^{\rm{ga}}(\bx,U) \right\|^2  \right] \nonumber \\
& = \E_{U, \{\Xi_{1}^{(j)}\}_{j=1}^m, \{\Xi_{2}^{(j)}\}_{j=1}^m} \bigg[ \bigg\|\frac{1}{N}\sum_{i=1}^{N}\frac{\frac{1}{m}\sum_{j=1}^{m} f(\bx+\mu \bu^i, \bmxi^{1,i,(j)}) - \frac{1}{m}\sum_{j=1}^{m} f(\bx-\mu \bu^i, \bmxi^{2,i,(j)})}{2\mu}\bu^i \nonumber\\
    &\hspace{40mm} - \frac{1}{N} \sum_{i=1}^{N} \frac{F(\bx+\mu \bu^i)-F(\bx - \mu \bu^i)}{2\mu}\bu^i \bigg\|^2  \bigg] \nonumber\\
& = \E_{U, \{\Xi_{1}^{(j)}\}_{j=1}^m, \{\Xi_{2}^{(j)}\}_{j=1}^m} \bigg[ \bigg\|\frac{1}{N}\sum_{i=1}^{N}\frac{\frac{1}{m}\sum_{j=1}^{m} f(\bx+\mu \bu^i, \bmxi^{1,i,(j)}) - F(\bx+\mu \bu^i)}{2\mu}\bu^i \nonumber\\
    &\hspace{40mm} - \frac{1}{N} \sum_{i=1}^{N} \frac{\frac{1}{m}\sum_{j=1}^{m} f(\bx-\mu \bu^i, \bmxi^{2,i,(j)})-F(\bx - \mu \bu^i)}{2\mu}\bu^i \bigg\|^2  \bigg] \nonumber\\
& \le 2\E_{U,\{\Xi_{1}^{(j)}\}_{j=1}^m, \{\Xi_{2}^{(j)}\}_{j=1}^m} \Bigg[ \left\| \frac{1}{N}\sum_{i=1}^{N}\frac{\frac{1}{m}\sum_{j=1}^{m} f(\bx+\mu \bu^i, \bmxi^{1,i,(j)}) - F(\bx+\mu \bu^i)}{2\mu}\bu^i\right\|^2 \nonumber\\
    &\hspace{40mm} + \left\|\frac{1}{N} \sum_{i=1}^{N} \frac{\frac{1}{m}\sum_{j=1}^{m} f(\bx-\mu \bu^i, \bmxi^{2,i,(j)})-F(\bx - \mu \bu^i)}{2\mu}\bu^i \right\|^2  \Bigg] \nonumber\\
& = \frac{1}{2\mu^2N^2} \E_{U} \Bigg[ \E_{\{\Xi_{1}^{(j)}\}_{j=1}^m} \Bigg[ \Bigg\| \sum_{i=1}^{N}\left(\frac{1}{m}\sum_{j=1}^{m} f(\bx+\mu \bu^i, \bmxi^{1,i,(j)}) - F(\bx+\mu \bu^i)\right)\bu^i\Bigg\|^2 \Bigg]  \nonumber\\
    &\hspace{20mm} + \E_{\{\Xi_{2}^{(j)}\}_{j=1}^m} \Bigg[ \Bigg\|\sum_{i=1}^{N} \left(\frac{1}{m}\sum_{j=1}^{m} f(\bx-\mu \bu^i, \bmxi^{2,i,(j)}) - F(\bx-\mu \bu^i)\right)\bu^i \Bigg\|^2  \Bigg]  \Bigg]. \nonumber    
\end{align}
Here, for $i\neq k$, 
$$\E_{\bmxi^{1,i,(j)},\bmxi^{1,k,(\ell)}}\left[ \left(f(\bx+\mu \bu^i, \bmxi^{1,i,(j)}) - F(\bx+\mu \bu^i)\right)\left(f(\bx+\mu \bu^k, \bmxi^{1,k,(\ell)}) - F(\bx+\mu \bu^k)\right)\bu^{i\top} \bu^k\right]=0, \ {\rm and}$$
$$\E_{\bmxi^{2,i,(j)},\bmxi^{2,k,(\ell)}}\left[ \left(f(\bx-\mu \bu^i, \bmxi^{2,i,(j)}) - F(\bx-\mu \bu^i)\right)\left(f(\bx-\mu \bu^k, \bmxi^{2,k,(\ell)}) - F(\bx-\mu \bu^k)\right)\bu^{i\top} \bu^k\right]=0.$$
Therefore,
\begin{align}
& \E_{U, \{\Xi_{1}^{(j)}\}_{j=1}^m, \{\Xi_{2}^{(j)}\}_{j=1}^m} \left[ \left\| \bg^{\rm{ga}}_m(\bx,U, \{\Xi_{1}^{(j)}\}_{j=1}^m, \{\Xi_{2}^{(j)}\}_{j=1}^m)  -\bg_F^{\rm{ga}}(\bx,U) \right\|^2  \right] \nonumber \\
&\le \frac{1}{2\mu^2N^2} \E_{U} \Bigg[ \E_{\{\Xi_{1}^{(j)}\}_{j=1}^m} \Bigg[ \sum_{i=1}^{N} \left\| \left(\frac{1}{m}\sum_{j=1}^{m} f(\bx+\mu \bu^i, \bmxi^{1,i,(j)}) - F(\bx+\mu \bu^i)\right)\bu^i\right\|^2 \Bigg]  \nonumber\\
    &\hspace{20mm} + \E_{\{\Xi_{2}^{(j)}\}_{j=1}^m} \Bigg[ \sum_{i=1}^{N} \left\| \left(\frac{1}{m}\sum_{j=1}^{m} f(\bx-\mu \bu^i, \bmxi^{2,i,(j)}) - F(\bx-\mu \bu^i)\right)\bu^i \right\|^2  \Bigg]  \Bigg] \nonumber \\
& \overset{(a)}{\le} \frac{1}{2\mu^2N^2} \E_{U} \left[\sum_{i=1}^N \frac{\sigma^2}{m} \| \bu^i\|^2 + \sum_{i=1}^N \frac{\sigma^2}{m}\| \bu^i\|^2 \right] \nonumber \\
& \overset{(b)}{=} \frac{\sigma^2d}{\mu^2 Nm}, \label{eq:var_partial}
\end{align}
where (a) is due to Assumption \ref{asm:F_sample_var_bound} and Lemma \ref{lem:minibatch_var_reduction};
(b) comes from Lemma \ref{lem:moment_gauss_noise}.

Then,
\begin{align*}
& \E_{U, \{\Xi_{1}^{(j)}\}_{j=1}^m, \{\Xi_{2}^{(j)}\}_{j=1}^m} \left[ \left\| \bg^{\rm{ga}}_m(\bx,U, \{\Xi_{1}^{(j)}\}_{j=1}^m, \{\Xi_{2}^{(j)}\}_{j=1}^m)  - \nabla F(\bx)\right\|^2  \right] \\
& =\E_{U, \{\Xi_{1}^{(j)}\}_{j=1}^m, \{\Xi_{2}^{(j)}\}_{j=1}^m} [\| \bg^{\rm{ga}}_m(\bx,U, \{\Xi_{1}^{(j)}\}_{j=1}^m, \{\Xi_{2}^{(j)}\}_{j=1}^m) - \nabla F(\bx) \\
& \hspace{40mm} -\bg_F^{\rm{ga}}(\bx,U) + \bg_F^{\rm{ga}}(\bx,U)  - \nabla F_{\mu, \maN}(\bx) + \nabla F_{\mu, \maN}(\bx)  \|^2 ] \\
& \le 3 \E_{U, \{\Xi_{1}^{(j)}\}_{j=1}^m, \{\Xi_{2}^{(j)}\}_{j=1}^m} \big[ \| \bg^{\rm{ga}}_m(\bx,U, \{\Xi_{1}^{(j)}\}_{j=1}^m, \{\Xi_{2}^{(j)}\}_{j=1}^m) -\bg_F^{\rm{ga}}(\bx,U) \|^2 + \|\nabla F_{\mu, \maN}(\bx)  - \nabla F(\bx)\|^2  \\
&\hspace{40mm}  + \| \bg_F^{\rm{ga}}(\bx,U)  - \nabla F_{\mu, \maN}(\bx)\|^2 \big] \\ 
& \overset{(a)}{\le} \frac{3\sigma^2d}{\mu^2 Nm}+ 3\mu^2 M^2 d + \frac{6d}{N}\left(\frac{M^2\mu^2}{4}(d+2)(d+4) + 3 \|\nabla F(\bx)\|^2\right),
\end{align*}
\normalsize
where (a) follows from \eqref{eq:g_ga_F_variance_1}, \eqref{eq:var_partial}, and Lemma \ref{lem:ga_smoothed_error}.
\end{proof}

\begin{lemma}
\label{lem:g_variance_Ga_batch1_H_smooth}
Suppose that Assumptions \ref{asm:F_sample_var_bound} and \ref{asp:F_hessian_lipschitz} hold.
Then, the following holds.
\begin{align*}
&\E_{U,\{\Xi_{1}^{(j)}\}_{j=1}^m, \{\Xi_{2}^{(j)}\}_{j=1}^m} \left[ \left\| \bg^{\rm{ga}}_m(\bx,U, \{\Xi_{1}^{(j)}\}_{j=1}^m, \{\Xi_{2}^{(j)}\}_{j=1}^m)  - \nabla F(\bx) \right\|^2  \right] \\
&\le\frac{3\sigma^2d}{\mu^2 Nm}  + 3\mu^4 H^2 d^2 + \frac{H^2\mu^4}{6N}d(d+2)(d+4)(d+6)  + \frac{18d}{N} \|\nabla F(\bx)\|^2.
\end{align*}
\end{lemma}
\begin{proof}
As in the derivation of \eqref{eq:co_var_progress} in the proof of Lemma \ref{lem:g_variance_Ga_batch1}, we obtain 
\begin{align*}
&{\rm Var}(\bg_F^{{\rm ga}}(\bx, U)) \\
& \preceq \frac{2}{N}\E_{\bu \sim \maN(0, I)}\bigg[ \bigg( \frac{F(\bx+\mu \bu)-F(\bx) - \mu \nabla F(\bx)^\top \bu}{2\mu} \\
    &\hspace{30mm} - \frac{F(\bx-\mu \bu)-F(\bx) + \mu \nabla F(\bx)^\top \bu}{2\mu} \bigg)^2 \bu \bu^\top  + (\nabla F(\bx)^\top \bu)^2 \bu \bu^\top \bigg]\\
& =\frac{2}{N}\E_{\bu \sim \maN(0, I)}\bigg[ \bigg( \frac{F(\bx+\mu \bu)-F(\bx) - \mu \nabla F(\bx)^\top \bu - \frac{\mu^2}{2}\bu^\top \nabla^2F(\bx) \bu}{2\mu} \\
    &\hspace{30mm} - \frac{F(\bx-\mu \bu)-F(\bx) + \mu \nabla F(\bx)^\top \bu - \frac{\mu^2}{2}\bu^\top \nabla^2F(\bx) \bu}{2\mu} \bigg)^2 \bu \bu^\top  + (\nabla F(\bx)^\top \bu)^2 \bu \bu^\top \bigg].
\end{align*}
Then,
\begin{align*}
{\rm Var}(\bg_F^{{\rm ga}}(\bx, U)) 
& \overset{(a)}{\preceq}  \frac{2}{N}\E_{\bu \sim \maN(0, I)}\bigg[ \bigg( \frac{H\|\mu \bu\|^3}{12\mu} + \frac{H\|\mu \bu\|^3}{12\mu}\bigg)^2 \bu \bu^\top  + (\nabla F(\bx)^\top \bu)^2 \bu \bu^\top \bigg] \\
& = \frac{2}{N}\E_{\bu \sim \maN(0, I)}\bigg[ \bigg( \frac{H^2\mu^4 \|\bu\|^6}{36} \bigg) \bu \bu^\top  + (\nabla F(\bx)^\top \bu)^2 \bu \bu^\top \bigg] \\
& \overset{(b)}{=} \frac{2}{N}\left(\frac{H^2\mu^4}{36}(d+2)(d+4)(d+6)I + \nabla F(\bx)^\top \nabla F(\bx) I + 2 \nabla F(\bx) \nabla F(\bx)^\top\right)  \\
& \overset{(c)}{\preceq} \frac{2}{N}\left(\frac{H^2\mu^4}{36}(d+2)(d+4)(d+6)  + 3 \|\nabla F(\bx)\|^2\right) I,
\end{align*}
where (a) is due to Lemma \ref{lem:Hessian_Lipschitz_tr}, which follows from 
Assumption \ref{asp:F_hessian_lipschitz}, and (b) follows from Lemma \ref{lem:k_moment_uu_gauss_noise};
(c) is due to the fact that $\nabla F(\bx) \nabla F(\bx)^\top \preceq \|\nabla F(\bx)\|^2I$ as confirmed in the proof of Lemma \ref{lem:g_variance_sp_batch1}.
Then, 
\begin{align}
\E_U[\|\bg_F^{{\rm ga}}(\bx,U) - \nabla F_{\mu, \maN}(\bx) \|^2]
&\overset{(a)}{=} {\rm tr}({\rm Var}(\bg_F^{{\rm ga}}(\bx, U))) 
\le \frac{2d}{N}\left(\frac{H^2\mu^4}{36}(d+2)(d+4)(d+6)  + 3 \|\nabla F(\bx)\|^2\right), \label{eq:g_ga_F_variance}
\end{align}
where (a) holds as confirmed in the proof of \eqref{eq:var_transform_ga} in Lemma \ref{lem:g_variance_Ga_batch1}.

Moreover, similar to \eqref{eq:var_partial} in the proof of Lemma \ref{lem:g_variance_Ga_batch1},
\begin{align}
& \E_{U, \{\Xi_{1}^{(j)}\}_{j=1}^m, \{\Xi_{2}^{(j)}\}_{j=1}^m} \left[ \left\| \bg^{\rm{ga}}_m(\bx,U, \{\Xi_{1}^{(j)}\}_{j=1}^m, \{\Xi_{2}^{(j)}\}_{j=1}^m)  -\bg_F^{\rm{ga}}(\bx,U) \right\|^2  \right] \le \frac{\sigma^2d}{\mu^2 Nm}, \label{eq:var_partial_2}
\end{align}

Then, 
\begin{align*}
& \E_{U, \{\Xi_{1}^{(j)}\}_{j=1}^m, \{\Xi_{2}^{(j)}\}_{j=1}^m} \left[ \left\| \bg^{\rm{ga}}_m(\bx,U, \{\Xi_{1}^{(j)}\}_{j=1}^m, \{\Xi_{2}^{(j)}\}_{j=1}^m)  - \nabla F(\bx)\right\|^2  \right] \\
& =\E_{U, \{\Xi_{1}^{(j)}\}_{j=1}^m, \{\Xi_{2}^{(j)}\}_{j=1}^m} \big[ \big\| \bg^{\rm{ga}}_m(\bx,U, \{\Xi_{1}^{(j)}\}_{j=1}^m, \{\Xi_{2}^{(j)}\}_{j=1}^m)  - \nabla F(\bx) \nonumber \\
&\hspace{40mm} -\bg_F^{\rm{ga}}(\bx,U) + \bg_F^{\rm{ga}}(\bx,U)  - \nabla F_{\mu, \maN}(\bx) + \nabla F_{\mu, \maN}(\bx)\big\|^2  \big] \\
& \le 3 \E_{U, \{\Xi_{1}^{(j)}\}_{j=1}^m, \{\Xi_{2}^{(j)}\}_{j=1}^m} \big[ \| \bg^{\rm{ga}}_m(\bx,U, \{\Xi_{1}^{(j)}\}_{j=1}^m, \{\Xi_{2}^{(j)}\}_{j=1}^m) -\bg_F^{\rm{ga}}(\bx,U) \|^2 + \|\nabla F_{\mu, \maN}(\bx)  - \nabla F(\bx)\|^2   \\
& \hspace{40mm} + \| \bg_F^{\rm{ga}}(\bx,U)  - \nabla F_{\mu, \maN}(\bx)\|^2\big] \\ 
& \overset{(a)}{\le} \frac{3\sigma^2d}{\mu^2 Nm}  + 3\mu^4 H^2 d^2 + \frac{H^2\mu^4}{6N}d(d+2)(d+4)(d+6)  + \frac{18d}{N} \|\nabla F(\bx)\|^2,
\end{align*}
\normalsize
where (a) follows from \eqref{eq:g_ga_F_variance}, \eqref{eq:var_partial_2}, and Lemma \ref{lem:ga_smoothed_error}.
\end{proof}

Then, we show the sample complexity for Algorithm \ref{alg:simple} with $\bg^{\rm{ga}}_m$.
\begin{theorem} \label{thm:ga_convergence_grad_lip}
Suppose that $\epsilon\le \frac{1}{3}$, and Assumptions \ref{asm:F_sample_var_bound} and \ref{asm:F_smooth} hold.
Let $\bg_t:=\bg^{\rm{ga}}_m(\bx_t, U,\{\Xi_1^{(j)}\}_{j=1}^m, \{\Xi_2^{(j)}\}_{j=1}^m)$, where $\bg^{\rm{ga}}_m$ is defined by 
\eqref{eq:mini_batch_gradient_est_ga}.
Let $\mu:=\Theta(\epsilon d^{-\frac{1}{2}})$, $N:=d^2 \epsilon^{-4}$, $\eta\le \frac{1}{4M}$, $m:=1$, and $T:=\Theta(\epsilon^{-2})$.
Then, the sample complexity, to satisfy $\E[\|\nabla F(\bar{\bx}_T)\|^2]\le \epsilon^2$ for output $\bar{\bx}_T$ of Algorithm \ref{alg:simple}, is $O(d^2\epsilon^{-6})$.
\end{theorem}
\begin{proof}
We have
\small
\begin{align*}
&\E[\|\nabla F(\bar{\bx}_T)\|^2] \\
& \overset{(a)}{\le} \frac{4(F(\bx_0) - F^*)}{\eta (T+1)}  + \frac{3}{T+1} \sum_{t=0}^T\E[\|\nabla F(\bx_t)-\bg_t\|^2] \\
& \overset{(b)}{\le} \frac{4(F(\bx_0) - F^*)}{\eta (T+1)}  + \frac{9\sigma^2d}{\mu^2 N m}+ 9\mu^2 M^2 d + 
\frac{9dM^2\mu^2}{2N}(d+2)(d+4)
+\frac{54d}{N(T+1)} \sum_{t=0}^T \E [\|\nabla F(\bx_t)\|^2] \\
&\overset{(c)}{=} O(\epsilon^2) + 54\epsilon^4 d^{-1} \frac{1}{T+1} \sum_{t=0}^T\E [\|\nabla F(\bx_t)\|^2]\\
&\overset{(d)}{\le} O(\epsilon^2) + \frac{2}{3} \frac{1}{T+1} \sum_{t=0}^T\E [\|\nabla F(\bx_t)\|^2],
\end{align*}
\normalsize
where (a) comes from Lemma \ref{lem:simple_descent}; 
(b) follows from Lemma \ref{lem:g_variance_Ga_batch1}; 
(c) is due to the facts that $T=\Theta(\epsilon^{-2})$, $N=d^2 \epsilon^{-4}$, $\mu:=\Theta(\epsilon d^{-\frac{1}{2}})$, and $m=1$;
(d) is due to $d\ge 1$ and $\epsilon\le \frac{1}{3}$.

Since $\E[\|\nabla F(\bar{\bx}_T)\|^2]=\frac{1}{T+1} \sum_{t=0}^T\E [\|\nabla F(\bx_t)\|^2]$,
\begin{align*}
&\left(1- \frac{2}{3}\right)\E[\|\nabla F(\bar{\bx}_T)\|^2] = O(\epsilon^2). 
\end{align*}
Therefore,
\begin{align*}    
&\E[\|\nabla F(\bar{\bx}_T)\|^2] = 3 O(\epsilon^2)=O(\epsilon^2).
\end{align*}
Therefore, the sample complexity is $O(NmT)=O(d^2\epsilon^{-6})$.
\end{proof}

\begin{remark}
\label{rem:N} 
Theorem~\ref{thm:ga_convergence_grad_lip} achieves lower sample complexity than \citep{hikima2025zeroth,hikima2025guided}.
This improvement is due to two factors:
(i) whereas prior work analyzed the case $N=1$ (with $m>1$), we analyze $N>1$ (with $m=1$); by Lemmas~\ref{lem:g_variance_Ga_batch1} and~\ref{lem:g_variance_Ga_batch1_H_smooth}, the upper bounds on the error of the gradient estimator for the true gradient decrease more efficiently by increasing $N$ than by increasing $m$;
(ii) for the bias term $\|\nabla F_{\mu,\maN}(\bx)-\nabla F(\bx)\|^2$, prior analyses used the bound of \citep[Lemma~3]{nesterov2017random}, whereas we employ the tighter bound of Lemma \ref{lem:ga_smoothed_error} \citep[Eq.~(2.10) with $e_f=0$]{berahas2022theoretical}.\footnote{The bound was stated earlier in \citep{maggiar2018derivative}, without proof.}
\end{remark}

\begin{theorem} \label{thm:ga_convergence_Hess_lip}
Suppose that $\epsilon\le \frac{1}{4}$, Assumptions \ref{asm:F_sample_var_bound}--\ref{asp:F_hessian_lipschitz} hold.
Let $\bg_t:=\bg^{\rm{ga}}_m(\bx_t, U,\{\Xi_1^{(j)}\}_{j=1}^m, \{\Xi_2^{(j)}\}_{j=1}^m)$, where $\bg^{\rm{ga}}_m$ is defined by 
\eqref{eq:mini_batch_gradient_est_ga}.
Let $\mu:=\Theta(\epsilon^{\frac{1}{2}} d^{-\frac{1}{2}})$, $N:=d^2 \epsilon^{-3}$, $m=1$, $\eta\le \frac{1}{4M}$, and $T:=\Theta(\epsilon^{-2})$.
Then, the sample complexity, to satisfy $\E[\|\nabla F(\bar{\bx}_T)\|^2]\le \epsilon^2$ for output $\bar{\bx}_T$ of Algorithm \ref{alg:simple}, is $O(d^2\epsilon^{-5})$.
\end{theorem}
\begin{proof}
We have
\small
\begin{align*}
&\E[\|\nabla F(\bar{\bx}_T)\|^2] \\
& \overset{(a)}{\le} \frac{4(F(\bx_0) - F^*)}{\eta (T+1)}  + \frac{3}{T+1} \sum_{t=0}^T\E[\|\nabla F(\bx_t)-\bg_t\|^2] \\
& \overset{(b)}{\le} \frac{4(F(\bx_0) - F^*)}{\eta (T+1)}  + \frac{9\sigma^2d}{\mu^2 Nm} + 9\mu^4 H^2 d^2 + \frac{H^2\mu^4}{2N}d(d+2)(d+4)(d+6)  + \frac{54d}{N(T+1)} \sum_{t=0}^T \E [\|\nabla F(\bx_t)\|^2]\\
&\overset{(c)}{=} O(\epsilon^2) + 54 \epsilon^3 d^{-1} \frac{1}{T+1} \sum_{t=0}^T\E [\|\nabla F(\bx_t)\|^2]\\
&\overset{(d)}{\le} O(\epsilon^2) + \frac{27}{32} \frac{1}{T+1} \sum_{t=0}^T\E [\|\nabla F(\bx_t)\|^2]
\end{align*}
\normalsize
where (a) comes from Lemma \ref{lem:simple_descent}; 
(b) follows from Lemma \ref{lem:g_variance_Ga_batch1_H_smooth}; 
(c) is due to the facts that $T=\Theta(\epsilon^{-2})$, $N=d^2 \epsilon^{-3}$, $\mu:=\Theta(\epsilon^{\frac{1}{2}} d^{-\frac{1}{2}})$, and $m=1$;
(d) is due to $d\ge 1$ and $\epsilon\le \frac{1}{4}$.
Since $\E[\|\nabla F(\bar{\bx}_T)\|^2]=\frac{1}{T+1} \sum_{t=0}^T\E [\|\nabla F(\bx_t)\|^2]$,
\begin{align*}
&\left(1- \frac{27}{32}\right)\E[\|\nabla F(\bar{\bx}_T)\|^2] = O(\epsilon^2). 
\end{align*}
Therefore,
\begin{align*}    
&\E[\|\nabla F(\bar{\bx}_T)\|^2] = \frac{32}{5} O(\epsilon^2)=O(\epsilon^2).
\end{align*}
Therefore, the sample complexity is $O(NmT)=O(d^2\epsilon^{-5})$.
\end{proof}

\begin{remark}
\label{rem:ga_sp_order_same}
Ignoring constant factors, the gradient estimators  $\bg_m^{\rm{sp}}$ and $\bg_m^{\rm{ga}}$ exhibit similar properties by scaling the smoothing parameter $\mu$, namely
$$
\mu_{\mathrm{ga}}^2=\frac{\mu_{\mathrm{sp}}^2}{d}\quad \left(\text{equivalently } \mu_{\mathrm{ga}}=\frac{\mu_{\mathrm{sp}}}{\sqrt d}\right).
$$
Indeed, comparing the upper bound in Lemma~\ref{lem:g_variance_sp_batch1} (about the gradient estimator via smoothing on a sphere) with that in Lemma~\ref{lem:g_variance_Ga_batch1} (about the Gaussian-smoothed gradient estimator), and substituting $\mu_{\mathrm{ga}}^2=\mu_{\mathrm{sp}}^2/d$, the error terms in the upper bounds match in order:

$$
\begin{aligned}
\text{(finite-sample variance for $F$)}\quad 
&\frac{3\sigma^2 d}{\mu_{\mathrm{ga}}^2 N m}
= \frac{3\sigma^2 d^2}{\mu_{\mathrm{sp}}^2 N m},\\[4pt]
\text{(smoothing bias)}\quad 
& 3 M^2 \mu_{\mathrm{ga}}^2 d
= 3 M^2 \mu_{\mathrm{sp}}^2,\\[4pt]
\text{(variance from randomness of directions)}\quad 
& \frac{3M^2 \mu_{\mathrm{ga}}^2}{2N} d(d+2)(d+4) + \frac{18d}{N}\|\nabla F(\bx)\|^2 \\
&\asymp \frac{3 M^2 \mu_{\mathrm{sp}}^2d^2}{2N}+ \frac{d^2}{N(d+2)}\|\nabla F(\bx)\|^2.
\end{aligned}
$$
This match also explains why, under the optimal settings $\mu=\Theta(\epsilon)$ (sphere) and $\mu=\Theta(\epsilon d^{-1/2})$ (Gaussian), both estimators yield the same sample-complexity orders.
\end{remark}

\subsection{Comparison of gradient estimators with coordinate-wise and random directions}
\label{subsec:comparison}
We summarize the bounds derived in Sections \ref{subsec:grad_est_coordinate}--\ref{subsec:gaussian_estimator}.
The sample complexities (to guarantee $\E\|\nabla F(\bar{\bx}_T)\|^2\le \epsilon^2$) are as follows.

$$
\begin{aligned}
&\textbf{Coordinate-wise:} 
&&\begin{cases}
O\!\left(d^{3}\,\epsilon^{-6}\right) & \text{(Gradient Lipschitz; Theorem \ref{thm:sample_complexity_co_grad_lip}}),\\[2pt]
O\!\left(d^{\frac{5}{2}}\,\epsilon^{-5}\right) & \text{(Gradient and Hessian Lipschitz; Theorem~\ref{thm:sample_complexity_co_Hess_lip})},
\end{cases}\\[6pt]
&\textbf{Uniform distribution on a sphere:} 
&&\begin{cases}
O\!\left(d^{2}\,\epsilon^{-6}\right) & \text{(Gradient Lipschitz; Theorem~\ref{thm:sp_sample_complexity})},\\[2pt]
O\!\left(d^{2}\,\epsilon^{-5}\right) & \text{(Gradient and Hessian Lipschitz; Theorem~\ref{thm:sample_complexity_sp_Hessian_lip})},
\end{cases}\\[6pt]
&\textbf{Gaussian distribution:} 
&&\begin{cases}
O\!\left(d^{2}\,\epsilon^{-6}\right) & \text{(Gradient Lipschitz; Theorem~\ref{thm:ga_convergence_grad_lip})},\\[2pt]
O\!\left(d^{2}\,\epsilon^{-5}\right) & \text{(Gradient and Hessian Lipschitz;  Theorem~\ref{thm:ga_convergence_Hess_lip})}.
\end{cases}
\end{aligned}
$$
Thus, under our analysis, the estimators with directions drawn from the unit sphere or a Gaussian distribution admit smaller worst-case sample-complexity upper bounds than those of the gradient estimator with coordinate-wise directions.
We also analyzed a stochastic variant of the coordinate-direction estimator in which each element of the direction follows a Rademacher distribution. 
However, this modification did not improve the convergence performance. See Section \ref{sec:rademcher} for details.

\paragraph{The reason why our results reverse the conclusion of \citep{berahas2022theoretical}.}
\citet{berahas2022theoretical}, considering bounded and possibly biased noise, conclude that the gradient estimator with coordinate-wise directions (or the gradient estimator by linear-interpolation) achieves higher accuracy than those with random directions. 
The difference from our conclusion arises from the assumption on the noise: whereas we consider unbiased stochastic noise, i.e., $\E_{\bmxi}[f(\bx,\bmxi)] = F(\bx)$, they assume biased noise. 
Under our assumption, when we reduce the error from randomness of directions (the third and fourth terms of upper bounds in Lemmas \ref{lem:g_variance_sp_batch1}--\ref{lem:g_variance_sp_batch1_H_smooth} and \ref{lem:g_variance_Ga_batch1}--\ref{lem:g_variance_Ga_batch1_H_smooth}) by increasing the number of random directions $N$, we can simultaneously reduce the error from approximating $F$ with finite samples (the first term of the upper bounds)  by independently resampling $\bmxi$ for each direction’s evaluations. 
In contrast, since \citet{berahas2022theoretical} assume biased noise, increasing the number of samples cannot reduce the error from the bias.

\section{Experiments} \label{sec:experiment}
Following \citep{hikima2025guided}, we conducted two experiments on applications of \emph{multiproduct pricing} \citep{hikima2025zeroth} and \emph{strategic classification} \citep{levanon2021strategic} to evaluate Algorithm \ref{alg:simple} with different gradient estimators.
All experiments were conducted on a computer with an AMD EPYC 7413 24-Core Processor, 503.6 GiB
of RAM, and Ubuntu 20.04.6 LTS.
The program code was implemented in Python 3.12.2.

\paragraph{Compared methods.} 
We implemented the following methods.
Detailed parameter settings can be found in Appendices \ref{sec:price_parameters} and 
\ref{app:sclassification_parameters}.
\\
\textbf{ZO-CO:} Algorithm~\ref{alg:simple} using the gradient estimator \eqref{eq:mini_batch_gradient_est}.
\\
\textbf{ZO-SPH:} Algorithm~\ref{alg:simple} using the gradient estimator \eqref{eq:mini_batch_gradient_est_sp}.
\\
\textbf{ZO-GA:} Algorithm~\ref{alg:simple} using the gradient estimator \eqref{eq:mini_batch_gradient_est_ga}.
\\
\textbf{ZO-OG:} a Zeroth-Order method with a One-point Gradient estimator.
It is analogous to the zeroth-order method used in the existing studies \citep{ray2022decision,liu2023time}.
\\
\textbf{ZO-OGVR:} a Zeroth-Order method using a One-point Gradient estimator with a Variance Reduction parameter \citep[Algorithm 1]{hikima2025zeroth}.
Note that this method relies on a stronger information model: it requires knowledge of $f$ and access to realizations of $\bmxi$ sampled at previous iterations, whereas our methods require only noisy function-value observations and do not require observing $\bmxi$ itself.

\subsection{Multiproduct pricing application} \label{sec:exp_mp}
We conducted experiments on the same problem as \citep{hikima2025zeroth}.
Specifically, we considered the following multi-product pricing problem:
\begin{align*}
    \min_{\bx \in \R^{30}} \; \E_{\bmxi \sim D(\bx)} \big[ f(\bx,\bmxi) \big],
\end{align*}
where $\bx \in \R^{30}$ and $\bmxi \in \R^{30}$ denote the price and demand vectors for $30$ products, respectively.
We define $f(\bx,\bmxi) \coloneqq -\,s(\bx,\bmxi) + c(\bmxi)$, where $s(\bx,\bmxi)$ and $c(\bmxi)$ represent the sales revenue and production cost, respectively.
Let $D(\bx)$ denote the (price-dependent) probability distribution of $\bmxi$.
These experiments were semi-synthetic: some parameters were set using real retail data from a supermarket service provider in Japan;\footnote{The dataset, ``New Product Sales Ranking'', is publicly available from KSP-SP Co., Ltd.; see \url{http://www.ksp-sp.com}. Accessed August 23, 2025.}
the remaining parameters were set synthetically.
Details of the parameterization are provided in Appendix~\ref{sec:pricing_setting}.

\subsubsection*{Setting and metric} 
We performed our experiments with the following settings.\\
\textbf{Initial point.}
For all methods, we set the initial point to $\bx_0:=0.5\cdot \bone$, where $\bone:=(1,\dots,1) \in \R^{30}$.
\\
\textbf{Metric.}
To evaluate the output $\hat{\bx}$, we computed $obj:= \frac{1}{10^3}\sum_{q=1}^{10^3} f(\hat{\bx},\bmxi^q(\hat{\bx}))$, where $\bmxi^q(\hat{\bx}) \sim D(\hat{\bx})$.\\
\textbf{Termination criteria.}
We terminated each method once it had taken 5000 samples from $D(\bx)$ for some $\bx$.

\subsubsection*{Experimental Results} \label{subsec:exre}
Table \ref{tab:real} shows the results of the experiments using real data from different weeks.
These results are consistent with our theoretical analysis.
First, the methods with lower theoretical sample complexity, \textbf{ZO-GA} and \textbf{ZO-SPH}, tend to achieve higher performance. 
Second, as stated in Remark~\ref{rem:ga_sp_order_same}, the performance of \textbf{ZO-GA} and \textbf{ZO-SPH} is nearly identical.
Furthermore, although \textbf{ZO-OGVR} relies on additional information, namely, observations of $\bmxi$ itself and knowledge of $f$, \textbf{ZO-SPH} and \textbf{ZO-GA} overall outperform \textbf{ZO-OGVR}.

Figure \ref{fig2} illustrates the objective value (\textit{obj}) obtained by each method against the cumulative number of samples from $D(\bx)$ in its optimization process. 
The figure implies that \textbf{ZO-GA} and \textbf{ZO-SPH} tend to reduce the objective value more sharply than \textbf{ZO-CO} and  \textbf{ZO-OG}.
This is due to the following reasons:
(i) As established by our theoretical analysis (Section \ref{subsec:prop_each_method}), \textbf{ZO-CO} exhibits a greater distance between the gradient estimator and the true gradient than \textbf{ZO-GA} and \textbf{ZO-SPH};
(ii) \textbf{ZO-OG} could not stably decrease the objective value because of the large variance of the one-point gradient estimators.\footnote{The large variance of the one-point (also called single-point) gradient estimator is also noted in \citep{chen2022improve}.}
Although \textbf{ZO-OGVR} reduces the objective value more sharply than \textbf{ZO-GA} and \textbf{ZO-SPH} in (c), \textbf{ZO-GA} and \textbf{ZO-SPH} tend to outperform \textbf{ZO-OGVR} overall.

\begin{table*}[t]
  \centering
  \caption{Experimental results for multiple product pricing in 20 randomly generated problem instances. 
The \emph{obj} (\emph{sd}) column represents the average (standard deviation) of the \emph{obj}.
The best value of the average \emph{obj} for each experiment is in bold.}
  \label{tab:real}
      \begin{tabular}{ccrrrrrrrrrrrrrrrrrr} \toprule
date
&\multicolumn{2}{c}{ZO-CO}&
\multicolumn{2}{c}{
ZO-SPH}&
\multicolumn{2}{c}{
ZO-GA}&
\multicolumn{2}{c}{
ZO-OG 
}&
\multicolumn{2}{c}{
ZO-OGVR
}
\\ 
\cmidrule(lr){2-3} \cmidrule(lr){4-5} \cmidrule(lr){6-7} \cmidrule(lr){8-9} \cmidrule(lr){10-11} \cmidrule(lr){12-13}
&\multicolumn{1}{c}{obj}&\multicolumn{1}{c}{sd}& 
\multicolumn{1}{c}{obj}&\multicolumn{1}{c}{sd}&
\multicolumn{1}{c}{obj}&\multicolumn{1}{c}{sd}&
\multicolumn{1}{c}{obj}&\multicolumn{1}{c}{sd}&
\multicolumn{1}{c}{obj}&\multicolumn{1}{c}{sd}&
 \\ 
\midrule
2/21--2/27
&-33.2	&0.9
&\textbf{-33.9}	&1.0
&-33.7	&0.8
&-30.4	&1.5
&-31.6	&1.4
\\
3/21--3/27
&-33.7	&1.2
&\textbf{-34.2}	&1.2
&-33.9	&1.1
&-31.1	&1.5
&-33.2	&1.2
\\ 
5/23--5/29
&-45.7	&2.1
&-50.1	&1.6
&\textbf{-50.2}	&1.8
&-29.4	&2.6
&-48.6	&1.4
\\
6/20--6/26
&-13.7	&1.2
&-13.9	&1.3
&\textbf{-14.0}	&1.6
&-13.1	&1.4
&-13.9	&1.2
\\
7/18--7/24
&-15.7	&1.0
&-15.2	&1.2
&-15.0	&1.3
&-14.7	&1.2
&\textbf{-16.0}	&1.2
\\
8/08--8/14
&-36.2	&1.2
&\textbf{-37.3}	&1.1
&-37.2	&1.1
&-32.0	&1.8
&-35.8	&1.4
\\ \bottomrule
\end{tabular}
\end{table*}

\begin{figure*}[t!] 
\centering
    \begin{tabular}{c}
        \begin{minipage}{0.33\hsize}
        \centering
        \includegraphics[keepaspectratio, scale=0.4, angle=0]
        {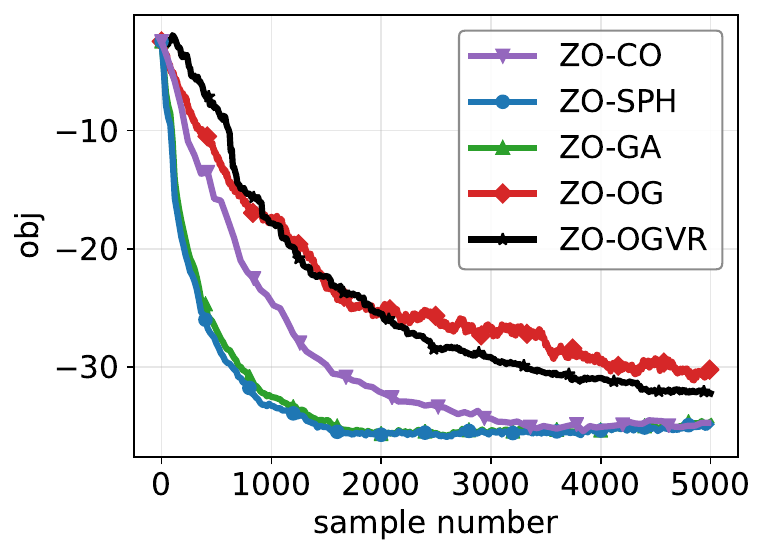}
        \subcaption{02/21--02/27}
        \label{n_ex}
        \end{minipage}
        
        \begin{minipage}{0.33\hsize}
        \centering
        \includegraphics[keepaspectratio, scale=0.4, angle=0]
        {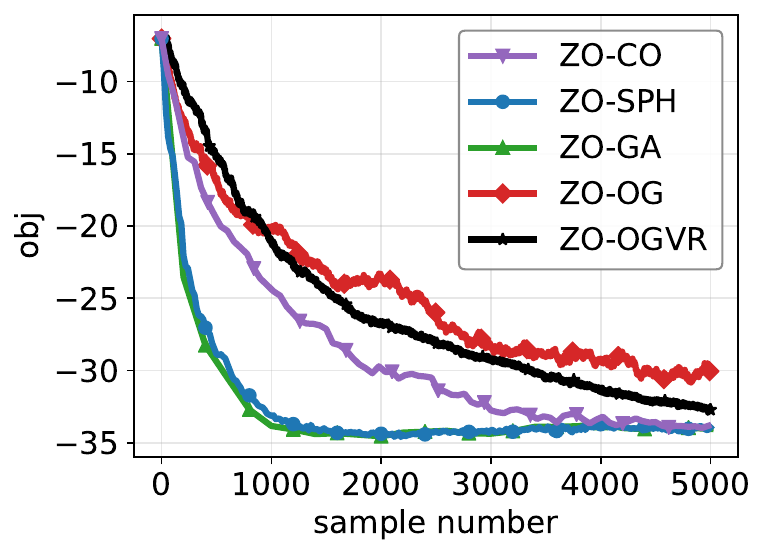}
        \subcaption{03/21--03/27}
        \label{m_ex}
        \end{minipage} 
        
        \begin{minipage}{0.33\hsize}
        \centering
        \includegraphics[keepaspectratio, scale=0.4, angle=0]
        {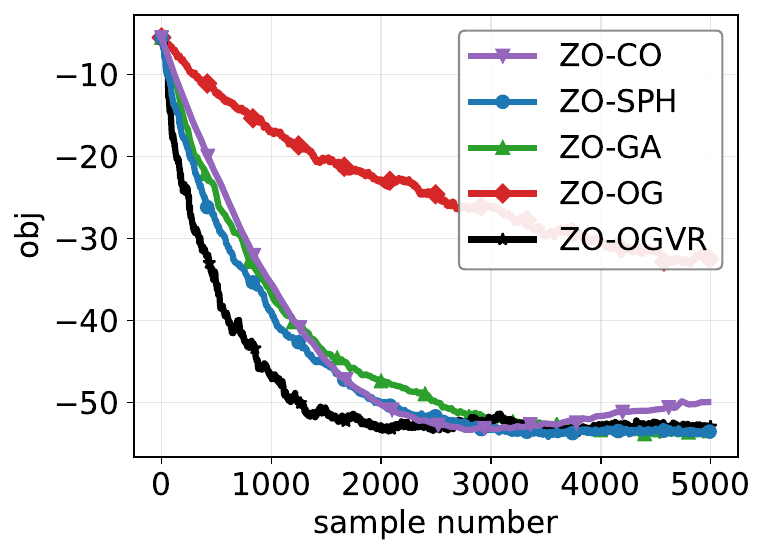}
        \subcaption{05/23--05/29}
        \end{minipage}\vspace{2mm}\\

        \begin{minipage}{0.33\hsize}
        \centering
        \includegraphics[keepaspectratio, scale=0.4, angle=0]
        {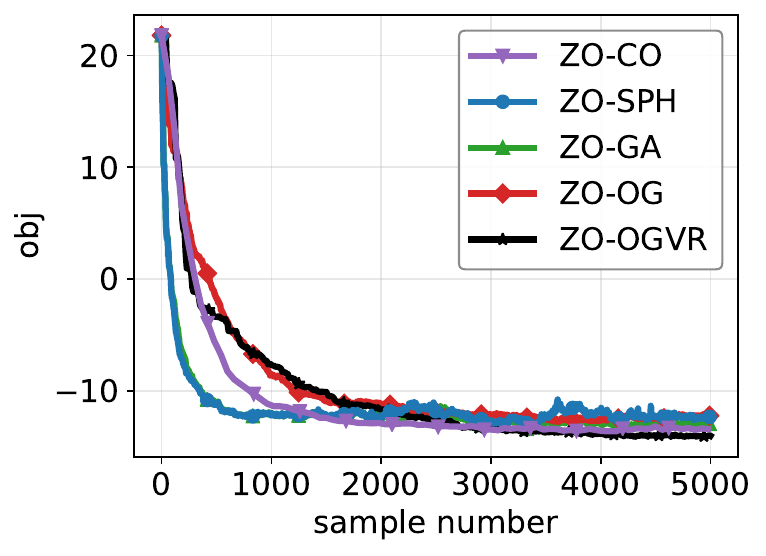}
        \subcaption{06/20--06/26}
        \end{minipage}
        
        \begin{minipage}{0.33\hsize}
        \centering        \includegraphics[keepaspectratio, scale=0.4, angle=0]
        {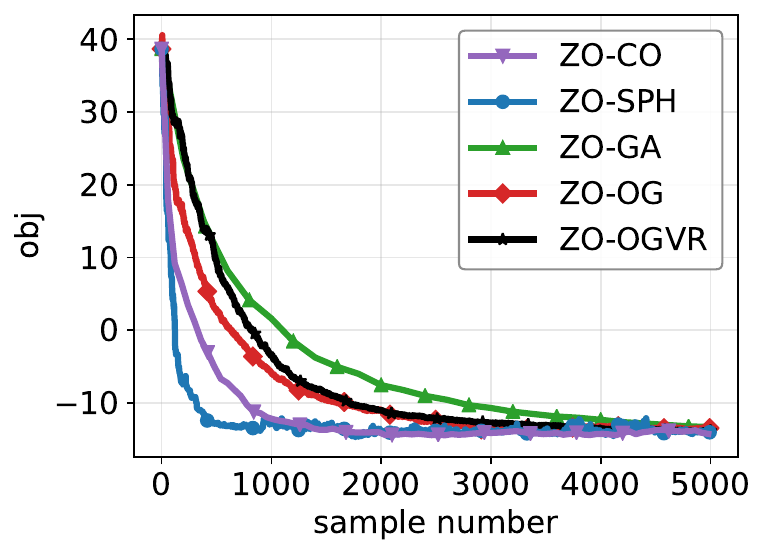}
        \subcaption{07/18--07/24}
        \end{minipage}
        
        \begin{minipage}{0.33\hsize}
        \centering
        \includegraphics[keepaspectratio, scale=0.4, angle=0]
        {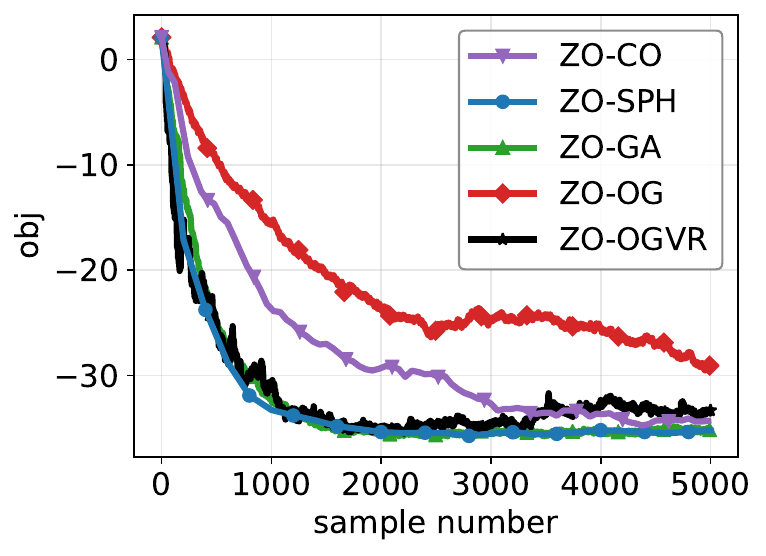}
        \subcaption{08/08--08/14}
        \end{minipage}
        
    \end{tabular}
\caption{
Change in \textit{obj} in the simulation experiment with real data.
Each graph shows the result in one problem instance for each week.
The horizontal axis indicates the number of samples, and the vertical axis indicates \textit{obj}.
}\label{fig2}
\end{figure*}

\subsection{Strategic Classification with Unknown Agents' Cost Functions}
Following prior work \citep{hikima2025guided,jagadeesan2021alternative}, we conducted experiments on the application of strategic classification with a real dataset from \citep{YEH20092473}.\footnote{As with \citep{levanon2021strategic}, we used a processed version of the data from \citep{ustun2019actionable}.} 
We consider a variant of the setting in \citep[Section~4]{levanon2021strategic}.
A decision maker selects a classifier parameter $\bx \in \R^{12}$ to determine loan approval or denial for an agent with observable features $\bmxi_F \in \R^{11}$ and a label $L \in \{0,1\}$.
The feature vector $\bmxi_F$ includes credit-card spending patterns, while $L$ indicates whether the agent defaults.
Because agents may modify features in response to the classifier parameterized by $\bx$, the joint distribution of $(\bmxi_F, L)$ depends on $\bx$.
We denote the distribution by $D(\bx)$.

Then, the loss-minimization problem is as follows:
\begin{align*}
    \min_{\bx\in \R^{12}} \quad \E_{(\bmxi_F, L) \sim D(\bx)} \left[ f(\bmxi_F, L ;\bx) \right].
\end{align*}
Here, $f$ is the loss function defined by
\begin{align*}
f(\bmxi_F,L;\bx) := &-L \log \left( \frac{1}{1+e^{-(\bx_{[11]}^\top \bmxi_F+x_{12})}}\right)- (1-L) \log \left(1-\frac{1}{1+e^{-(\bx_{[11]}^\top \bmxi_F+x_{12})}}\right),
\end{align*}
\normalsize
where $\bx_{[11]}:=(x_1, x_2, \dots, x_{11})$.

\paragraph{Modeling the unknown $D(\bx)$.}
Each agent possesses latent (``true'') features $\bmxi_{\mathrm{true}} \sim D_{\mathrm{true}}$ but may modify the observable features in response to the classifier parameter and an associated adaptation cost.
Following the noisy-response model \citep{jagadeesan2021alternative}, we assume that agents do not observe the deployed classifier parameter $\bx$ perfectly.
Instead, each agent observes a perturbed parameter
\begin{align}
  \tilde{\bx} := \bx + \bdelta,
  \qquad
  \bdelta \sim \mathcal N(0,\tau^2 I), \label{eq:noisy_observe}
\end{align}
where $\bdelta$ represents the agent's observation error.
Given the perceived classifier parameter $\tilde{\bx}$, the reported features are generated by the response map
\begin{align}
  \bmxi_F(\tilde{\bx},\bmxi_{\mathrm{true}})
  :=
  \arg\max_{\bmxi \in \R^{11}}
  \big\{
    r(\bmxi;\tilde{\bx}) - c(\bmxi,\bmxi_{\mathrm{true}})
  \big\}.
\label{eq:noisy_response}
\end{align}
Here, the payoff and cost are given by
\[
  r(\bmxi;\tilde{\bx}) :=
  \begin{cases}
    2, & \tilde{\bx}_{[11]}^\top \bmxi + \tilde{x}_{12} \ge 0,\\
    0, & \text{otherwise},
  \end{cases}
  \qquad
  c(\bmxi,\bmxi_{\mathrm{true}})
  :=
  \|\bmxi-\bmxi_{\mathrm{true}}\|_2^2 .
\]
The payoff $r(\bmxi;\tilde{\bx})$ captures the value of a positive decision (loan approval),\footnote{We set the reward to $2$, following \citep{levanon2021strategic}.}
while $c(\bmxi,\bmxi_{\mathrm{true}})$ encodes the adjustment cost.
We treat the real dataset as samples of $(\bmxi_{\mathrm{true}},L)$ and set $D(\bx)$ through the above noisy response mapping.

\begin{remark}
\label{remark:D_x_smooth_noisy_response}
Under the above setting, the objective function is twice continuously differentiable.
Although the response map $\bmxi_F(\tilde{\bx},\bmxi_{\mathrm{true}})$ is discontinuous in $\tilde{\bx}$, the perturbation $\bdelta$ introduced in $\tilde{\bx}=\bx+\bdelta$ smooths the objective function $F$.
Indeed, letting $\psi_{\tau}$ denote the probability density of $\mathcal N(0,\tau^2 I)$, we obtain
\begin{align*}
F(\bx)
&=
\E_{(\bmxi_{\mathrm{true}},L)}
\E_{\bdelta}
\left[
f\bigl(\bmxi_F(\bx+\bdelta,\bmxi_{\mathrm{true}}),L;\bx\bigr)
\right] =
\E_{(\bmxi_{\mathrm{true}},L)}
\left[
\int_{\R^{12}}
f\bigl(\bmxi_F(\bx+\bdelta,\bmxi_{\mathrm{true}}),L;\bx\bigr)
\psi_\tau(\bdelta)
\md \bdelta
\right] \\
&\overset{(a)}{=}\E_{(\bmxi_{\mathrm{true}},L)}
\left[
\int_{\R^{12}}
f\bigl(\bmxi_F(\tilde{\bx},\bmxi_{\mathrm{true}}),L;\bx\bigr)
\psi_\tau(\tilde{\bx}-\bx)
\md \tilde{\bx}
\right],
\end{align*}
where (a) follows from the change of variables
$\tilde{\bx}=\bx+\bdelta$.
In this representation, the decision variable $\bx$ appears through the loss function $f(\bmxi_F(\tilde{\bx},\bmxi_{\mathrm{true}}),L;\bx)$ and the probability density function $\psi_{\tau}(\tilde{\bx}-\bx)$ of the Gaussian distribution.
Since these functions are twice continuously differentiable with respect to $\bx$, it follows that $F(\bx)$ is twice continuously differentiable.
\end{remark}

\subsubsection*{Setting and metric} 
We performed our experiments under the following settings.\\
\textbf{Initial point.}
For all methods, we set the initial point to $\bx_0:=\bone$, where $\bone:=(1,\dots,1) \in \R^{12}$.
\\
\textbf{Metrics.} 
For the output of each method, we used the training loss, the test loss, the test AUC, and the test accuracy as metrics. \\
\textbf{Termination criteria.}
We terminated each method once it had taken $5000$ samples from $D(\bx)$ for some $\bx$.

\begin{table*}[t]
  \centering
  \caption{
  Results of 100 independent trials on a strategic-classification problem instance.
  The mean and standard deviation of each metric are reported.
  The best mean value for each metric across the compared methods is shown in bold.
  The symbol $\downarrow$ ($\uparrow$) indicates that lower (higher) values are better.
  }
\label{tab:real_2}
\begin{tabular}{ccccccccccc}
\toprule
metric
& \multicolumn{2}{c}{ZO-CO}
& \multicolumn{2}{c}{ZO-SPH}
& \multicolumn{2}{c}{ZO-GA}
& \multicolumn{2}{c}{ZO-OG}
& \multicolumn{2}{c}{ZO-OGVR}
\\
\cmidrule(lr){2-3}
\cmidrule(lr){4-5}
\cmidrule(lr){6-7}
\cmidrule(lr){8-9}
\cmidrule(lr){10-11}
& \multicolumn{1}{c}{mean}
& \multicolumn{1}{c}{sd}
& \multicolumn{1}{c}{mean}
& \multicolumn{1}{c}{sd}
& \multicolumn{1}{c}{mean}
& \multicolumn{1}{c}{sd}
& \multicolumn{1}{c}{mean}
& \multicolumn{1}{c}{sd}
& \multicolumn{1}{c}{mean}
& \multicolumn{1}{c}{sd}
\\
\midrule
train loss ($\downarrow$)
& 0.784 & 0.076
& \textbf{0.782} & 0.080
& 0.804 & 0.123
& 0.900 & 0.099
& 0.877 & 0.147
\\
test loss ($\downarrow$)
& 0.785 & 0.078
& \textbf{0.779} & 0.082
& 0.800 & 0.123
& 0.887 & 0.104
& 0.865 & 0.163
\\
test accuracy ($\uparrow$)
& \textbf{0.615} & 0.031
& 0.602 & 0.034
& 0.603 & 0.034
& 0.588 & 0.038
& 0.593 & 0.031
\\
test AUC ($\uparrow$)
& \textbf{0.660} & 0.035
& 0.646 & 0.041
& 0.644 & 0.042
& 0.631 & 0.045
& 0.635 & 0.038
\\
\bottomrule
\end{tabular}
\end{table*}

\subsubsection*{Experimental Results}
Table~\ref{tab:real_2} reports the experimental results across multiple metrics.\footnote{Note that the performance of each method in our experiments is lower than that in typical settings without strategic agents, because strategic agents manipulate their features to mislead the classifier.}
First, \textbf{ZO-CO}, \textbf{ZO-SPH}, and \textbf{ZO-GA} outperform \textbf{ZO-OG} and \textbf{ZO-OGVR}, suggesting that two-point gradient estimators are more effective than one-point estimators in this setting.
This observation is consistent with the fact that one-point estimators tend to have larger variance than two-point estimators.
On the other hand, \textbf{ZO-CO} performs comparably to \textbf{ZO-SPH} and better than \textbf{ZO-GA}, which is not fully aligned with the worst-case sample-complexity analysis.
This discrepancy may be due to the fact that the theoretical analysis considers worst-case guarantees, whereas the experiments are conducted with a limited sample budget and are affected by constant factors and hyperparameter choices.

\section{Conclusion}
In this paper, we conduct a unified sample complexity analysis across different gradient estimators, together with appropriate selections of the algorithmic parameters.
As a result, we show that gradient estimators that average over multiple directions ($N>1$), either uniformly from the unit sphere or from a Gaussian distribution, achieve the lowest sample complexity.
Moreover, our numerical experiments demonstrate the competitiveness of the multiple-random-direction estimators.

As a direction for future research, it would be interesting to incorporate more advanced acceleration techniques \citep{huang2022accelerated,chen2022improve} into our framework. For instance, the momentum scheme proposed by \citep{huang2022accelerated} achieves improved convergence rates for problem \eqref{eq:opt_problem_no_dependent}; extending such an approach to problem \eqref{eq:opt_problem} may further enhance the convergence rates of zeroth-order methods.
Moreover, it would be interesting to investigate lower bounds on the sample complexity of the zeroth-order methods for stochastic problems with decision-dependent noise.
Although lower bounds for stochastic optimization have been studied \citep{hu2016bandit,arjevani2023lower}, these existing results do not directly cover our setting. 
This is because \cite{hu2016bandit} focus on convex objective functions and \cite{arjevani2023lower} assume access to unbiased stochastic gradient oracles.
Establishing lower bounds for our problem would help determine whether our complexity bounds are tight.

\bibliographystyle{abbrvnat}
\bibliography{reference}   

\appendix

\section{Details of Our Experiments}
\label{app:detail_exp}

\subsection{Experiments on multiproduct pricing}
\subsubsection{Problem setting} \label{sec:pricing_setting}
We performed semi-synthetic experiments based on \citep{hikima2025zeroth}. 
A subset of model parameters was set according to retail records provided by a Japanese supermarket service vendor (details are in Section \ref{sec:data_detail}).
We consider a multi-product pricing environment with $n=30$ items and $m=120$ potential buyers.
Each buyer may purchase at most one unit of a single product. 
Let $\bx=(x_1,\ldots,x_{30})\in\R^{30}$ denote the price (decision) vector, and let $\bmxi\in\{0,1,\ldots,m\}^{31}$ represent random demand, where $\xi_i$ (for $i=1,\ldots,30$) counts the units sold of product $i$, and $\xi_{31}$ counts buyers who make no purchase.

The optimization problem is expressed as
\begin{align*}
    \min_{\bx\in\R^{30}}~\E_{\bmxi\sim D(\bx)}\!\left[\, f(\bx,\bmxi)\,\right],
\end{align*}
with $f(\bx,\bmxi):=-s(\bx,\bmxi)+c(\bmxi)$ and $D(\bx)$ the demand distribution induced by prices $\bx$.
Here, $s(\bx,\bmxi)$ denotes sales revenue and $c(\bmxi)$ the production cost:
\begin{align*}
s(\bx,\bmxi):=\sum_{i=1}^n x_i\,\xi_i,
\qquad
c(\bmxi):=\sum_{i=1}^n c_i(\xi_i),
\end{align*}
where $c_i$ is the production cost for each item and is defined by
\begin{align*}
c_i(\xi_i):=
\begin{cases}
2w_i\,\xi_i, & \xi_i\le l_i,\\
w_i(\xi_i-l_i)+2w_i l_i, & l_i<\xi_i\le u_i,\\
3w_i(\xi_i-u_i)+w_i(u_i-l_i)+2w_i l_i, & \xi_i>u_i.
\end{cases}
\end{align*}
We set $l_i:=\frac{0.5m}{n}$, $u_i:=\frac{1.5m}{n}$, and $w_i:=\rho_i\theta_i$, 
with $\rho_i\sim\mathrm{Unif}[0.25,0.5]$ and $\theta_i$ the normalized recorded average price of product $i$.
This setting of $c_i$ reflects the scenario where the production cost rate varies based on the number of units
sold.

\paragraph{Demand model $D(\bx)$ (unknown to the algorithms).}
Each buyer independently selects one item $i\in \{1,\ldots,n\}$ with probability
\[
p_i(\bx)=\frac{\exp\{\gamma_i(\theta_i-x_i)\}}{a_0+\sum_{j=1}^n \exp\{\gamma_j(\theta_j-x_j)\}},
\]
or opts out (no purchase) with probability
\[
p_0(\bx)=\frac{a_0}{a_0+\sum_{j=1}^n \exp\{\gamma_j(\theta_j-x_j)\}}.
\]
We take $\gamma_i:=\tfrac{2\pi}{\sqrt{6}\,\theta_i}$ and $a_0:=0.1n$.
Consequently,
\[
\Pr(\xi_i=k\mid\bx)
=
\binom{m}{k}
p_i(\bx)^k
\bigl(1-p_i(\bx)\bigr)^{m-k},
\qquad
k=0,\ldots,m,
\quad i=1,\ldots,n.
\]
Importantly, the methods evaluated do not exploit any oracle access to $D(\bx)$; the distribution is used only to generate sample demand in the experiments.

\subsubsection{Parameters of methods} \label{sec:price_parameters}
We selected hyperparameters for each week of data via grid search with
$\beta \in \{10^{-5}, 10^{-4}, 10^{-3},10^{-2}\}$, \\
$\mu \in \{0.004, 0.02, 0.1, 0.5, 2.5\}$,
$m \in \{1, 10, 10^2\}$, and $N \in \{1,10,10^2\}$.
For \textbf{ZO-SPH} and \textbf{ZO-GA}, we fixed $m=1$ and chose $N$ from $\{1, 10, 10^2\}$.
For \textbf{ZO-CO}, \textbf{ZO-OG}, and \textbf{ZO-OGVR}, we selected $m$ from $\{1, 10, 10^2\}$.
Moreover, we set other parameters of \textbf{ZO-OGVR} as follows:
$c_0 = \sum_{j=1}^{20} f(\bx_0,\bmxi^{j}(\bx_0))$, $s_{\max} = 10$, $M=0.1$.
For each dataset, we ran three independent trials and selected the hyperparameter configuration with the highest mean performance.

\subsubsection{Data details.} \label{sec:data_detail}
The experimental dataset, ``New Product Sales Ranking,'' provided by KSP-SP Co., Ltd., includes confectionery price data.
We used the actual prices of confectionery as $\theta$ in the buyer's probability function for purchase.
We also determined the cost of each product based on the actual prices.

\subsection{Experiments on strategic classification}
\label{app:details_sclassification}
\subsubsection{Parameters of methods} 
\label{app:sclassification_parameters}
Hyperparameters were chosen from the following grids:
$\beta \in \{10^{-2}, 10^{-1}, 1\}$, $\mu \in \{0.1, 0.5, 1, 2, 4\}$,
$m \in \{1, 10, 10^2\}$, and $N \in \{1,10,10^2\}$.
We ran 20 independent trials per configuration and selected the one with the highest average over all trials.

\subsubsection{Data details}
The experimental dataset \citep{YEH20092473} includes features describing credit-card spending patterns, along with labels indicating default on payment. As with \citep{levanon2021strategic}, we used the preprocessed version of the data by \cite{ustun2019actionable}. The dataset includes $n=11$ features. We divided $13272$ data samples into a $12272$-sample training set and $1000$-sample test set in our experiments.

\section{Additional Experiments: multiproduct pricing experiments with different problem sizes}
To examine how the methods behave when the problem size is varied, we additionally conducted fully synthetic experiments for the multiproduct pricing problem.
We considered four problem sizes:
\[
(n_{\rm{product}}, m_{\rm{buyer}})
\in
\{(30,120),(60,240),(90,360),(120,480)\},
\]
where $n_{\rm{product}}$ denotes the number of products and $m_{\rm{buyer}}$ denotes the number of buyers.
The maximum numbers of samples used as the termination criteria were set to (5000), (10000), (15000), and (20000) for these problem sizes, respectively.
The model follows the setting in Appendix~\ref{sec:pricing_setting}, but $\theta_i$ was generated synthetically because the real dataset used in Section~\ref{sec:experiment} contains at most 50 products.
Specifically, the parameters $\{\theta_i\}_{i=1}^{n_{\rm product}}$ were independently sampled from the uniform distribution on $[0.1,1.0]$, and the resulting vector was normalized by its maximum value so that $\max_i \theta_i = 1$.
We compared the same baselines as in Section~\ref{sec:experiment}.
For each method and problem size, we first performed a grid search over the hyperparameters using three independently generated problem instances.
We then selected the hyperparameter combination that achieved the smallest average final objective value.
Using the selected hyperparameters, we reevaluated each method on 20 newly generated problem instances and computed the mean and standard deviation of the final objective value.
The selected hyperparameters and the corresponding evaluation results are reported in Table~\ref{tab:add_exp_scaled_pricing}.

\begin{table*}[t]
\centering
\caption{Additional experimental results for synthetic multiproduct pricing instances with different problem sizes.
For each problem size, all methods were evaluated on the same 20 problem instances.
Here, $\beta$, $\mu$, $m$, and $N$ denote the stepsize, smoothing parameter, batch size, and number of directions, respectively.
The columns obj and sd report the average final objective value and its standard deviation over the 20 instances, respectively.
The best value of obj for each problem size is shown in bold.}
\label{tab:add_exp_scaled_pricing}
\begin{tabular}{lcrrrrrr}
\toprule
Method
& \multicolumn{1}{c}{$(n_{\rm product}, m_{\rm buyer})$}
& \multicolumn{1}{c}{$\beta$}
& \multicolumn{1}{c}{$\mu$}
& \multicolumn{1}{c}{$m$}
& \multicolumn{1}{c}{$N$}
& \multicolumn{1}{c}{obj}
& \multicolumn{1}{c}{sd} \\
\midrule
\textbf{ZO-CO}
& (30, 120) & $10^{-3}$ & 0.5 & 1 & -- & -48.720 & 6.516 \\
\textbf{ZO-CO}
& (60, 240) & $10^{-3}$ & 0.5 & 1 & -- & -96.237 & 7.099 \\
\textbf{ZO-CO}
& (90, 360) & $10^{-3}$ & 0.5 & 1 & -- & -138.755 & 8.009 \\
\textbf{ZO-CO}
& (120, 480) & $10^{-3}$ & 0.5 & 1 & -- & -185.939 & 9.644 \\
\midrule
\textbf{ZO-SPH}
& (30, 120) & $10^{-2}$ & 0.5 & 1 & 100 & -52.101 & 4.549 \\
\textbf{ZO-SPH}
& (60, 240) & $10^{-3}$ & 0.5 & 1 & 10 & -100.790 & 4.272 \\
\textbf{ZO-SPH}
& (90, 360) & $10^{-4}$ & 0.5 & 1 & 1 & -138.190 & 5.601 \\
\textbf{ZO-SPH}
& (120, 480) & $10^{-3}$ & 2.5 & 1 & 10 & -189.080 & 8.103 \\
\midrule
\textbf{ZO-GA}
& (30, 120) & $10^{-3}$ & 0.1 & 1 & 10 & \textbf{-52.184} & 4.658 \\
\textbf{ZO-GA}
& (60, 240) & $10^{-4}$ & 0.1 & 1 & 1 & \textbf{-103.602} & 4.830 \\
\textbf{ZO-GA}
& (90, 360) & $10^{-4}$ & 0.1 & 1 & 1 & \textbf{-150.043} & 5.974 \\
\textbf{ZO-GA}
& (120, 480) & $10^{-3}$ & 0.1 & 1 & 10 & -198.247 & 6.697 \\
\midrule
\textbf{ZO-OG}
& (30, 120) & $10^{-5}$ & 0.1 & 1 & -- & -30.277 & 2.806 \\
\textbf{ZO-OG}
& (60, 240) & $10^{-5}$ & 0.1 & 1 & -- & -68.237 & 6.269 \\
\textbf{ZO-OG}
& (90, 360) & $10^{-5}$ & 0.1 & 1 & -- & -90.240 & 8.640 \\
\textbf{ZO-OG}
& (120, 480) & $10^{-5}$ & 0.1 & 1 & -- & -101.521 & 7.771 \\
\midrule
\textbf{ZO-OGVR}
& (30, 120) & $10^{-4}$ & 0.1 & 1 & -- & -51.018 & 4.712 \\
\textbf{ZO-OGVR}
& (60, 240) & $10^{-4}$ & 0.1 & 1 & -- & -95.487 & 6.168 \\
\textbf{ZO-OGVR}
& (90, 360) & $10^{-5}$ & 0.1 & 1 & -- & -149.169 & 4.451 \\
\textbf{ZO-OGVR}
& (120, 480) & $10^{-5}$ & 0.1 & 1 & -- & \textbf{-203.480} & 6.214 \\
\bottomrule
\end{tabular}
\end{table*}

In Table~\ref{tab:add_exp_scaled_pricing}, \textbf{ZO-GA} achieves the best average objective value for three of the four problem sizes, that is, 
$(n_{\rm product},m_{\rm buyer})=(30,120), (60,240), (90,360)$.
The strong performance of \textbf{ZO-GA} for these instances is consistent with the theoretical results in Section~\ref{subsec:prop_each_method}, 
which show that the Gaussian random-direction estimator has lower worst-case sample complexity.
On the other hand, \textbf{ZO-OGVR} is superior to other methods for the larger instances (i.e., $(n_{\rm product},m_{\rm buyer})=(120,480)$).
This empirical result differs from our theoretical results.
The difference arises because our theoretical analysis compares worst-case sample complexities for finding stationary points, whereas the experiments compare the final objective values attained under fixed, finite sample budgets.
The superior performance of \textbf{ZO-OGVR} on the largest instance should therefore be interpreted as an empirical benefit arising in a finite-budget setting.

\section{Additional Experiments: multiproduct pricing experiments with different hyperparameter settings}
\label{sec:add_exp_hparam_pricing}

To investigate the sensitivity of each method to the hyperparameter choices, we conducted additional experiments for the synthetic multiproduct pricing problem with the problem size fixed at $(n_{\rm product},m_{\rm buyer})=(30,120)$.
For each method, we first identified the best hyperparameter configuration in terms of the average final objective value and used it as the default configuration. We then changed only one hyperparameter at a time from this default configuration, while keeping all the other hyperparameters fixed at their default values. This design allows us to examine the effect of each hyperparameter separately.

\begin{longtable}{llrrrrrr}
\caption{Additional experimental results for the synthetic multiproduct pricing problem with fixed problem size $(n_{\rm product},m_{\rm buyer})=(30,120)$. For each method, the default row is the best hyperparameter configuration. Each subsequent row changes only the hyperparameter indicated in the second column, while keeping the remaining hyperparameters fixed at their default values. The columns obj and sd represent the average final objective value and its standard deviation over three trials, respectively.
Positive or divergent objective values are reported as ``--'', because the objective value, defined as the negative profit, should be negative when the method converges properly.
}
\label{tab:add_exp_hparam_one_factor}\\
\toprule
Method
& Varied parameter
& \multicolumn{1}{c}{$\beta$}
& \multicolumn{1}{c}{$\mu$}
& \multicolumn{1}{c}{$m$}
& \multicolumn{1}{c}{$N$}
& \multicolumn{1}{c}{obj}
& \multicolumn{1}{c}{sd} \\
\midrule
\endfirsthead

\toprule
Method
& Varied parameter
& \multicolumn{1}{c}{$\beta$}
& \multicolumn{1}{c}{$\mu$}
& \multicolumn{1}{c}{$m$}
& \multicolumn{1}{c}{$N$}
& \multicolumn{1}{c}{obj}
& \multicolumn{1}{c}{sd} \\
\midrule
\endhead

\midrule
\multicolumn{8}{r}{Continued on the next page}\\
\endfoot

\bottomrule
\endlastfoot

\textbf{ZO-CO}
& Default
& $10^{-3}$ & 0.5 & 1 & --
& \textbf{-45.944} & 4.459 \\

\textbf{ZO-CO}
& $\beta$
& $10^{-5}$ & 0.5 & 1 & --
& -- & -- \\

\textbf{ZO-CO}
& $\beta$
& $10^{-4}$ & 0.5 & 1 & --
& -15.911 & 0.666 \\

\textbf{ZO-CO}
& $\beta$
& $10^{-2}$ & 0.5 & 1 & --
& -36.394 & 2.296 \\

\textbf{ZO-CO}
& $\mu$
& $10^{-3}$ & 0.004 & 1 & --
& -- & -- \\

\textbf{ZO-CO}
& $\mu$
& $10^{-3}$ & 0.02 & 1 & --
& -- & -- \\

\textbf{ZO-CO}
& $\mu$
& $10^{-3}$ & 0.1 & 1 & --
& -42.681 & 1.516 \\

\textbf{ZO-CO}
& $\mu$
& $10^{-3}$ & 2.5 & 1 & --
& -2.480 & 0.351 \\

\textbf{ZO-CO}
& $m$
& $10^{-3}$ & 0.5 & 10 & --
& -15.248 & 0.750 \\

\textbf{ZO-CO}
& $m$
& $10^{-3}$ & 0.5 & 100 & --
& -- & -- \\

\midrule

\textbf{ZO-SPH}
& Default
& $10^{-2}$ & 0.5 & 1 & 100
& \textbf{-50.268} & 3.518 \\

\textbf{ZO-SPH}
& $\beta$
& $10^{-5}$ & 0.5 & 1 & 100
& -- & -- \\

\textbf{ZO-SPH}
& $\beta$
& $10^{-4}$ & 0.5 & 1 & 100
& -1.097 & 1.672 \\

\textbf{ZO-SPH}
& $\beta$
& $10^{-3}$ & 0.5 & 1 & 100
& -21.281 & 1.158 \\

\textbf{ZO-SPH}
& $\mu$
& $10^{-2}$ & 0.004 & 1 & 100
& -- & -- \\

\textbf{ZO-SPH}
& $\mu$
& $10^{-2}$ & 0.02 & 1 & 100
& -- & -- \\

\textbf{ZO-SPH}
& $\mu$
& $10^{-2}$ & 0.1 & 1 & 100
& -37.610 & 3.607 \\

\textbf{ZO-SPH}
& $\mu$
& $10^{-2}$ & 2.5 & 1 & 100
& -41.574 & 6.713 \\

\textbf{ZO-SPH}
& $N$
& $10^{-2}$ & 0.5 & 1 & 1
& -- & -- \\

\textbf{ZO-SPH}
& $N$
& $10^{-2}$ & 0.5 & 1 & 10
& -36.122 & 3.632 \\

\midrule

\textbf{ZO-GA}
& Default
& $10^{-3}$ & 0.1 & 1 & 10
& \textbf{-50.380} & 3.509 \\

\textbf{ZO-GA}
& $\beta$
& $10^{-5}$ & 0.1 & 1 & 10
& -1.223 & 1.521 \\

\textbf{ZO-GA}
& $\beta$
& $10^{-4}$ & 0.1 & 1 & 10
& -21.481 & 1.355 \\

\textbf{ZO-GA}
& $\beta$
& $10^{-2}$ & 0.1 & 1 & 10
& -- & -- \\

\textbf{ZO-GA}
& $\mu$
& $10^{-3}$ & 0.004 & 1 & 10
& -- & -- \\

\textbf{ZO-GA}
& $\mu$
& $10^{-3}$ & 0.02 & 1 & 10
& -43.963 & 2.661 \\

\textbf{ZO-GA}
& $\mu$
& $10^{-3}$ & 0.5 & 1 & 10
& -39.873 & 7.837 \\

\textbf{ZO-GA}
& $\mu$
& $10^{-3}$ & 2.5 & 1 & 10
& -- & -- \\

\textbf{ZO-GA}
& $N$
& $10^{-3}$ & 0.1 & 1 & 1
& -42.249 & 0.843 \\

\textbf{ZO-GA}
& $N$
& $10^{-3}$ & 0.1 & 1 & 100
& -21.598 & 1.006 \\

\midrule

\textbf{ZO-OG}
& Default
& $10^{-5}$ & 0.1 & 1 & --
& \textbf{-31.854} & 1.948 \\

\textbf{ZO-OG}
& $\beta$
& $10^{-4}$ & 0.1 & 1 & --
& -28.283 & 2.722 \\

\textbf{ZO-OG}
& $\beta$
& $10^{-3}$ & 0.1 & 1 & --
& -- & -- \\

\textbf{ZO-OG}
& $\beta$
& $10^{-2}$ & 0.1 & 1 & --
& -- & -- \\

\textbf{ZO-OG}
& $\mu$
& $10^{-5}$ & 0.004 & 1 & --
& -- & -- \\

\textbf{ZO-OG}
& $\mu$
& $10^{-5}$ & 0.02 & 1 & --
& -9.431 & 5.111 \\

\textbf{ZO-OG}
& $\mu$
& $10^{-5}$ & 0.5 & 1 & --
& -9.554 & 7.020 \\

\textbf{ZO-OG}
& $\mu$
& $10^{-5}$ & 2.5 & 1 & --
& -- & -- \\

\textbf{ZO-OG}
& $m$
& $10^{-5}$ & 0.1 & 10 & --
& -5.113 & 0.420 \\

\textbf{ZO-OG}
& $m$
& $10^{-5}$ & 0.1 & 100 & --
& -- & -- \\

\midrule

\textbf{ZO-OGVR}
& Default
& $10^{-4}$ & 0.1 & 1 & --
& \textbf{-49.379} & 4.184 \\

\textbf{ZO-OGVR}
& $\beta$
& $10^{-5}$ & 0.1 & 1 & --
& -31.874 & 1.228 \\

\textbf{ZO-OGVR}
& $\beta$
& $10^{-3}$ & 0.1 & 1 & --
& -- & -- \\

\textbf{ZO-OGVR}
& $\beta$
& $10^{-2}$ & 0.1 & 1 & --
& -- & -- \\

\textbf{ZO-OGVR}
& $\mu$
& $10^{-4}$ & 0.004 & 1 & --
& -- & -- \\

\textbf{ZO-OGVR}
& $\mu$
& $10^{-4}$ & 0.02 & 1 & --
& -- & -- \\

\textbf{ZO-OGVR}
& $\mu$
& $10^{-4}$ & 0.5 & 1 & --
& -28.649 & 7.403 \\

\textbf{ZO-OGVR}
& $\mu$
& $10^{-4}$ & 2.5 & 1 & --
& -- & -- \\

\textbf{ZO-OGVR}
& $m$
& $10^{-4}$ & 0.1 & 10 & --
& -28.639 & 4.560 \\

\textbf{ZO-OGVR}
& $m$
& $10^{-4}$ & 0.1 & 100 & --
& -2.887 & 0.664 \\

\end{longtable}

Table~\ref{tab:add_exp_hparam_one_factor} shows that the performance of each method is sensitive to the hyperparameter choices.
In the following, we discuss the empirical effect of each hyperparameter and the trade-offs underlying the observed results.

\paragraph{Effect of the stepsize parameter $\beta$.}
The stepsize parameter $\beta$ has a substantial impact on the performance of all methods.
When $\beta$ is too small, the iterates move only slightly at each update, and hence the optimization process cannot make sufficient progress within the fixed sample budget.
On the other hand, choosing $\beta$ too large can make the updates unstable for some methods.
These results suggest that the appropriate scale of $\beta$ depends on the gradient estimator.

\paragraph{Effect of the smoothing parameter $\mu$.}
The smoothing parameter $\mu$ also exhibits a clear trade-off.
When $\mu$ is too small, the finite-difference quotient amplifies the stochastic noise in the function-value observations, because the difference between two noisy function values is divided by $\mu$, as discussed in Section~\ref{subsec:basic_form}.
This phenomenon is clearly observed in the results for $\mu=0.004$.
For all methods, the objective values indicate unsuccessful convergence.
Conversely, when $\mu$ is too large, the bias induced by smoothing becomes large.
For example, \textbf{ZO-GA} achieves its best value at $\mu=0.1$, while increasing $\mu$ to $0.5$ worsens the objective value.
Similarly, for all other methods, increasing $\mu$ leads to worse or invalid objective values.
These observations are consistent with the theoretical error bounds in Section~\ref{subsec:prop_each_method}, where the smoothing parameter controls a trade-off between the variance term, which becomes large when $\mu$ is small, and the smoothing bias term, which becomes large when $\mu$ is large.
Therefore, an appropriate intermediate value of $\mu$ is crucial for obtaining good empirical performance.

\paragraph{Effect of the number of directions $N$ for \textbf{ZO-SPH} and \textbf{ZO-GA}.}
For the random-direction two-point estimators \textbf{ZO-SPH} and \textbf{ZO-GA}, the effect of increasing the number of directions $N$ differs between the two methods.
For \textbf{ZO-SPH}, increasing $N$ improves the final objective value: $N=1$ produces an invalid result, whereas the objective values for $N=10$ and $N=100$ are $-36.122$ and $-50.268$, respectively.
This result is consistent with our theoretical analysis, in which increasing $N$ reduces the variance of the gradient estimator.
For \textbf{ZO-GA}, increasing $N$ from $1$ to $10$ improves the objective value from $-42.249$ to $-50.380$, whereas a further increase to $N=100$ worsens it to $-21.598$.
Under a fixed sample budget, increasing $N$ reduces the variance of each gradient estimate but also reduces the number of update iterations.
Overall, the number of directions should be selected by accounting for both variance reduction and the resulting decrease in the number of iterations.

\paragraph{Effect of the batch size $m$ for \textbf{ZO-CO}, \textbf{ZO-OG}, and \textbf{ZO-OGVR}.}
Increasing $m$ reduces the variance caused by the stochastic approximation of $F(\bx)$, but it also increases the number of samples used per iteration.
Thus, under a fixed sample budget, a larger $m$ reduces the number of update iterations.
In the present experiments, all three methods achieve their best objective values with $m=1$.
These results suggest that, in our experimental setting, the benefit of performing more update iterations with $m=1$ outweighs the variance reduction obtained by increasing the batch size.
Nevertheless, the optimal batch size may depend on the stepsize and the total sample budget, and these parameters should therefore be tuned jointly.

\section{Sufficient conditions for Assumptions \ref{asm:F_smooth} and \ref{asp:F_hessian_lipschitz}}
\label{app:suf_condition}

The following lemma is based on \citet[Lemma~1]{ray2022decision},
with corrected constants.

\begin{lemma}
Suppose that there exist a matrix $\bA$ and distribution $D'$ such that
$$ \bmxi \sim D(\bx) \Longleftrightarrow \bmxi = \bnu + \bA\bx,$$
where $\bnu \sim D'$. 
If $\nabla_{(\bx,\bmxi)} f(\bx,\bmxi)$ is $\beta$-Lipschitz continuous,
then Assumption \ref{asm:F_smooth} holds with
\[
M:=\beta(1+\|A\|_{\mathrm{op}}^2).
\]
where $\|A\|_{\textrm{op}}$ is the operator norm of $A$, that is, $\|A\|_{\textrm{op}}:= \sup_{x\in \R^d, x \neq 0}\frac{\|Ax\|}{\|x\|}$.
If $\nabla_{(\bx,\bmxi)}^2 f(\bx,\bmxi)$ is $\rho$--Lipschitz continuous,
then Assumption \ref{asp:F_hessian_lipschitz} holds with
\[
H:=\rho(1+\|A\|_{\mathrm{op}}^2)^{3/2}.
\]
\end{lemma}

\begin{proof}
Let $g_{\bnu}(\bx):=f(\bx,\bnu+A\bx)$ and
$B:=\begin{pmatrix}I\\A\end{pmatrix}$.\footnote{In \citet[Lemma~1]{ray2022decision}, the authors use the
block-diagonal matrix
$B:=\begin{pmatrix}I&O\\O&A\end{pmatrix}$.
However, the chain rule for the gradient shows that
our definition of $B$ is correct.
Consequently, the constants in our result differ from
those in \citet{ray2022decision}.}
Then,
\begin{align}
\|B\|_{\mathrm{op}}^2
=\lambda_{\max}(B^\top B)
=\lambda_{\max}(I+A^\top A)
=1+\lambda_{\max}(A^\top A)
=1+\|A\|_{\mathrm{op}}^2. \label{eq:B_op_2norm}
\end{align}
For each fixed $\bnu$, the chain rule gives
\begin{align*}
\nabla_{\bx} g_{\bnu}(\bx)
&=\nabla_{\bx} f(\bx,\bnu+A\bx)
  +A^\top\nabla_{\bmxi} f(\bx,\bnu+A\bx)\\
&=
\begin{pmatrix}I&A^\top\end{pmatrix}
\begin{pmatrix}
\nabla_{\bx} f(\bx,\bnu+A\bx)\\
\nabla_{\bmxi} f(\bx,\bnu+A\bx)
\end{pmatrix}\\
&=B^\top\nabla_{(\bx,\bmxi)} f(\bx,\bnu+A\bx).
\end{align*}
Then, for every fixed $\bnu$
and all $\bx,\by\in\R^d$, we have
\begin{align*}
\|\nabla_{\bx} g_{\bnu}(\bx)-\nabla_{\bx} g_{\bnu}(\by)\|
&=
\left\|
B^\top\bigl[
\nabla_{(\bx,\bmxi)} f(\bx,\bnu+A\bx)
-\nabla_{(\bx,\bmxi)} f(\by,\bnu+A\by)
\bigr]
\right\|\\
&\le
\|B\|_{\mathrm{op}}
\left\|
\nabla_{(\bx,\bmxi)} f(\bx,\bnu+A\bx)
-\nabla_{(\bx,\bmxi)} f(\by,\bnu+A\by)
\right\|\\
&\overset{(a)}{\le}
\beta\|B\|_{\mathrm{op}}
\left\|
\begin{pmatrix}\bx\\\bnu+A\bx\end{pmatrix}
-
\begin{pmatrix}\by\\\bnu+A\by\end{pmatrix}
\right\|\\
&=
\beta\|B\|_{\mathrm{op}}
\left\|
\begin{pmatrix}
\bx-\by\\
A(\bx-\by)
\end{pmatrix}
\right\|\\
&=
\beta\|B\|_{\mathrm{op}}\|B(\bx-\by)\|\\
&\le
\beta\|B\|_{\mathrm{op}}^2\|\bx-\by\|\\
&\overset{(b)}{=}
\beta(1+\|A\|_{\mathrm{op}}^2)\|\bx-\by\|.
\end{align*}
where (a) uses the $\beta$-Lipschitz continuity
of $\nabla_{(\bx,\bmxi)} f$ and (b) comes from  \eqref{eq:B_op_2norm}.
Hence,
\begin{align*}
\|\nabla F(\bx)-\nabla F(\by)\|
&=
\left\|
\E_{\bnu\sim D'}\left[
\nabla g_{\bnu}(\bx)-\nabla g_{\bnu}(\by)
\right]
\right\|\\
&\le
\E_{\bnu\sim D'}\left[
\|\nabla g_{\bnu}(\bx)-\nabla g_{\bnu}(\by)\|
\right]\\
&\le
\beta(1+\|A\|_{\mathrm{op}}^2)\|\bx-\by\|.
\end{align*}
Thus, Assumption~\ref{asm:F_smooth} holds with
$M=\beta(1+\|A\|_{\mathrm{op}}^2)$.

Moreover, we obtain
\begin{align*}
\nabla_{\bx}^2 g_{\bnu}(\bx)
&=\nabla_{\bx}^2 f(\bx,\bnu+A\bx)
  +\nabla_{\bx\bmxi}^2 f(\bx,\bnu+A\bx)A\\
&\quad
  +A^\top\nabla_{\bmxi\bx}^2 f(\bx,\bnu+A\bx)
  +A^\top\nabla_{\bmxi}^2 f(\bx,\bnu+A\bx)A\\
&=
\begin{pmatrix}I&A^\top\end{pmatrix}
\begin{pmatrix}
\nabla_{\bx}^2 f & \nabla_{\bx\bmxi}^2 f\\
\nabla_{\bmxi\bx}^2 f & \nabla_{\bmxi}^2 f
\end{pmatrix}(\bx,\bnu+A\bx)
\begin{pmatrix}I\\A\end{pmatrix}\\
&=B^\top\nabla_{(\bx,\bmxi)}^2 f(\bx,\bnu+A\bx)B,
\end{align*}
where $\nabla_{\bx\bmxi}^2 f$ and $\nabla_{\bmxi\bx}^2 f$
denote the Jacobian matrices of $\nabla_{\bx} f$ with respect
to $\bmxi$ and of $\nabla_{\bmxi} f$ with respect to $\bx$,
respectively.
For every fixed
$\bnu$ and all $\bx,\by\in\R^d$, we obtain
\begin{align*}
\|\nabla_{\bx}^2 g_{\bnu}(\bx)-\nabla_{\bx}^2 g_{\bnu}(\by)\|_{\mathrm{op}}
&=
\left\|
B^\top\bigl[
\nabla_{(\bx,\bmxi)}^2 f(\bx,\bnu+A\bx)
-\nabla_{(\bx,\bmxi)}^2 f(\by,\bnu+A\by)
\bigr]B
\right\|_{\mathrm{op}}\\
&\le
\|B\|_{\mathrm{op}}^2
\left\|
\nabla_{(\bx,\bmxi)}^2 f(\bx,\bnu+A\bx)
-\nabla_{(\bx,\bmxi)}^2 f(\by,\bnu+A\by)
\right\|_{\mathrm{op}}\\
&\overset{(a)}{\le}
\rho\|B\|_{\mathrm{op}}^2
\left\|
\begin{pmatrix}\bx\\\bnu+A\bx\end{pmatrix}
-
\begin{pmatrix}\by\\\bnu+A\by\end{pmatrix}
\right\|\\
&=
\rho\|B\|_{\mathrm{op}}^2
\left\|
\begin{pmatrix}
\bx-\by\\
A(\bx-\by)
\end{pmatrix}
\right\|\\
&=
\rho\|B\|_{\mathrm{op}}^2\|B(\bx-\by)\|\\
&\le
\rho\|B\|_{\mathrm{op}}^3\|\bx-\by\|\\
&\overset{(b)}{=}
\rho(1+\|A\|_{\mathrm{op}}^2)^{3/2}\|\bx-\by\|,
\end{align*}
where (a) uses the $\rho$-Lipschitz
continuity of $\nabla_{(\bx,\bmxi)}^2 f$ and (b) comes from \eqref{eq:B_op_2norm}.
Hence, 
\begin{align*}
\|\nabla^2 F(\bx)-\nabla^2 F(\by)\|_{\mathrm{op}}
&=
\left\|
\E_{\bnu\sim D'}\left[
\nabla^2 g_{\bnu}(\bx)-\nabla^2 g_{\bnu}(\by)
\right]
\right\|_{\mathrm{op}}\\
&\le
\E_{\bnu\sim D'}\left[
\|\nabla^2 g_{\bnu}(\bx)-\nabla^2 g_{\bnu}(\by)\|_{\mathrm{op}}
\right]\\
&\le
\rho(1+\|A\|_{\mathrm{op}}^2)^{3/2}\|\bx-\by\|.
\end{align*}
Thus, Assumption~\ref{asp:F_hessian_lipschitz} holds with
$H=\rho(1+\|A\|_{\mathrm{op}}^2)^{3/2}$.
\end{proof}

\section{Convergence analysis of a gradient estimator based on Rademacher directions}
\label{sec:rademcher}
We consider the following gradient estimator:
\begin{align*}
\bg^{\rm ra}(\bx, \hat{E}, \Xi_1, \Xi_2):=
\frac{1}{N} \sum_{i=1}^{N} \frac{f(\bx+\mu \hat{\bme}^i,\bmxi^{1,i})-f(\bx - \mu \hat{\bme}^i,\bmxi^{2,i})}{2\mu} \hat{\bme}^i.
\end{align*}
Here, $\mu \in \R_{>0}$;
$\hat{E}:=\{\hat{\bme}^i\}_{i=1}^N$, where each entry $\hat e^i_j$ of $\hat{\bme}^i$ satisfies
\[
\hat e^i_j \sim {\rm Rad}\!\left(\tfrac12\right),
\qquad\text{i.e.,}\qquad
\mpr(\hat e^i_j = +1)=\mpr (\hat e^i_j = -1)=\frac{1}{2}
\quad (j\in[d]);
\]
$\Xi_1:=\{\bmxi^{1,i}\}_{i=1}^N$, where $\bmxi^{1,i} \sim D(\bx+\mu \hat{\bme}^i)$;
$\Xi_2:=\{\bmxi^{2,i}\}_{i=1}^N$, where $\bmxi^{2,i} \sim D(\bx-\mu \hat{\bme}^i)$.
Moreover, we give the mini-batch version of the gradient estimator:
\begin{align}
\bg^{\rm ra}_m(\bx, \hat{E},\{\Xi_1^{(j)}\}_{j=1}^m, \{\Xi_2^{(j)}\}_{j=1}^m)
:=\frac{1}{m}\sum_{j=1}^{m} \bg^{\rm ra}(\bx,\hat{E},\Xi_1^{(j)},\Xi_2^{(j)}).
\label{eq:mini_batch_gradient_est_ra}
\end{align}
In the proof, we let
\begin{align*}
\bg_F^{\rm ra}(\bx, \hat{E})
:= \frac{1}{N} \sum_{i=1}^N  \frac{F(\bx+\mu \hat{\bme}^i)-F(\bx - \mu \hat{\bme}^i)}{2\mu}\hat{\bme}^i.
\end{align*}

Then, the following lemmas hold.
\begin{lemma}
\label{lem:g_variance_Ra_batch1}
Suppose that Assumptions \ref{asm:F_sample_var_bound} and \ref{asm:F_smooth} hold.
Then, the following holds.
\begin{align*}
&\E_{\hat{E},\{\Xi_{1}^{(j)}\}_{j=1}^m, \{\Xi_{2}^{(j)}\}_{j=1}^m}
\left[
\left\|
\bg^{\rm ra}_m(\bx,\hat{E}, \{\Xi_{1}^{(j)}\}_{j=1}^m, \{\Xi_{2}^{(j)}\}_{j=1}^m)
- \nabla F(\bx)
\right\|^2
\right]
\\
&\le \frac{3\sigma^2d}{\mu^2 N m}
+ 3d \mu^2 M^2 \left(\sqrt{d-1}+\frac12\right)^2
+ \frac{6d}{N}\left(\frac{M^2\mu^2}{4}\,d^2 + 3 \|\nabla F(\bx)\|^2\right).
\end{align*}
\end{lemma}

\begin{proof}
First, we have 
\begin{align}
& \E_{\hat{E}, \{\Xi_{1}^{(j)}\}_{j=1}^m, \{\Xi_{2}^{(j)}\}_{j=1}^m}
\left[
\left\|
\bg^{\rm ra}_m(\bx,\hat{E}, \{\Xi_{1}^{(j)}\}_{j=1}^m, \{\Xi_{2}^{(j)}\}_{j=1}^m)
- \nabla F(\bx)
\right\|^2
\right] \nonumber\\
&=\E_{\hat{E}, \{\Xi_{1}^{(j)}\}_{j=1}^m, \{\Xi_{2}^{(j)}\}_{j=1}^m}
\big[
\|
\bg^{\rm ra}_m(\bx,\hat{E}, \{\Xi_{1}^{(j)}\}_{j=1}^m, \{\Xi_{2}^{(j)}\}_{j=1}^m)
-\bg_F^{\rm ra}(\bx,\hat{E}) \nonumber\\
&\hspace{40mm} + \E_{\hat{E}}[\bg_F^{\rm ra}(\bx,\hat{E})]  - \nabla F(\bx)
+\bg_F^{\rm ra}(\bx,\hat{E})  - \E_{\hat{E}}[\bg_F^{\rm ra}(\bx,\hat{E})]\|^2
\big] \nonumber\\
&\le
3 \E_{\hat{E}, \{\Xi_{1}^{(j)}\}_{j=1}^m, \{\Xi_{2}^{(j)}\}_{j=1}^m}
\big[
\underbrace{
\|
\bg^{\rm ra}_m(\bx,\hat{E}, \{\Xi_{1}^{(j)}\}_{j=1}^m, \{\Xi_{2}^{(j)}\}_{j=1}^m)
-\bg_F^{\rm ra}(\bx,\hat{E})
\|^2
}_{\text{(I)}}\big] \nonumber\\
&+ 3 \underbrace{
\| \E_{\hat{E}}[\bg_F^{\rm ra}(\bx,\hat{E})] - \nabla F(\bx)\|^2
}_{\text{(II)}}
+ 3 \E_{\hat{E}}\big[ \underbrace{
\|\bg_F^{\rm ra}(\bx,\hat{E}) - \E_{\hat{E}}[\bg_F^{\rm ra}(\bx,\hat{E})]\|^2
}_{\text{(III)}}
\big]. \label{eq:split_I_III_ra_first}
\end{align}
Therefore, we derive the upper bound for (I)--(III).
Let $\hat{\bme}\in\{-1,+1\}^d$ be a random vector with i.i.d.\ entries
$\hat e_j\sim{\rm Rad}\!\left(\tfrac12\right)$.

\textbf{(I):} We have
\begin{align}
& \E_{\hat{E}, \{\Xi_{1}^{(j)}\}_{j=1}^m, \{\Xi_{2}^{(j)}\}_{j=1}^m}
\left[
\left\|
\bg^{\rm ra}_m(\bx,\hat{E}, \{\Xi_{1}^{(j)}\}_{j=1}^m, \{\Xi_{2}^{(j)}\}_{j=1}^m)
-\bg_F^{\rm ra}(\bx,\hat{E})
\right\|^2
\right] \nonumber \\
&=
\E_{\hat{E}, \{\Xi_{1}^{(j)}\}_{j=1}^m, \{\Xi_{2}^{(j)}\}_{j=1}^m}
\bigg[
\bigg\|
\frac{1}{N}\sum_{i=1}^{N}
\frac{\frac{1}{m}\sum_{j=1}^{m} f(\bx+\mu \hat{\bme}^i, \bmxi^{1,i,(j)})
-\frac{1}{m}\sum_{j=1}^{m} f(\bx-\mu \hat{\bme}^i, \bmxi^{2,i,(j)})}{2\mu}\hat{\bme}^i
\nonumber\\
&\hspace{40mm}
-\frac{1}{N} \sum_{i=1}^{N} \frac{F(\bx+\mu \hat{\bme}^i)-F(\bx - \mu \hat{\bme}^i)}{2\mu}\hat{\bme}^i
\bigg\|^2
\bigg] \nonumber\\
& \le 2\E_{\hat{E},\{\Xi_{1}^{(j)}\}_{j=1}^m, \{\Xi_{2}^{(j)}\}_{j=1}^m} \Bigg[ \left\| \frac{1}{N}\sum_{i=1}^{N}\frac{\frac{1}{m}\sum_{j=1}^{m} f(\bx+\mu \hat{\bme}^i, \bmxi^{1,i,(j)}) - F(\bx+\mu \hat{\bme}^i)}{2\mu}\hat{\bme}^i\right\|^2 \nonumber\\
    &\hspace{40mm} + \left\|\frac{1}{N} \sum_{i=1}^{N} \frac{\frac{1}{m}\sum_{j=1}^{m} f(\bx-\mu \hat{\bme}^i, \bmxi^{2,i,(j)})-F(\bx - \mu \hat{\bme}^i)}{2\mu}\hat{\bme}^i \right\|^2  \Bigg] \nonumber\\
& = \frac{1}{2\mu^2N^2} \E_{\hat{E}} \Bigg[ \E_{\{\Xi_{1}^{(j)}\}_{j=1}^m} \Bigg[ \left\| \sum_{i=1}^{N}\left(\frac{1}{m}\sum_{j=1}^{m} f(\bx+\mu \hat{\bme}^i, \bmxi^{1,i,(j)}) - F(\bx+\mu \hat{\bme}^i)\right)\hat{\bme}^i\right\|^2 \Bigg]  \nonumber\\
    &\hspace{20mm} + \E_{\{\Xi_{2}^{(j)}\}_{j=1}^m} \Bigg[ \left\|\sum_{i=1}^{N} \left(\frac{1}{m}\sum_{j=1}^{m} f(\bx-\mu \hat{\bme}^i, \bmxi^{2,i,(j)}) - F(\bx-\mu \hat{\bme}^i)\right)\hat{\bme}^i \right\|^2  \Bigg]  \Bigg] \nonumber \\
& \overset{(a)}{=}\frac{1}{2\mu^2N^2} \E_{\hat{E}} \Bigg[ \E_{\{\Xi_{1}^{(j)}\}_{j=1}^m} \Bigg[ \sum_{i=1}^{N} \left\| \left(\frac{1}{m}\sum_{j=1}^{m} f(\bx+\mu \hat{\bme}^i, \bmxi^{1,i,(j)}) - F(\bx+\mu \hat{\bme}^i)\right)\hat{\bme}^i\right\|^2 \Bigg]  \nonumber\\
    &\hspace{20mm} + \E_{\{\Xi_{2}^{(j)}\}_{j=1}^m} \Bigg[ \sum_{i=1}^{N} \left\| \left(\frac{1}{m}\sum_{j=1}^{m} f(\bx-\mu \hat{\bme}^i, \bmxi^{2,i,(j)}) - F(\bx-\mu \hat{\bme}^i)\right)\hat{\bme}^i \right\|^2  \Bigg]  \Bigg] \nonumber \\
& \overset{(b)}{\le} \frac{1}{2\mu^2N^2} \E_{\hat{E}} \left[\sum_{i=1}^N \frac{\sigma^2}{m} \| \hat{\bme}^i\|^2 + \sum_{i=1}^N \frac{\sigma^2}{m}\| \hat{\bme}^i\|^2 \right] \nonumber \\
& \overset{(c)}{=} \frac{\sigma^2d}{\mu^2 Nm}, 
\label{eq:var_partial_ra_r}
\end{align}
where (a) is due to the fact that, for $i\neq k$, 
\footnotesize
\begin{align*}
\E_{\bmxi^{1,i,(j)},\bmxi^{1,k,(\ell)}}\left[ \left(f(\bx+\mu \hat{\bme}^i, \bmxi^{1,i,(j)}) - F(\bx+\mu \hat{\bme}^i)\right)\left(f(\bx+\mu \hat{\bme}^k, \bmxi^{1,k,(\ell)}) - F(\bx+\mu \hat{\bme}^k)\right)\hat{\bme}^{i\top} \hat{\bme}^k\right]=0,   
\end{align*}
\normalsize
and 
\footnotesize
\begin{align*}
\E_{\bmxi^{2,i,(j)},\bmxi^{2,k,(\ell)}}\left[ \left(f(\bx-\mu \hat{\bme}^i, \bmxi^{2,i,(j)}) - F(\bx-\mu \hat{\bme}^i)\right)\left(f(\bx-\mu \hat{\bme}^k, \bmxi^{2,k,(\ell)}) - F(\bx-\mu \hat{\bme}^k)\right)\hat{\bme}^{i\top} \hat{\bme}^k\right]=0;
\end{align*}
\normalsize
(b) is due to Assumption \ref{asm:F_sample_var_bound} and Lemma \ref{lem:minibatch_var_reduction};
(c) comes from $\|\hat{\bme}^i\|^2=d$.

\textbf{(II):}
Fix any $i\in[d]$ and let $\hat{\bme}_{-i}$ be the vector obtained from $\hat{\bme}$ by
replacing its $i$-th entry with $0$.
Define
\[
A_i(\bx):=\E_{\hat{\bme}_{-i}}\!\left[\partial_i F(\bx+\mu \hat{\bme}_{-i})\right].
\]
Then, by the triangle inequality,
\begin{align}
\left|(\E_{\hat{E}}[\bg_F^{\rm ra}(\bx,\hat{E})])_i-\partial_i F(\bx)\right|
\le \left|(\E_{\hat{E}}[\bg_F^{\rm ra}(\bx,\hat{E})])_i-A_i(\bx)\right|
    +\left|A_i(\bx)-\partial_i F(\bx)\right|.                 \label{eq:b_split}
\end{align}

By Assumption~\ref{asm:F_smooth},
$|\partial_i F(u)-\partial_i F(v)|\le \|\nabla F(u)-\nabla F(v)\|\le M\|u-v\|$ for any $u,v$.
Hence, using $\|\hat{\bme}_{-i}\|=\sqrt{d-1}$,
\begin{align}
\left|A_i(\bx)-\partial_i F(\bx)\right|
&\le \E_{\hat{\bme}_{-i}}\!\left[\left|\partial_i F(\bx+\mu \hat{\bme}_{-i})-\partial_i F(\bx)\right|\right]
\le M\,\E_{\hat{\bme}_{-i}}\!\left[\|\mu \hat{\bme}_{-i}\|\right]
= M \mu\sqrt{d-1}.                                          \label{eq:b_shift}
\end{align}

Since $\hat e_i\in\{-1,+1\}$, we have
\begin{align}
(\E_{\hat{E}}[\bg_F^{\rm ra}(\bx,\hat{E})])_i
&=\frac{1}{2\mu}\E_{\hat{\bme}_{-i}}\!\left[
F(\bx+\mu \hat{\bme}_{-i}+\mu \bme_i)-F(\bx+\mu \hat{\bme}_{-i}-\mu \bme_i)
\right],  \label{eq:b_ra_component}
\end{align}
where $\bme_i$ is the $i$-th canonical basis vector.
For a fixed $\by$, define the one-dimensional function $h(t):=F(\by+t\bme_i)$.
Then $h'(t)=\partial_i F(\by+t\bme_i)$ and $|h'(t)-h'(0)|\le M |t|$ for all $t$ by $M$-smoothness.
It yields from
\[
\frac{F(\by+\mu \bme_i)-F(\by-\mu \bme_i)}{2\mu}
= \frac{1}{2\mu}\int_{-\mu}^{\mu} h'(t)\,dt
\]
that
\[
\left|\partial_i F(\by)-\frac{F(\by+\mu \bme_i)-F(\by-\mu \bme_i)}{2\mu}\right|
\le \frac{1}{2\mu}\int_{-\mu}^{\mu} |h'(0)-h'(t)|\,dt
\le \frac{1}{2\mu}\int_{-\mu}^{\mu} M |t|\,dt
= \frac{M\mu}{2}.
\]
Applying this with $\by=\bx+\mu \hat{\bme}_{-i}$ yields
\begin{align}
&\left|(\E_{\hat{E}}[\bg_F^{\rm ra}(\bx,\hat{E})])_i-A_i(\bx)\right|\nonumber\\
&\overset{(a)}{=}\left|\frac{1}{2\mu}\E_{\hat{\bme}_{-i}}\!\left[
F(\bx+\mu \hat{\bme}_{-i}+\mu \bme_i)-F(\bx+\mu \hat{\bme}_{-i}-\mu \bme_i)
\right] - \E_{\hat{\bme}_{-i}}\!\left[\partial_i F(\bx+\mu \hat{\bme}_{-i})\right] \right| \nonumber\\
&=\left|\E_{\hat{\bme}_{-i}}\!\left[\frac{F(\bx+\mu \hat{\bme}_{-i}+\mu \bme_i)-F(\bx+\mu \hat{\bme}_{-i}-\mu \bme_i)}{2\mu}
 - \partial_i F(\bx+\mu \hat{\bme}_{-i})\right] \right| \nonumber\\
&\le \frac{M\mu}{2},      \label{eq:b_cdiff}
\end{align}
where (a) comes from \eqref{eq:b_ra_component}.

By \eqref{eq:b_split}, \eqref{eq:b_shift}, and \eqref{eq:b_cdiff}, we obtain for all $i\in[d]$,
\[
\left|(\E_{\hat{E}}[\bg_F^{\rm ra}(\bx,\hat{E})])_i-\partial_i F(\bx)\right|
\le M\mu\left(\sqrt{d-1}+\frac{1}{2}\right).
\]
Therefore,
\begin{align}
\|\E_{\hat{E}}[\bg_F^{\rm ra}(\bx,\hat{E})]-\nabla F(\bx)\|^2
&=\sum_{i=1}^d \left((\E_{\hat{E}}[\bg_F^{\rm ra}(\bx,\hat{E})])_i-\partial_i F(\bx)\right)^2
\le d\,M^2\mu^2\left(\sqrt{d-1}+\frac{1}{2}\right)^2.        \label{eq:b_final_ra}
\end{align}

\textbf{(III):}
For any $\bg \in \R^d$, we have
\[
\left(\E_{\hat{\bme}}\big[(\bg^\top \hat{\bme})^2\,\hat{\bme}\hat{\bme}^\top\big]\right)_{pq}
=\sum_{k=1}^d\sum_{\ell=1}^d g_k g_\ell\,\E_{\hat{\bme}}\!\left[\hat e_k \hat e_\ell \hat e_p \hat e_q\right].
\]
Then, since
\[
\E_{\hat{\bme}}[\hat e_k \hat e_\ell \hat e_p \hat e_q]
=
\begin{cases}
1, & (k=\ell,\ p=q),\\
1, & (k=p,\ \ell=q),\\
1, & (k=q,\ \ell=p),\\
0, & \text{otherwise},
\end{cases}
\]
we have
\begin{align*}
\left(\E_{\hat{\bme}}\big[(\bg^\top \hat{\bme})^2\,\hat{\bme}\hat{\bme}^\top\big]\right)_{pq}
=
\begin{cases}
\displaystyle
\sum_{k=1}^d g_k^2
=\|\bg\|^2,
& p=q,\\[3mm]
\displaystyle
g_p g_q + g_q g_p
=2g_p g_q,
& p\neq q. 
\end{cases} 
\end{align*}
Therefore,
\begin{align}
\E_{\hat{\bme}}\big[(\bg^\top \hat{\bme})^2\,\hat{\bme}\hat{\bme}^\top\big] =  \bg^\top \bg\, I
+ 2 \bg\bg^\top
- 2\,{\rm Diag}(\bg_1^2,\dots,\bg_d^2). 
    \label{eq_rad_st}
\end{align}

Here,
\small
\begin{align}
&{\rm Var}(\bg_F^{{\rm ra}}(\bx, \hat{E})) \\
&= \frac{1}{N} {\rm Var}\left( \frac{F(\bx+\mu \hat{\bme})-F(\bx - \mu \hat{\bme})}{2\mu}\hat{\bme}\right)
\nonumber \\
&= \frac{1}{N} \E_{\hat{\bme}\sim {\rm Unif}(\{-1,+1\}^d)}\left[\left(\frac{F(\bx+\mu \hat{\bme})-F(\bx - \mu \hat{\bme})}{2\mu}\hat{\bme}\right) \left(\frac{F(\bx+\mu \hat{\bme})-F(\bx - \mu \hat{\bme})}{2\mu}\hat{\bme}\right)^\top \right] 
\nonumber \\
& \quad - \frac{1}{N}\E_{\hat{\bme}\sim {\rm Unif}(\{-1,+1\}^d)}\left[\frac{F(\bx+\mu \hat{\bme})-F(\bx - \mu \hat{\bme})}{2\mu}\hat{\bme}\right] \E_{\hat{\bme}\sim {\rm Unif}(\{-1,+1\}^d)}\left[\frac{F(\bx+\mu \hat{\bme})-F(\bx - \mu \hat{\bme})}{2\mu}\hat{\bme} \right]^\top \nonumber \\
& \preceq \frac{1}{N}\E_{\hat{\bme}\sim {\rm Unif}(\{-1,+1\}^d)}\left[ \left( \frac{F(\bx+\mu \hat{\bme})-F(\bx - \mu \hat{\bme})}{2\mu}\right)^2  \hat{\bme} \hat{\bme}^\top \right] \nonumber\\
& =
\frac{1}{N}\E_{\hat{\bme}\sim {\rm Unif}(\{-1,+1\}^d)}
\left[
\left(
\frac{F(\bx+\mu \hat{\bme})-F(\bx - \mu \hat{\bme}) - 2\mu \nabla F(\bx)^\top \hat{\bme}}{2\mu}
+ \nabla F(\bx)^\top \hat{\bme}
\right)^2
\hat{\bme}\hat{\bme}^\top
\right] \nonumber \\
& \preceq
\frac{2}{N}\E_{\hat{\bme}\sim {\rm Unif}(\{-1,+1\}^d)}
\left[
\left(
\frac{F(\bx+\mu \hat{\bme})-F(\bx - \mu \hat{\bme}) - 2\mu \nabla F(\bx)^\top \hat{\bme}}{2\mu}
\right)^2
\hat{\bme}\hat{\bme}^\top
+ (\nabla F(\bx)^\top \hat{\bme})^2 \hat{\bme}\hat{\bme}^\top
\right] \nonumber \\
&=
\frac{2}{N}\E_{\hat{\bme}\sim {\rm Unif}(\{-1,+1\}^d)}
\bigg[
\bigg(
\frac{F(\bx+\mu \hat{\bme})-F(\bx) - \mu \nabla F(\bx)^\top \hat{\bme}}{2\mu}
-\frac{F(\bx-\mu \hat{\bme})-F(\bx) + \mu \nabla F(\bx)^\top \hat{\bme}}{2\mu}
\bigg)^2
\hat{\bme}\hat{\bme}^\top \nonumber\\
&\hspace{45mm}
+ (\nabla F(\bx)^\top \hat{\bme})^2 \hat{\bme}\hat{\bme}^\top
\bigg] \label{eq:co_var_progress_ra} \\
&\overset{(a)}{\preceq}
\frac{2}{N}\E_{\hat{\bme}\sim {\rm Unif}(\{-1,+1\}^d)}
\left[
\left(
\frac{M\|\mu \hat{\bme}\|^2}{4\mu} + \frac{M\|\mu \hat{\bme}\|^2}{4\mu}
\right)^2
\hat{\bme}\hat{\bme}^\top
+ (\nabla F(\bx)^\top \hat{\bme})^2 \hat{\bme}\hat{\bme}^\top
\right] \nonumber \\
&=
\frac{2}{N}\E_{\hat{\bme}\sim {\rm Unif}(\{-1,+1\}^d)}
\left[
\frac{M^2\mu^2 \|\hat{\bme}\|^4}{4}\,\hat{\bme}\hat{\bme}^\top
+ (\nabla F(\bx)^\top \hat{\bme})^2 \hat{\bme}\hat{\bme}^\top
\right] \nonumber \\
&\overset{(b)}{=}
\frac{2}{N}\left(
\frac{M^2\mu^2}{4}\,d^2 I
+ \nabla F(\bx)^\top \nabla F(\bx)\, I
+ 2 \nabla F(\bx)\nabla F(\bx)^\top
- 2\,{\rm Diag}(\nabla F(\bx)_1^2,\dots,\nabla F(\bx)_d^2)
\right) \nonumber \\
&\preceq
\frac{2}{N}\left(
\frac{M^2\mu^2}{4}\,d^2 I
+ \nabla F(\bx)^\top \nabla F(\bx)\, I
+ 2 \nabla F(\bx)\nabla F(\bx)^\top
\right) \nonumber \\
&\overset{(c)}{\preceq}
\frac{2}{N}\left(
\frac{M^2\mu^2}{4}\,d^2 + 3\|\nabla F(\bx)\|^2
\right)I, \nonumber
\end{align}
\normalsize
where (a) is due to Lemma \ref{lem:first_order_taylor_remainder}, which follows from Assumption \ref{asm:F_smooth};
(b) holds from $\|\hat{\bme}\|^2=d$ and \eqref{eq_rad_st} with $\bg:= \nabla F(\bx)$;
(c) follows from $\nabla F(\bx)\nabla F(\bx)^\top \preceq \|\nabla F(\bx)\|^2 I$
(as confirmed in the proof of Lemma \ref{lem:g_variance_sp_batch1}).
Then, 
\begin{align}
&\E_{\hat{E}}\!\left[\|\bg_F^{{\rm ra}}(\bx,\hat{E}) - \E_{\hat{E}}[\bg_F^{{\rm ra}}(\bx,\hat{E})] \|^2\right] =
{\rm tr}\!\left({\rm Var}(\bg_F^{{\rm ra}}(\bx, \hat{E}))\right) \le \frac{2d}{N}\left(\frac{M^2\mu^2}{4}\,d^2 + 3 \|\nabla F(\bx)\|^2 \right).
\label{eq:g_ra_F_variance_1_r}
\end{align}

Finally, we have
\begin{align*}
& \E_{\hat{E}, \{\Xi_{1}^{(j)}\}_{j=1}^m, \{\Xi_{2}^{(j)}\}_{j=1}^m}
\left[
\left\|
\bg^{\rm ra}_m(\bx,\hat{E}, \{\Xi_{1}^{(j)}\}_{j=1}^m, \{\Xi_{2}^{(j)}\}_{j=1}^m)
- \nabla F(\bx)
\right\|^2
\right] \\
& \overset{(a)}{\le}
\frac{3\sigma^2d}{\mu^2 N m}
+ 3d \mu^2 M^2 \left(\sqrt{d-1}+\frac12\right)^2
+ \frac{6d}{N}\left(\frac{M^2\mu^2}{4}\,d^2 + 3 \|\nabla F(\bx)\|^2\right),
\end{align*}
where (a) follows from \eqref{eq:split_I_III_ra_first},
\eqref{eq:var_partial_ra_r},
\eqref{eq:b_final_ra}, and
\eqref{eq:g_ra_F_variance_1_r}.
\end{proof}

\begin{lemma}
\label{lem:g_variance_Ra_batch1_H_smooth}
Suppose that Assumptions \ref{asm:F_sample_var_bound} and \ref{asp:F_hessian_lipschitz} hold.
Then, the following holds:
\begin{align}
&\E_{\hat{E},\{\Xi_{1}^{(j)}\}_{j=1}^m, \{\Xi_{2}^{(j)}\}_{j=1}^m}
\left[
\left\|
\bg^{\rm ra}_m(\bx,\hat{E}, \{\Xi_{1}^{(j)}\}_{j=1}^m, \{\Xi_{2}^{(j)}\}_{j=1}^m)
- \nabla F(\bx)
\right\|^2
\right] \nonumber
\\
&\le \frac{3\sigma^2d}{\mu^2 N m}
\;+\;
\frac{dH^2\mu^4}{12}(3d-2)^2
\;+\;
\frac{d^4 H^2\mu^4}{6N} + \frac{18d}{N}\|\nabla F(\bx)\|^2.
\label{eq:ra_var_H_main}
\end{align}
\end{lemma}

\begin{proof}
As in the proof of Lemma \ref{lem:g_variance_Ra_batch1}, 
\begin{align}
& \E_{\hat{E}, \{\Xi_{1}^{(j)}\}_{j=1}^m, \{\Xi_{2}^{(j)}\}_{j=1}^m}
\left[
\left\|
\bg^{\rm ra}_m(\bx,\hat{E}, \{\Xi_{1}^{(j)}\}_{j=1}^m, \{\Xi_{2}^{(j)}\}_{j=1}^m)
- \nabla F(\bx)
\right\|^2
\right] \nonumber\\
&=\E_{\hat{E}, \{\Xi_{1}^{(j)}\}_{j=1}^m, \{\Xi_{2}^{(j)}\}_{j=1}^m}
\big[
\|
\bg^{\rm ra}_m(\bx,\hat{E}, \{\Xi_{1}^{(j)}\}_{j=1}^m, \{\Xi_{2}^{(j)}\}_{j=1}^m)
-\bg_F^{\rm ra}(\bx,\hat{E}) \nonumber\\
&\hspace{40mm} + \E_{\hat{E}}[\bg_F^{\rm ra}(\bx,\hat{E})]  - \nabla F(\bx)
+\bg_F^{\rm ra}(\bx,\hat{E})  - \E_{\hat{E}}[\bg_F^{\rm ra}(\bx,\hat{E})]\|^2
\big] \nonumber \\
&\le
3 \E_{\hat{E}, \{\Xi_{1}^{(j)}\}_{j=1}^m, \{\Xi_{2}^{(j)}\}_{j=1}^m}
\big[
\underbrace{
\|
\bg^{\rm ra}_m(\bx,\hat{E}, \{\Xi_{1}^{(j)}\}_{j=1}^m, \{\Xi_{2}^{(j)}\}_{j=1}^m)
-\bg_F^{\rm ra}(\bx,\hat{E})
\|^2
}_{\text{(I)}}\big] \nonumber \\
&\quad + 3 \underbrace{
\|\E_{\hat{E}}[\bg_F^{\rm ra}(\bx,\hat{E})] - \nabla F(\bx)\|^2
}_{\text{(II)}}
+ 3 \E_{\hat{E}}\big[ \underbrace{
\|\bg_F^{\rm ra}(\bx,\hat{E}) - \E_{\hat{E}}[\bg_F^{\rm ra}(\bx,\hat{E})]\|^2
}_{\text{(III)}}
\big]. \label{eq:split_I_III_ra}
\end{align}
Therefore, we derive the upper bound for (I)--(III).
Let $\hat{\bme}\in\{-1,+1\}^d$ be a random vector with i.i.d.\ entries
$\hat e_j\sim{\rm Rad}\!\left(\tfrac12\right)$.

\textbf{(I):}
This part is identical to Lemma \ref{lem:g_variance_Ra_batch1}, hence
\begin{align}
\E\Big[\|
\bg^{\rm ra}_m(\bx,\hat{E}, \{\Xi_{1}^{(j)}\}_{j=1}^m, \{\Xi_{2}^{(j)}\}_{j=1}^m)
-\bg_F^{\rm ra}(\bx,\hat{E})
\|^2\Big]
\le \frac{\sigma^2 d}{\mu^2 N m}.
\label{eq:ra_I_same}
\end{align}

\textbf{(II):}
Fix $i\in[d]$ and let $\hat{\bme}_{-i}$ be $\hat{\bme}$ with the $i$-th entry replaced by $0$.
Define $A_i(\bx):=\E_{\hat{\bme}_{-i}}[\partial_i F(\bx+\mu \hat{\bme}_{-i})]$. Then
\begin{align}
\left|(\E_{\hat{E}}[\bg_F^{\rm ra}(\bx,\hat{E})])_i-\partial_iF(\bx)\right|
\le
\underbrace{\left|(\E_{\hat{E}}[\bg_F^{\rm ra}(\bx,\hat{E})])_i-A_i(\bx)\right|}_{\text{(A)}}
+
\underbrace{\left|A_i(\bx)-\partial_iF(\bx)\right|}_{\text{(B)}}. \label{eq:II_bound_dec}
\end{align}

First, we derive the upper bound for (A) in \eqref{eq:II_bound_dec}.
Let $\by:=\bx + \mu \hat{\bme}_{-i}$.
As in the derivation of \eqref{eq:hesse_lipschitz_continuous_first} in the proof of Lemma \ref{lem:g_variance_FD_H_smooth},
Assumption \ref{asp:F_hessian_lipschitz} implies that
\begin{align}
    \left|\frac{F(\by+\mu \bme_i)-F(\by-\mu \bme_i)}{2\mu}  - \nabla F(\by)^\top \bme_i\right| \le \frac{H\mu^2}{6} \label{eq:hess_smooth_standard}
\end{align}
for all $i\in[d]$ and $\bx \in \R^d$. 
Therefore,
\begin{align}
&\left|(\E_{\hat{E}}[\bg_F^{\rm ra}(\bx,\hat{E})])_i-A_i(\bx)\right| \nonumber\\
& = \left| \E_{\hat{\bme}_{-i}}\!\left[
\frac{F(\bx+\mu \hat{\bme}_{-i}+\mu \bme_i)-F(\bx+\mu \hat{\bme}_{-i}-\mu \bme_i)}{2\mu}
\right]
- \E_{\hat{\bme}_{-i}}[\partial_i F(\bx+\mu \hat{\bme}_{-i})] \right| \nonumber
\\
& =  \left|\E_{\hat{\bme}_{-i}}\!\left[
\frac{F(\bx+\mu \hat{\bme}_{-i}+\mu \bme_i)-F(\bx+\mu \hat{\bme}_{-i}-\mu \bme_i)}{2\mu}
- \partial_i F(\bx+\mu \hat{\bme}_{-i})\right] \right| \nonumber\\
& \le \E_{\hat{\bme}_{-i}}\!\left[\left|
\frac{F(\bx+\mu \hat{\bme}_{-i}+\mu \bme_i)-F(\bx+\mu \hat{\bme}_{-i}-\mu \bme_i)}{2\mu}
- \partial_i F(\bx+\mu \hat{\bme}_{-i})\right|\right] \nonumber \\
& \overset{(a)}{\le} \frac{H\mu^2}{6}.
\label{eq:cdiff_H}
\end{align}
where (a) comes from \eqref{eq:hess_smooth_standard} with $\by:=\bx+\mu \hat{\bme}_{-i}$.

Next, we derive the upper bound for (B) in \eqref{eq:II_bound_dec}.
From Lemma \ref{lem:Hessian_Lipschitz_tr} with $\by:=\bx + \mu\hat{\bme}_{-i}$, we have
\begin{align}
|\partial_iF(\bx+\mu\hat{\bme}_{-i})-\partial_iF(\bx) - (\nabla^2F(\bx)\mu\hat{\bme}_{-i})_i| 
& \overset{(a)}{\le} \|\nabla F(\bx+\mu\hat{\bme}_{-i})-\nabla  F(\bx) - (\nabla^2F(\bx)\mu\hat{\bme}_{-i})\| 
 \nonumber\\
& \le \frac{H}{2}\|\mu\hat{\bme}_{-i}\|^2 \nonumber\\
&= \frac{H}{2}\mu^2(d-1). \label{eq:partial_F_Hessian_bound}
\end{align}
where (a) is due to the fact that $|v_i| \le \|\bv\|$ for any $\bv \in \R^d$.
Therefore, 
\begin{align}
\left|A_i(\bx)-\partial_iF(\bx)\right|
&=\left|\E_{\hat{\bme}_{-i}}[\partial_i F(\bx+\mu \hat{\bme}_{-i})]-\partial_iF(\bx)\right| \nonumber\\
&\overset{(a)}{=}\left|\E_{\hat{\bme}_{-i}}[\partial_i F(\bx+\mu \hat{\bme}_{-i})-\partial_iF(\bx) - (\nabla^2F(\bx)\mu\hat{\bme}_{-i})_i]\right| \nonumber\\
&\overset{(b)}{\le} \E_{\hat{\bme}_{-i}}\left[\left|\partial_i F(\bx+\mu \hat{\bme}_{-i})-\partial_iF(\bx) - (\nabla^2F(\bx)\mu\hat{\bme}_{-i})_i\right|\right] \nonumber\\
&\overset{(c)}{\le} \frac{H}{2}\mu^2(d-1), \label{eq:ra_bias_part2}
\end{align}
where (a) is due to the fact that $\E[(\nabla^2F(\bx)\mu\hat{\bme}_{-i})_i]=0$, (b) comes from Jensen's inequality and (c) is due to \eqref{eq:partial_F_Hessian_bound}.

Combining \eqref{eq:II_bound_dec}, \eqref{eq:cdiff_H}, and \eqref{eq:ra_bias_part2},
\[
\left|(\E_{\hat{E}}[\bg_F^{\rm ra}(\bx,\hat{E})])_i-\partial_iF(\bx)\right|
\le H\mu^2\left(\frac{d-1}{2}+\frac{1}{6}\right)
= H\mu^2\left(\frac{3d-2}{6}\right).
\]
Therefore,
\begin{align}
\|\E_{\hat{E}}[\bg_F^{\rm ra}(\bx,\hat{E})]-\nabla F(\bx)\|^2
\le d\left(H\mu^2\left(\frac{3d-2}{6}\right)\right)^2
= \frac{dH^2\mu^4}{36}(3d-2)^2.
\label{eq:ra_bias_final_H}
\end{align}

\textbf{(III):}
As in the derivation of \eqref{eq:co_var_progress_ra} in the proof of Lemma \ref{lem:g_variance_Ra_batch1},
\small
\begin{align}
&{\rm Var}(\bg_F^{{\rm ra}}(\bx, \hat{E}))  \nonumber \\
& \preceq
\frac{2}{N}\E_{\hat{\bme}\sim {\rm Unif}(\{-1,+1\}^d)}
\bigg[
\bigg(
\frac{F(\bx+\mu \hat{\bme})-F(\bx) - \mu \nabla F(\bx)^\top \hat{\bme}}{2\mu}
-\frac{F(\bx-\mu \hat{\bme})-F(\bx) + \mu \nabla F(\bx)^\top \hat{\bme}}{2\mu}
\bigg)^2
\hat{\bme}\hat{\bme}^\top \nonumber\\
&\hspace{45mm}
+ (\nabla F(\bx)^\top \hat{\bme})^2 \hat{\bme}\hat{\bme}^\top
\bigg] \nonumber\\
& =
\frac{2}{N}\E_{\hat{\bme}\sim {\rm Unif}(\{-1,+1\}^d)}
\bigg[
\bigg(
\frac{F(\bx+\mu \hat{\bme})-F(\bx) - \mu \nabla F(\bx)^\top \hat{\bme}-\frac{\mu^2}{2}\hat{\bme}^\top \nabla^2 F(\bx) \hat{\bme}}{2\mu}  \nonumber\\
& \hspace{40mm} -\frac{F(\bx-\mu \hat{\bme})-F(\bx) + \mu \nabla F(\bx)^\top \hat{\bme}-\frac{\mu^2}{2}\hat{\bme}^\top \nabla^2 F(\bx) \hat{\bme}}{2\mu}
\bigg)^2
\hat{\bme}\hat{\bme}^\top \nonumber\\
&\hspace{45mm}
+ (\nabla F(\bx)^\top \hat{\bme})^2 \hat{\bme}\hat{\bme}^\top
\bigg] \nonumber\\
&\overset{(a)}{\preceq}
\frac{2}{N}\E_{\hat{\bme}\sim {\rm Unif}(\{-1,+1\}^d)}
\left[
\left(
\frac{H\|\mu\hat{\bme}\|^3}{12\mu} + \frac{H\|\mu\hat{\bme}\|^3}{12\mu}
\right)^2
\hat{\bme}\hat{\bme}^\top
+ (\nabla F(\bx)^\top \hat{\bme})^2 \hat{\bme}\hat{\bme}^\top
\right] \nonumber \\
&=
\frac{2}{N}\E_{\hat{\bme}\sim {\rm Unif}(\{-1,+1\}^d)}
\left[
\frac{H^2\mu^4 \|\hat{\bme}\|^6}{36}\,\hat{\bme}\hat{\bme}^\top
+ (\nabla F(\bx)^\top \hat{\bme})^2 \hat{\bme}\hat{\bme}^\top
\right] \nonumber \\
&\overset{(b)}{=}
\frac{2}{N}\left(
\frac{H^2\mu^4}{36}\,d^3 I
+ \nabla F(\bx)^\top \nabla F(\bx)\, I
+ 2 \nabla F(\bx)\nabla F(\bx)^\top
- 2\,{\rm Diag}(\nabla F(\bx)_1^2,\dots,\nabla F(\bx)_d^2)
\right) \nonumber \\
&\preceq
\frac{2}{N}\left(
\frac{H^2\mu^4}{36}\,d^3 I
+ \nabla F(\bx)^\top \nabla F(\bx)\, I
+ 2 \nabla F(\bx)\nabla F(\bx)^\top
\right) \nonumber \\
&\overset{(c)}{\preceq}
\frac{2}{N}\left(
\frac{H^2\mu^4}{36}\,d^3 + 3\|\nabla F(\bx)\|^2
\right)I,  \nonumber
\end{align}
\normalsize
where (a) is due to Lemma \ref{lem:Hessian_Lipschitz_tr}, which follows from 
Assumption \ref{asp:F_hessian_lipschitz};
(b) uses $\|\hat{\bme}\|^2=d$ and \eqref{eq_rad_st} in the proof of Lemma \ref{lem:g_variance_Ra_batch1};
(c) follows from $\nabla F(\bx)\nabla F(\bx)^\top \preceq \|\nabla F(\bx)\|^2 I$
(as confirmed in the proof of Lemma \ref{lem:g_variance_sp_batch1}).
Then,
\begin{align}
&\E_{\hat{E}}\!\left[\|\bg_F^{{\rm ra}}(\bx,\hat{E}) - \E_{\hat{E}}[\bg_F^{{\rm ra}}(\bx,\hat{E})] \|^2\right] =
{\rm tr}\!\left({\rm Var}(\bg_F^{{\rm ra}}(\bx, \hat{E}))\right) \le \frac{2d}{N} \left(
\frac{H^2\mu^4}{36}\,d^3 + 3\|\nabla F(\bx)\|^2
\right).
\label{eq:g_ra_F_variance_1_r_ms}
\end{align}

Substituting \eqref{eq:ra_I_same}, \eqref{eq:ra_bias_final_H}, and \eqref{eq:g_ra_F_variance_1_r_ms} into \eqref{eq:split_I_III_ra}
yields \eqref{eq:ra_var_H_main}.
\end{proof}

Then, we show the sample complexity for Algorithm \ref{alg:simple} with $\bg^{\rm{ra}}_m$.

\begin{theorem} \label{thm:ra_sample_complexity}
Suppose that $\epsilon\le \frac{1}{3}$ and Assumptions \ref{asm:F_sample_var_bound} and \ref{asm:F_smooth} hold.
Let $\bg_t:=\bg^{\rm ra}_m(\bx_t, \hat{E},\{\Xi_1^{(j)}\}_{j=1}^m, \{\Xi_2^{(j)}\}_{j=1}^m)$, where $\bg^{\rm ra}_m$ is defined by
\eqref{eq:mini_batch_gradient_est_ra}.
Let $\mu:=\Theta\!\left(\frac{\epsilon}{d}\right)$, $N:=d^3 \epsilon^{-4}$, $m:=1$, $\eta\le \frac{1}{4M}$, and $T:=\Theta(\epsilon^{-2})$.
Then, the sample complexity, to satisfy $\E[\|\nabla F(\bar{\bx}_T)\|^2]\le \epsilon^2$ for output $\bar{\bx}_T$ of Algorithm \ref{alg:simple}, is $O(NmT)=O(d^3\epsilon^{-6})$.
\end{theorem}

\begin{proof}
We have
\small
\begin{align*}
&\E[\|\nabla F(\bar{\bx}_T)\|^2] \\
& \overset{(a)}{\le} \frac{4(F(\bx_0) - F^*)}{\eta (T+1)}  + \frac{3}{T+1} \sum_{t=0}^T\E[\|\nabla F(\bx_t)-\bg_t\|^2] \\
& \overset{(b)}{\le}
\frac{4(F(\bx_0) - F^*)}{\eta (T+1)}
+ \frac{9\sigma^2d}{\mu^2 N m}
+ 9d \mu^2 M^2 \left(\sqrt{d-1}+\frac12\right)^2
+ \frac{9M^2\mu^2\,d^3}{2N}
+ \frac{54d}{N}\frac{1}{T+1} \sum_{t=0}^T \E [\|\nabla F(\bx_t)\|^2] \\
&\overset{(c)}{=} O(\epsilon^2) \;+\; 54 d^{-2} \epsilon^4 \frac{1}{T+1}\sum_{t=0}^T \E[\|\nabla F(\bx_t)\|^2] \\
&\overset{(d)}{\le} O(\epsilon^2) \;+\; \frac{2}{3}\,\frac{1}{T+1}\sum_{t=0}^T \E[\|\nabla F(\bx_t)\|^2],
\end{align*}
\normalsize
where (a) comes from Lemma \ref{lem:simple_descent};
(b) follows from Lemma \ref{lem:g_variance_Ra_batch1};
(c) holds since $T=\Theta(\epsilon^{-2})$, $\mu=\Theta(\epsilon/d)$, $N=d^3\epsilon^{-4}$, and $m=1$;
(d) is due to $\epsilon\le \frac{1}{3}$ and $d\ge 1$.

Since $\E[\|\nabla F(\bar{\bx}_T)\|^2]=\frac{1}{T+1} \sum_{t=0}^T\E [\|\nabla F(\bx_t)\|^2]$, we obtain
\[
\left(1-\frac{2}{3}\right)\E[\|\nabla F(\bar{\bx}_T)\|^2] = O(\epsilon^2),
\]
hence $\E[\|\nabla F(\bar{\bx}_T)\|^2]=O(\epsilon^2)$.
Therefore, the sample complexity is $O(NmT)=O(d^3\epsilon^{-6})$.
\end{proof}

\begin{theorem} \label{thm:sample_complexity_ra_Hessian_lip}
Suppose that $\epsilon\le \frac{1}{4}$ and Assumptions \ref{asm:F_sample_var_bound}--\ref{asp:F_hessian_lipschitz} hold.
Let $\bg_t:=\bg^{\rm ra}_m(\bx_t, \hat{E},\{\Xi_1^{(j)}\}_{j=1}^m, \{\Xi_2^{(j)}\}_{j=1}^m)$, where $\bg^{\rm ra}_m$ is defined by
\eqref{eq:mini_batch_gradient_est_ra}.
Let $\mu:=\Theta\!\left(\epsilon^{\frac12}d^{-\frac34}\right)$, $N:=d^{\frac52}\epsilon^{-3}$, $m:=1$, $\eta\le \frac{1}{4M}$, and $T:=\Theta(\epsilon^{-2})$.
Then, the sample complexity, to satisfy $\E[\|\nabla F(\bar{\bx}_T)\|^2]\le \epsilon^2$ for output $\bar{\bx}_T$ of Algorithm \ref{alg:simple}, is $O(NmT)=O(d^{\frac52}\epsilon^{-5})$.
\end{theorem}

\begin{proof}
We have
\begin{align*}
&\E[\|\nabla F(\bar{\bx}_T)\|^2] \\
& \overset{(a)}{\le} \frac{4(F(\bx_0) - F^*)}{\eta (T+1)}  + \frac{3}{T+1} \sum_{t=0}^T\E[\|\nabla F(\bx_t)-\bg_t\|^2] \\
& \overset{(b)}{\le}
\frac{4(F(\bx_0) - F^*)}{\eta (T+1)}
+ \frac{9\sigma^2d}{\mu^2 N m}
+ \frac{dH^2\mu^4}{4}(3d-2)^2
+ \frac{d^4 H^2\mu^4}{2N}
+ \frac{54d}{N}\frac{1}{T+1} \sum_{t=0}^T \E [\|\nabla F(\bx_t)\|^2] \\
&\overset{(c)}{=} O(\epsilon^2) \;+\; 54\,d^{-\frac{3}{2}}\epsilon^3 \,\frac{1}{T+1}\sum_{t=0}^T \E[\|\nabla F(\bx_t)\|^2] \\
&\overset{(d)}{\le} O(\epsilon^2) \;+\; \frac{27}{32}\,\frac{1}{T+1}\sum_{t=0}^T \E[\|\nabla F(\bx_t)\|^2],
\end{align*}
where (a) comes from Lemma \ref{lem:simple_descent};
(b) follows from Lemma \ref{lem:g_variance_Ra_batch1_H_smooth};
(c) holds since $T=\Theta(\epsilon^{-2})$, $\mu=\Theta(\epsilon^{1/2}d^{-3/4})$, $N=d^{5/2}\epsilon^{-3}$, and $m=1$;
(d) is due to $\epsilon\le \frac{1}{4}$ and $d\ge 1$.

Since $\E[\|\nabla F(\bar{\bx}_T)\|^2]=\frac{1}{T+1} \sum_{t=0}^T\E [\|\nabla F(\bx_t)\|^2]$, we obtain
\[
\left(1-\frac{27}{32}\right)\E[\|\nabla F(\bar{\bx}_T)\|^2] = O(\epsilon^2),
\]
hence $\E[\|\nabla F(\bar{\bx}_T)\|^2]=O(\epsilon^2)$.
Therefore, the sample complexity is $O(NmT)=O(d^{\frac52}\epsilon^{-5})$.
\end{proof}

\section{Technical lemmas}
We provide technical lemmas that are needed to prove the lemmas and theorems in our paper. 

\begin{techlemma}
\cite[Proposition 3.4(i)]{khanh2024globally}
\label{lem:g_bound_FD}
Suppose that Assumption \ref{asm:F_smooth} holds.
For any $\bx \in \R^d$, the following holds.
\begin{align*}
&\left\| \sum_{i=1}^{d}\frac{F(\bx+\mu \bme_i)-F(\bx-\mu \bme_i)}{2\mu}\bme_i  - \nabla F(\bx) \right\| \le \frac{M\sqrt{d}\mu}{2}.
\end{align*}
\end{techlemma}

\begin{techlemma} \label{lem:sp_smoothed_error}
\cite[(2.35) and (2.36) with $e_f:=0$]{berahas2022theoretical}
If Assumption \ref{asm:F_smooth} holds, we have
$$\|\nabla F_{\mu,\maB}(\bx) - \nabla F(\bx)\| \le \mu M.$$
If Assumption \ref{asp:F_hessian_lipschitz} holds, we have
$$\|\nabla F_{\mu,\maB}(\bx) - \nabla F(\bx)\| \le \mu^2 H.$$
\end{techlemma}

\begin{techlemma}
\label{lem:ga_smoothed_error}
\cite[(2.10) and (2.11) with $e_f:=0$]{berahas2022theoretical}
If Assumption \ref{asm:F_smooth} holds, we have
$$\|\nabla F_{\mu,\maN}(\bx) - \nabla F(\bx)\| \le \sqrt{d} \mu M.$$
If Assumption \ref{asp:F_hessian_lipschitz} holds, we have
$$\|\nabla F_{\mu,\maN}(\bx) - \nabla F(\bx)\| \le d \mu^2 H.$$
\end{techlemma}

\begin{techlemma} \label{lem:moment_gauss_noise}
\citep[Lemma 1]{nesterov2017random}
\begin{align*}
\E_{\bu \sim \mathcal N(0, \bI)} [ \| \bu \|^2] = d.
\end{align*}
\end{techlemma}

\begin{techlemma} \label{lem:k_moment_ss_sp_noise}
\citep[(2.39)]{berahas2022theoretical}
For $\ba \in \R^d$,
\begin{align*}
\E_{\bs \sim \maS} [ (\ba^\top \bs)^2 \bs \bs^\top] = \frac{\ba^\top \ba I + 2\ba \ba^\top}{d(d+2)},
\end{align*}
where $\maS$ denotes the uniform distribution on the unit sphere.
Moreover, for $k =0,1,2,\dots,$
\begin{align*}
\E_{\bs \sim \maS} [ \| \bs \|^k \bs \bs^\top] = \frac{1}{d}I.
\end{align*}
\end{techlemma}

\begin{techlemma} \label{lem:k_moment_uu_gauss_noise}
\citep[(2.18)]{berahas2022theoretical}
For $\ba \in \R^d$,
\begin{align*}
\E_{\bu \sim \mathcal N(0, I)} [ (\ba^\top \bu)^2 \bu \bu^\top] = \ba^\top \ba I + 2\ba \ba^\top.
\end{align*}
For $k=0,2,4,\dots,$
\begin{align*}
\E_{\bu \sim \mathcal N(0, I)} [ \| \bu \|^k \bu \bu^\top] = (d+2)(d+4)\dots(d+k)I.
\end{align*}
\end{techlemma}

\begin{techlemma} \label{lem:minibatch_var_reduction}
\cite[p183]{freund1986mathematical}
For any $\bx \in \R^d$ and distribution $D$, 
$$\E_{\{\bmxi^j\}_{j=1}^m \sim D}\left[\left( \frac{1}{m}\sum_{j=1}^{m} f(\bx, \bmxi^j) - \E_{\bmxi \sim D}[f(\bx, \bmxi)] \right)^2 \right] \le \frac{1}{m} \E_{\bmxi' \sim D}\left[\left( f(\bx, \bmxi') - \E_{\bmxi \sim D}[f(\bx, \bmxi)] \right)^2 \right]. $$
\end{techlemma}

\begin{techlemma} \label{lem:1_to_n_two_norm_bound}
For any $s_1,s_2,\dots, s_m \in \R$,
\begin{align*}
    \left(\sum_{i\in[m]} s_i \right)^2 \le m \sum_{i\in[m]} s_i^2.
\end{align*}
\end{techlemma}
\begin{proof}
\begin{align*}
\left(\sum_{i\in[m]} s_i \right)^2 
= \left((s_1, s_2, \dots , s_m) \begin{pmatrix}1 \\ 1 \\ \vdots \\ 1 \end{pmatrix}\right)^2
\overset{(a)}{\le} \|(s_1, s_2, \dots , s_m)\|^2 \|(1,1,\dots,1) \|^2 = 
m \sum_{i\in[m]} s_i^2,
\end{align*}
where (a) comes from the Cauchy–Schwarz inequality.
\end{proof}

\begin{techlemma}
\cite[Lemma 1.2.3]{nesterov2018lectures}
\label{lem:first_order_taylor_remainder}
Suppose that Assumption~\ref{asm:F_smooth} holds. Then, for any
$\bx,\by\in\R^d$,
\[
\left|
F(\by)-F(\bx)
-\nabla F(\bx)^\top(\by-\bx)
\right|
\leq
\frac{M}{2}\|\by-\bx\|^2.
\]
\end{techlemma}

\begin{techlemma}
\cite[Lemma 1]{nesterov2006cubic}
\label{lem:Hessian_Lipschitz_tr}
Under Assumption~\ref{asp:F_hessian_lipschitz}, for any
$\bx,\by\in\R^d$, we have
\begin{align*}
&\left\|
\nabla F(\by)-\nabla F(\bx)
-\nabla^2 F(\bx)(\by-\bx)
\right\|
\le
\frac{H}{2}\|\by-\bx\|^2,
\\
&\left|
F(\by)-F(\bx)
-\nabla F(\bx)^\top(\by-\bx)
-\frac{1}{2}(\by-\bx)^\top
\nabla^2 F(\bx)(\by-\bx)
\right|
\le
\frac{H}{6}\|\by-\bx\|^3.
\end{align*}
\end{techlemma}

\end{document}